\documentclass[10pt,a4paper]{article}
\usepackage{epsfig}
\usepackage{amsmath}
\usepackage{amssymb}
\usepackage{amsfonts}
\usepackage{amsthm}
\usepackage{mathrsfs}
\usepackage{hyperref}
\usepackage{algorithm}
\usepackage{algorithmic}
\usepackage{blindtext}
\usepackage[utf8]{inputenc}
\usepackage{graphicx}
\usepackage{subfigure}
\usepackage[usenames, dvipsnames]{color}
\usepackage{caption}
\usepackage{booktabs}
\usepackage{varwidth}

\newcommand {\C}        {{\mathbb{C}}}

\newcommand {\ri}	{{\mathrm i}}

 \newtheorem{theorem}{Theorem}[section]
 \newtheorem{lemma}[theorem]{Lemma}

 \newtheorem{definition}[theorem]{Definition}
\newtheorem{example}[theorem]{Example}
 \numberwithin{equation}{section}

\DeclareMathOperator*{\argmin}{arg\,min}
\DeclareMathOperator*{\argmax}{arg\,max}

 \newenvironment{remark}[1][Remark]{\begin{trivlist}
 \item[\hskip \labelsep {\bfseries #1}]}{\end{trivlist}}

\title{Computation of Stability Radii for Large-Scale Dissipative Hamiltonian Systems}
\author{
	Nicat~Aliyev\footnotemark[1] \and 
	Volker~Mehrmann\footnotemark[2] \and 
	Emre~Mengi\footnotemark[3]
 }

\begin{document}
\maketitle

\renewcommand{\thefootnote}{\fnsymbol{footnote}}
\footnotetext[1]{Istanbul Sabahattin Zaim University, Department of Industrial Engineering,
  Halkal{\i} mahallesi, Halkal{\i} Caddesi  34303, K\"{u}\c{c}\"{u}k\c{c}ekmece, Istanbul, Turkey. (naliyev@ku.edu.tr).}
\footnotetext[2]{Technische Universit\"{a}t Berlin, Sekreteriat MA 4-5. Strasse des 17 Juni 136, D-10623 Berlin, Germany (mehrmann@math.tu-berlin.de). Supported by Deutsche Forschungsgemeinschaft via Project A02 of Sonderforschungsbereich 910.}
\footnotetext[3]{Ko\c{c} University, Department of Mathematics, Rumelifeneri Yolu, 34450 Sar\i yer, Istanbul, Turkey (emengi@ku.edu.tr). Supported by Deutsche Forschungsgemeinschaft via Project A02 of Sonderforschungsbereich 910.}

\begin{abstract}
\noindent
A linear time-invariant dissipative Hamiltonian (DH) system $\dot x = (J-R)Q x$,
with a skew-Hermitian $J$, an Hermitian positive semi-definite $R$, and an Hermitian positive definite $Q$,
is always Lyapunov stable and under weak further conditions even asymptotically stable.
In various applications there is uncertainty
on the system matrices $J, R, Q$, and it is desirable to know whether the system remains
asymptotically stable uniformly against all possible uncertainties within a given perturbation set. Such robust stability
considerations motivate the concept of stability radius for  DH systems, i.e., what is the maximal
perturbation permissible to the coefficients $J, R, Q$, while preserving the asymptotic stability.
We consider two  stability radii, the unstructured one where $J, R, Q$
are subject to unstructured perturbation, and the structured one where the perturbations preserve the DH structure.
We employ characterizations for these radii that have been
derived recently in [\emph{SIAM J. Matrix Anal. Appl., 37, pp. 1625-1654, 2016}] and propose new algorithms to compute
these stability radii for large scale problems by tailoring  subspace frameworks that
are interpolatory and guaranteed to converge at a super-linear rate in theory.
At every iteration, they first solve a reduced problem and then expand
the subspaces in order to attain certain Hermite interpolation properties between the full and
reduced problems. The reduced problems are solved by means of the adaptations
of existing level-set algorithms for ${\mathcal H}_\infty$-norm computation in the unstructured case,
while, for the structured radii, we benefit from algorithms that approximate the
objective eigenvalue function with a piece-wise quadratic global underestimator.
The performance of the new approaches is illustrated with several examples including a system
that arises from a finite-element modeling of an industrial disk brake.
\\[10pt]
\noindent
\textbf{Key words.} Linear Time-Invariant Dissipative Hamiltonian System, Port-Hamiltonian system, Robust Stability, Stability Radius, Eigenvalue Optimization, Subspace Projection, Structure Preserving Subspace Framework,
Hermite Interpolation.
\\[5pt]
\noindent
\textbf{AMS subject classifications.} 65F15, 93D09, 93A15, 90C26
\end{abstract}

\section{Introduction}
Linear time-invariant \emph{Dissipative Hamiltonian (DH) systems} are dynamical systems of the form
\begin{equation}\label{DH}
		\dot x	 \;\; = \;\;	(J-R)Qx.
\end{equation}
They arise as homogeneous part of \emph{port-Hamiltonian (PH) systems} of the form
\begin{equation}\label{ph}
	\begin{split}
		\dot x	& \;\; = \;\;	(J-R)Qx(t) + (B-P)u(t), \\
		y(t)	& \;\; = \;\; 	(B+P)^H Qx(t) +  D u(t),
	\end{split}
\end{equation}
when the input $u$ is $0$ and the output  $y$ is not considered.
Here  $Q=Q^H\in\mathbb{C}^{n \times n}$ is an Hermitian positive definite matrix (denoted as $Q>0$), $J\in\mathbb{C}^{n \times n}$ is a skew-Hermitian matrix associated with the energy flux of the system, $R\in\mathbb{C}^{n \times n}$ is the Hermitian positive semi-definite (denoted by $R\geq 0$) \emph{dissipation matrix} of the system, $B\pm P\in\mathbb{C}^{n \times m}$ are the \emph{port matrices}, and $D$ describes the \emph{direct feed-through} from input to output.
The function $\mathcal H(x)= \frac 12 x^HQx$ (called \emph{Hamiltonian function}) describes the total internal energy of the system. Here and elsewhere $A^H$ denotes the conjugate transpose of a complex matrix $A$.

PH and DH  systems play an essential role in most areas of science and engineering, see e.g. \cite{JacZ12,SchJ14}, due to their very important structural properties; e.g., they allow modularized  modeling and easy model reduction via Galerkin projection. An important structural  property is that DH systems are automatically \emph{Lyapunov stable}, i.e., all eigenvalues of $A=(J-R)Q$ are in the closed left half of the complex plane, and those on the imaginary axis are semisimple, see \cite{MehMS16a}.
%
%
However, DH systems are not necessarily \emph{asymptotically stable}, since $A$ may have purely imaginary eigenvalues, e.g., when the dissipation matrix $R$
vanishes, then all eigenvalues are purely imaginary. If a DH system is Lyapunov stable but not asymptotically stable, then arbitrarily small unstructured perturbations (such as rounding errors) may cause the system to become unstable.
%

These issues are our motivation to analyse whether a DH system is  \textit{robustly asymptotically stable}, i.e.,
whether small (structured or unstructured) perturbations keep it asymptotically stable.

\begin{example}\label{ex:ex1} {\rm Disk brake squeal is a well-known problem in  mechanical engineering.
It occurs due to self-excited vibrations caused by instability at the pad-rotor interface \cite{Akay_2002}.
The transition from stability to instability of the brake system is generally examined by  finite element (FE) analysis of the system. In \cite{GraMQSW16} FE models resulting for disk brakes are derived in form of second order differential equations
 \begin{equation}\label{brake1}
M \ddot x+D(\Omega) \dot x+ K(\Omega) x=f,
\end{equation}
with large and sparse coefficient matrixes $M$, $D(\Omega)$, and $K(\Omega) \in \mathbb{R}^{n \times n}$, where
$D(\Omega)$ and $K(\Omega)$ depend on the rotational speed $\Omega>0$
of the disk, and have the form
\[
	D(\Omega)
		:=
	D_M		+	\frac{1}{\Omega} D_R,\
	K(\Omega)
		:=
	K_E		+	\Omega^2 K_g
\]
with $D_M, D_R$ representing material, friction-induced damping matrices,
$K_E, K_g$ corresponding to elastic, geometric stiffness matrices, respectively. Here,
$M>0$ and $K(\Omega)>0$, whereas $D(\Omega)\geq 0$. (The function $f$ represents a forcing term or control, but for the stability analysis one may assume that $f=0$, which we assume in the following.) The incorporation of  gyroscopic effects, modeled by the term $G(\Omega) \dot x$,  with $G(\Omega) := \Omega D_G=-\Omega D_G^H$, and circulatory effects, modeled by an unsymmetric term $N x$ gives rise to a system
\begin{equation}
\label{eq:full_brake}
M \ddot x + (D(\Omega) +G(\Omega)) \dot x +  (K(\Omega) + N)x=0,
\end{equation}
or in first order representation $ \widetilde{M} \dot z  +  \widetilde{K} z=0$,  where
\begin{equation}
\label{eq:first_order}
 \widetilde{M}=\begin{bmatrix}  M & 0 \\ 0 & K(\Omega) \end{bmatrix}, \;\:
\widetilde{K}=\begin{bmatrix}  D(\Omega) + G(\Omega) & K(\Omega) + N \\ -K(\Omega) & 0 \end{bmatrix}.
\end{equation}
Straightforward manipulations yields a system
\begin{equation}
\label{insta}
\dot {\widetilde{z}} = (J - R)Q\widetilde{z}
\end{equation}
with $\widetilde{z} = Q^{-1}z \:$, where
\begin{equation}\label{eq:insta_matrices}
\begin{split}
J =\begin{bmatrix}   -G(\Omega) & -(K(\Omega) + \frac12N) \\  K(\Omega) + \frac12N^H & 0 \end{bmatrix}, \hskip 22ex \\
R =\begin{bmatrix}  D(\Omega) & \frac12N \\ \frac12N^H & 0 \end{bmatrix}, \quad
Q =\begin{bmatrix}  M & 0 \\ 0 & K(\Omega) \end{bmatrix}^{-1}.
\end{split}
\end{equation}
In the absence of the circulatory effects, i.e., when $N=0$, the system in (\ref{insta}) is a DH system
and as a result it is Lyapunov stable and typically even asymptotically stable. However, small circulatory effects, i.e.,  perturbations by 
a non-symmetric $N$ of small norm,
may result in instability.
}
\end{example}

Asymptotic stability of a general linear dynamical system in the presence of uncertainty can only be guaranteed when the system has a
reasonable \textit{distance to instability}, i.e., to
systems with
purely imaginary eigenvalues. Hence, an estimation of the distance to instability, which is an optimization problem
over admissible perturbations, is an important ingredient of a proper stability analysis.

In this paper we focus on the stability analysis of large-scale (and typically sparse) DH systems of the form~\eqref{DH} in the presence of uncertainties in the coefficients. Considering perturbations in one of the coefficient matrices $J$, $R$, $Q$ of \eqref{DH}, in \cite{MehMS16a} characterizations for several structured distances to instability were derived under restricted perturbations of the form $B\Delta C$, with restriction matrices $B \in\mathbb{C}^ {n\times m}$ and $C \in \C^{p\times n}$ of full column rank and full row rank, respectively, allowing selected parts of the matrices $J, R, Q$ to be unperturbed.
We will use an adaptation of the subspace framework introduced in \cite{Aliyev2017}, based on model order reduction techniques to compute the stability radii using the characterizations in \cite{MehMS16a}.

The paper is organized as follows.
Section \ref{defns}  provides formal definitions of the structured and unstructured
stability radii, and in Section \ref{char} we briefly recall the characterizations of these stability radii derived in \cite{MehMS16a}. Section \ref{sec:DH_unstructured_dist} proposes subspace
frameworks for computing the unstructured stability radii problems exploiting these characterizations.
The performance of the proposed frameworks for the unstructured stability radii is illustrated
via the disk brake example and several synthetic examples
in Section \ref{sec:DH_numexps}. Finally, Section \ref{sec:structured}
focuses on the structured stability radius when only
$R$ is subject to Hermitian
perturbations. We first discuss how small-scale problems can be solved in Section~\ref{sec:small_scale_st}. A new  structured subspace framework is discussed in Section~\ref{sec:large_scale_st} followed by several numerical examples in Section \ref{sec:str_num_exp}.

\section{Unstructured and Structured Stability Radii}\label{defns}

In \cite{MehMS16a} computable formulas for DH systems of the form \eqref{DH} are derived using several notions of
unstructured and structured stability radii.
In this section we briefly recall the main definitions and results from \cite{MehMS16a} for restricted perturbations in one of the following forms.
\begin{equation}\label{pert_sys}
\left( \left(J+B\Delta C\right) - R\right)Q , \;
\left(J - \left(R+B\Delta C\right)\right)Q , \; \text{or} \;
 \left(J - R \right)\left(Q+B\Delta C\right).
\end{equation}
In the following ${\rm i} {\mathbb R}$ denotes the imaginary axis in the complex plane, $\Lambda(A)$ the spectrum of a matrix $A$, and $\|A\|_2$
the spectral norm.
\begin{definition}\label{defn:unst_stab_radii}
Consider a DH system of the form \eqref{DH} and suppose that $B \in \mathbb{C}^{n \times m}$ and
$C \in \mathbb{C}^{p \times n}$ are given full rank restriction matrices.
\begin{itemize}
\item[(i)] The \emph{unstructured restricted stability radius} $r(J; B,C)$  with respect to perturbations of $J$ under
the restriction matrices $B$, $C$ is defined by
\[
	r(J; B,C) := \inf \{\|\Delta\|_2 : \Delta\in \mathbb{C}^{m\times p}, \Lambda (\left( \left(J+B\Delta C\right) - R\right)Q)\cap {\rm i}\mathbb{R} \neq \emptyset\}.
\]
\item[(ii)] The \emph{unstructured restricted stability radius} $r(R; B,C)$  with respect to perturbations of $R$ under the restriction matrices $B$, $C$ is defined by
\[
	r(R; B,C) := \inf \{\|\Delta\|_2 : \Delta\in \mathbb{C}^{m\times p}, \Lambda (\left(J - \left(R+B\Delta C\right)\right)Q\cap {\rm i}\mathbb{R} \neq \emptyset\}.
\]
\item[(iii)] The \emph{unstructured restricted stability radius} $r(Q; B,C)$  with respect to perturbations of $Q$ under the restriction matrices $B$, $C$ is defined by
\[
	r(Q; B,C) := \inf \{\|\Delta\|_2 : \Delta\in \mathbb{C}^{m\times p}, \Lambda \left(J - R \right)\left(Q+B\Delta C\right)\cap {\rm i}\mathbb{R} \neq \emptyset\}.
\]
\end{itemize}
\end{definition}
%
%
\begin{example} \label{ex:ex2}{\rm  Consider again Example~\ref{ex:ex1}. Here it is of interest to know whether (for given $\Omega$) the norm of the non-symmetric matrix $N$ is tolerable to preserve the asymptotic stability
of the DH system in (\ref{insta}) without the circulatory effects. The relevant stability radius for a specified $\Omega$ is given by
\begin{equation}\label{eq:BH_stabradii}
	\inf
	\left\{
			\| N \|_2	\;	\bigg|	\;
		\Lambda
		\left(
		{\mathcal A}(N) \right)	\cap	\ri {\mathbb R}	\neq \emptyset
	\right\},
\end{equation}
where
\begin{equation*}
	\begin{split}
	{\mathcal A}(N)
		&	:=
		\left(		
			\begin{bmatrix}   -G(\Omega) & -(K(\Omega) + \frac12N) \\  K(\Omega) + \frac12N^H & 0 \end{bmatrix}
						-		
			\begin{bmatrix}  D(\Omega) & \frac12N \\ \frac12N^H & 0 \end{bmatrix}
		\right) 		
			\begin{bmatrix}  M & 0 \\ 0 & K(\Omega) \end{bmatrix}^{-1}	\\
		&	=
		\left\{		
			\begin{bmatrix}   -G(\Omega) & -K(\Omega)  \\  K(\Omega) & 0 \end{bmatrix}
						-	
			\left(	
			\begin{bmatrix}  D(\Omega) & 0 \\ 0 & 0 \end{bmatrix}
						+
			\begin{bmatrix}  0 & N \\ 0 & 0 \end{bmatrix}
			\right)
		\right\}		
			\begin{bmatrix}  M & 0 \\ 0 & K(\Omega) \end{bmatrix}^{-1}.
	\end{split}
\end{equation*}
Hence, the stability radius in (\ref{eq:BH_stabradii}) corresponds to the unstructured stability radius
$r(R; B, C)$ with the restriction matrices
$
	B =
		\begin{bmatrix}
			I	&	0
		\end{bmatrix}^T
$
and
$
	C =
		\begin{bmatrix}
			0	&	I
		\end{bmatrix}
$
with $n\times n$ blocks.

Furthermore, in the definition of ${\mathcal A}(N)$ the skew-Hermitian perturbations are more
influential on the imaginary parts of its eigenvalues, whereas the Hermitian perturbations
are more effective in moving its eigenvalues towards the imaginary axis. This leads us to
the consideration of the stability radius
\begin{equation}\label{eq:BH_stabradii2}
	\inf
	\left\{
			\| N \|_2	\;	\bigg|	\;
		\Lambda
		\left(
		{\mathcal A}_0(N) \right)	\cap	\ri {\mathbb R}	\neq \emptyset
	\right\}
\end{equation}
with
\[
	{\mathcal A}_0(N)
		:=
	\left(		
			\begin{bmatrix}   -G(\Omega) & -K(\Omega) \\  K(\Omega) & 0 \end{bmatrix}
						-		
			\begin{bmatrix}  D(\Omega) & \frac12N \\ \frac12N^H & 0 \end{bmatrix}
		\right) 		
			\begin{bmatrix}  M & 0 \\ 0 & K(\Omega) \end{bmatrix}^{-1}.
\]
}
\end{example}
Examples such as Example~\ref{ex:ex2}  motivate the following definition of the structured stability radius in \cite{MehMS16a}.
\begin{definition}\label{defn:st_stab_radii}
Consider a DH system of the form \eqref{DH} and suppose that $B \in \mathbb{C}^{n \times m}$ is a given
restriction matrix. The \emph{structured restricted stability radius} with respect to Hermitian
perturbations of $R$ under the restriction $B$ is defined by
\begin{equation}\label{eq:structured_radii_defn}
	\begin{split}
	r^{\rm Herm}(R; B) := \inf\{ \|\Delta \|_2 \; | \; \Delta = \Delta^H, \;\;\;		\hskip 25ex	\\
		\hskip 23ex
		\Lambda \big( (J - R)Q - (B\Delta B^H)Q \big) \cap {\rm i}\mathbb R\neq \emptyset\}.
	\end{split}
\end{equation}
\end{definition}


\section{Characterizations for Stability Radii}\label{char}

The numerical techniques that we will derive for the computation
of the unstructured and structured stability radii exploit
eigenvalue or singular value optimization characterizations derived in \cite{MehMS16a}.

\begin{theorem}\label{thm_unstr}
For an asymptotically stable DH system of the form \eqref{DH} and restriction matrices
$B\in\mathbb{C}^{n \times m}$, $C \in\mathbb C^{p \times n}$ the following assertions hold:
\begin{itemize}
\item[(i)] The unstructured stability radius $r(R; B,C)$ is finite if and only if
$G_R(\omega) := CQ({\rm i}\omega I_n-(J - R)Q)^{-1}B$ is not identically
zero if and only if $r(J; B,C)$ is finite. If $r(R; B,C)$ is finite, then we have
\begin{equation}\label{th2}
r(R; B,C) = r(J;B,C) = \inf_{\omega\in\mathbb R}\frac{1}{\|G_{R}(\omega)\|_2}.
\end{equation}
\item[(ii)] The unstructured stability radius $r(Q; B,C)$ is finite if and only if
$G_Q(\omega) : = C({\rm i}\omega I_n-(J - R)Q)^{-1}(J-R)B$ is not identically zero
for all $\omega\in\mathbb R$. If\, $r(Q; B,C)$ is finite, then we have
\begin{equation}\label{th1}
r(Q; B,C) = \inf_{\omega\in\mathbb R}\frac{1}{\|G_{Q}(\omega)\|_2}.
\end{equation}
\end{itemize}
\end{theorem}
For the structured stability radius and Hermitian perturbations of $R$ the following result is obtained in
\cite{MehMS16a}.
%
\begin{theorem}\label{thm_str_Herm}
For an asymptotically stable DH system of the form \eqref{DH}, and a restriction
matrix $B\in\mathbb{C}^{n\times m}$ of full column rank, let
\begin{enumerate}
	\item $W(\lambda) := (J-R)Q - \lambda I$ for a given $\lambda \in {\mathbb C}$ such
	that $W(\lambda)$ is invertible,
	\item $L(\lambda)$ be a lower triangular Cholesky factor of
			\begin{center}
			$\widetilde{H}_0(\lambda) := B^H W(\lambda)^{-H} Q B B^H Q W(\lambda)^{-1} B$,
			\end{center}
			i.e., $L(\lambda)$ is a lower triangular matrix satisfying $\widetilde{H}_0(\lambda) = L(\lambda) L(\lambda)^H$,
	\item $H_0(\lambda) := L(\lambda)^{-1} L(\lambda)^{-H}$,
	\item $H_1(\lambda) := \ri (\widetilde{H}_1(\lambda)  -  \widetilde{H}_1(\lambda)^H)$, where
	$\widetilde{H}_1(\lambda) := L(\lambda)^{-1} B^H W(\lambda)^{-H} Q B L(\lambda)^{-H}$.
\end{enumerate}
Then $r^{\rm Herm}(R; B)$ is finite, and given by
\begin{equation*}
	r^{\rm Herm}(R; B)
		\; = \;
	\left\{
	\inf_{\omega \in {\mathbb R}} \: \sup_{t \in {\mathbb R}} \: \lambda_{\min} (H_0(\ri \omega) + t H_1 (\ri \omega))
	\right\}^{1/2},
\end{equation*}
where $\lambda_{\min}( \cdot )$ denotes the smallest eigenvalue of its Hermitian matrix argument,
and the inner supremum is attained if and only if $H_1({\rm i} \omega)$ is indefinite.
\end{theorem}

The characterization in \cite{MehMS16a} is presented in a slightly different form. In particular,
it is stated in terms of an orthonormal basis $U(\lambda)$ for the kernel of $\left((I - BB^+) W(\lambda)\right)$.
It turns out that $U(\lambda)$ does not have to be orthonormal, rather the theorem can be stated in terms of any basis for the kernel of $\left((I - BB^+) W(\lambda)\right)$. In Theorem~\ref{thm_str_Herm}, we have employed a particular basis that simplifies the formulas and facilitates the computation.

\section{Computation of the Unstructured Stability Radii for Large-Scale Problems}\label{sec:DH_unstructured_dist}
In this section we study the computation of unstructured stability radii for large-scale DH systems using the characterizations of
$r(R; B,C)$, $r(Q; B,C)$, $r(J;B,C)$ given in Theorem~\ref{thm_unstr}.
One easily observes that
\begin{equation*}
	\begin{split}
	G_R({\rm i} \omega) & \; := \; CQ({\rm i}\omega I_n-(J - R)Q)^{-1}B,  \\	
	G_Q({\rm i} \omega) & \; := \; C({\rm i}\omega I_n-(J - R)Q)^{-1}(J-R)B
	\end{split}
\end{equation*}
can be viewed as restrictions of transfer functions of control systems to the imaginary axis. To be precise,
setting  $\widetilde{A} =  (J - R ) Q$, $\widetilde{B} = B$  and $\widetilde{C} = C Q$,
the matrix-valued function $G_R({\rm i} \omega) := CQ({\rm i}\omega I_n-(J - R)Q)^{-1}B \;$ becomes
\begin{equation}\label{def_GR}
\widetilde{G}_R({\rm i} \omega) := \widetilde{C}({\rm i}\omega I_n - \widetilde{A})^{-1} \widetilde{B}
\end{equation}
which can be considered as the transfer function of the
system
\begin{equation}\label{eq:Disc}
	\dot x = \widetilde{A} x + \widetilde{B} u, \quad
	y = \widetilde{C} x
\end{equation}
on the imaginary axis.
Theorem \ref{thm_unstr} suggests that if
$\widetilde{G}_R({\rm i} \omega) := \widetilde{C}({\rm i}\omega I_n - \widetilde{A})^{-1} \widetilde{B}$
is not identically zero, then
$r(R; B,C)$ and $r(J; B,C)$
are finite, and characterized by
%
\begin{equation}\label{char1}
\begin{split}
	r(R; B,C) \; = \; r(J; B,C) & \; = \; \inf_{\omega\in\mathbb{R}}\frac{1}{ \| \widetilde{G}_R({\rm i} \omega) \|_2}		\\
					& \; = \ \frac{1}{\sup_{\omega\in\mathbb{R}} \| \widetilde{G}_R({\rm i} \omega) \|_2}
					   \; = \; \frac{1}{\;\; \| \widetilde{G}_R \|_{\mathcal H_\infty}},
\end{split}
\end{equation}
where $\| \widetilde{G}_R \|_{{\mathcal H}_\infty} := \sup_{\omega \in {\mathbb R}} \: \sigma_{\max} (\widetilde{G}_R({\rm i} \omega))$
denotes the ${\mathcal H}_\infty$-norm of $\widetilde{G}_R$, and $\sigma_{\max}(\cdot)$ denotes the maximal singular value.

For the stability radius $r(Q; B,C)$, consideration of $G_Q({\rm i} \omega) : = C({\rm i}\omega I_n-(J - R)Q)^{-1}(J-R)B$,
by setting $\widetilde{A} = (J - R)Q$, $\widetilde{B} = (J - R )B$ and $\widetilde{C} = C$, leads us to a
similar characterization.

\subsection{A Subspace Framework}\label{sec:sf_descriptor}
Recently, in \cite{Aliyev2017}, a subspace framework for the computation of the ${\mathcal H}_\infty$-norm of a large-scale system has been
proposed, which is inspired from model order reduction techniques,
and has made the computation  of ${\mathcal H}_\infty$-norms feasible for very large control systems. We will now discuss how to use these techniques
for the computation of the unstructured stability radii
$r(R; B,C)$, $r(J; B,C)$, $r(Q; B,C)$ in the large-scale setting.

To briefly summarize the iterative procedure in the subspace framework of~\cite{Aliyev2017}, let us assume that in iteration $k$, two subspaces ${\mathcal V}_k$ and ${\mathcal W}_k$ of equal dimension have been determined, as well as
matrices $V_k$ and $W_k$ whose columns span orthonormal bases for these subspaces. Applying a Petrov-Galerkin projection to system (\ref{eq:Disc}), restricts the state $x$ to ${\mathcal V}_k$, i.e., in (\ref{eq:Disc}) we replace $x$
by $V_k x_k$, and imposes that the residual after this restriction is orthogonal to ${\mathcal W}_k$. This projection gives rise to a reduced order system
\begin{equation}\label{eq:red_sys1}
	\dot x_k = \widetilde{A}_k x_k+ \widetilde{B}_k u, \quad
	y_k = \widetilde{C}_k x_k,
\end{equation}
with
\begin{equation}\label{eq:red_sys2}
	\widetilde{A}_k := W_k^H \widetilde{A} V_k,	\;\;	\widetilde{B}_k := W_k^H \widetilde{B},	\;\;	
	\widetilde{C}_k = \widetilde{C} V_k.
\end{equation}
Then the ${\mathcal H}_\infty$-norm of a transfer function $G(s) :=  \widetilde{C}( s I_n - \widetilde{A})^{-1} \widetilde{B}$ in (\ref{eq:Disc}) can be approximated by computing the ${\mathcal H}_\infty$-norm of
\[
	G_k(s) := \widetilde{C}_k( s I_k - \widetilde{A}_k)^{-1} \widetilde{B}_k
\]
for instance by employing the method in \cite{Boyd1990} or \cite{Bruinsma1990}, in particular, by computing
$	\omega_{k+1}	:=	\argmax_{\omega \in {\mathbb R}} \: \sigma_{\max} ( G_k( \ri \omega) )$.
This is computationally cheap if the dimensions of ${\mathcal V}_k, {\mathcal W}_k$ are small. Once $\omega_{k+1}$ has been computed, then the subspaces ${\mathcal V}_k$ and ${\mathcal W}_k$ are expanded into larger subspaces
${\mathcal V}_{k+1}$ and ${\mathcal W}_{k+1}$ in such a way that the corresponding reduced
transfer function $G_{k+1}(s)$ satisfies the \emph{Hermite interpolation conditions}
\begin{equation}\label{eq:Hermite_interpolate}
	\begin{split}
	\sigma_{\max} (G(\ri \omega_{k+1}))	& = \sigma_{\max} (G_{k+1}(\ri \omega_{k+1})),	\\		
	\sigma_{\max}' (G(\ri \omega_{k+1}))	& = \sigma_{\max}' (G_{k+1}(\ri \omega_{k+1})),
	\end{split}
\end{equation}
where  $\sigma_{\max}'(G({\rm i}\omega))$ denotes the derivative of $\sigma_{\max} (G({\rm i} \omega))$
with respect to $\omega$.
Denoting the image space of a matrix by $A$ by $  {\rm Im} (A)$, it is shown in \cite{Aliyev2017} that 
\[
	\begin{split}
	{\mathcal V}_{k+1} & := {\mathcal V}_k \oplus {\rm Im} (( \ri \omega_{r+1} I_n - \widetilde{A})^{-1} \widetilde{B}), \\
	{\mathcal W}_{k+1} & := {\mathcal W}_k \oplus {\rm Im} ( (\widetilde{C} ( \ri \omega_{k+1} I_n - \widetilde{A})^{-1})^H),
	\end{split}
\]
more specifically the inclusions
\[
	{\rm Im} (( \ri \omega_{k+1} I_n - \widetilde{A})^{-1} \widetilde{B}) \subseteq {\mathcal V}_{k+1},\
	{\rm Im} ( (\widetilde{C} ( \ri \omega_{k+1} I_n - \widetilde{A})^{-1})^H) \subseteq {\mathcal W}_{k+1},
\]
ensure that the Hermite interpolation conditions (\ref{eq:Hermite_interpolate}) are satisfied.
The procedure is then repeated with the expanded subspaces ${\mathcal V}_{k+1}$, ${\mathcal W}_{k+1}$. In \cite{Aliyev2017}, it is shown that the sequence $\{ \omega_k \}$  converges at a super-linear rate and satisfies
\[
	\begin{split}
	\sigma_{\max} (G(\ri \omega_{j}))	& = \sigma_{\max} (G_{k}(\ri \omega_{j})),	\\		
	\sigma_{\max}' (G(\ri \omega_{j}))	& = \sigma_{\max}' (G_{k}(\ri \omega_{j}))
	\end{split}
\]
for $j = 1,\dots,k$.

A disadvantage of this general approach is that even if $\widetilde{A} = (J - R)Q$ has DH structure, this is not necessarily
true for  $\widetilde{A}_k$, so it cannot be guaranteed from the structure that the reduced system is stable.
In the next section we modify the procedure of~\cite{Aliyev2017} to preserve the DH structure.

\subsection{A Structure Preserving Subspace Framework}\label{sec:sf_PH}
In this subsection we derive an interpolating,  DH structure preserving version of the robust subspace projection framework.
Structure preserving subspace projection methods in the context of model order reduction of
large-scale PH and DH systems have been proposed in \cite{Gugercin_2012,Gugercin_2009,Polyuga_2009,Polyuga_2011,Schaft_2009,Wu_2014}.
Our approach is inspired by  \cite{Gugercin_2012} and uses a general interpolation result from \cite{Gallivan2005}.
\begin{theorem}\label{thm:Gal_interpolate}
Let $G(s)$ be the transfer function for a full order system as in (\ref{eq:Disc}), and let $G_k(s)$ be the transfer
function for the reduced system defined by (\ref{eq:red_sys1}), (\ref{eq:red_sys2}).
\begin{enumerate}
\item[\bf (i)] \emph{(Right Tangential Interpolation)} For given $\widehat{s} \in {\mathbb C}$ and $\widehat{b} \in {\mathbb C}^m$, if
 	\begin{equation}\label{eq:subspace_incl}
		\left[ (\widehat{s}I - \widetilde{A})^{-1} \right]^\ell \widetilde{B} \widehat{b}	\; \in \; {\mathcal V}_k		\quad \quad {\rm for} \;\;	\ell = 1, \dots, N,
 	\end{equation}
 	and $\: {\mathcal W}_k$ is such that $W_k^H V_k = I$, then we have
	\begin{equation}\label{eq:tangent_interpolate}
	 	G^{(\ell)} (\widehat{s}) \widehat{b}	=	G^{(\ell)}_k (\widehat{s}) \widehat{b}	\quad\quad {\rm for} \;\;	\ell = 0, \dots, N-1
	 \end{equation}
	 provided that both \, $\widehat{s} I - \widetilde{A}$ and\, $\widehat{s} I - \widetilde{A}_k$ are invertible.
\item[\bf (ii)] \emph{(Left Tangential Interpolation)} For a given $\widehat{s} \in {\mathbb C}$ and $\widehat{c} \in {\mathbb C}^p$, if
	\begin{equation}\label{eq:subspace_inclusion_left}
		\left( \widehat{c}^{\: H} \widetilde{C} \left[ (\widehat{s}I - \widetilde{A})^{-1} \right]^\ell \right)^H \; \in \; {\mathcal W}_k						\quad \quad {\rm for} \;\;	\ell = 1, \dots, N,
 	\end{equation}
 	and $\: {\mathcal V}_k$ is such that $W_k^H V_k = I$, then we have
	\begin{equation}\label{eq:tangent_interpolate_left}
	 	\widehat{c}^{\: H} G^{(\ell)} (\widehat{s}) 	=	\widehat{c}^{\: H} G^{(\ell)}_k (\widehat{s}) 	
														\quad\quad {\rm for} \;\;	\ell = 0, \dots, N-1
	 \end{equation}
	 provided that both \, $\widehat{s} I - \widetilde{A}$ and \, $\widehat{s} I - \widetilde{A}_k$ are invertible.
\end{enumerate}
\end{theorem}

\subsubsection{Computation of $r(R; B,C)$ and $r(J; B,C)$}\label{sec:unstructured_subspace_rR_rJ}
The computation of $r(R; B, C) = r(J; B,C)$ involves the maximization of the largest singular value
of the transfer function $G(s) = CQ (sI - (J-R)Q)^{-1} B$ associated with the system
\begin{equation}\label{eq:PHS}
	\begin{split}
	\dot x &= (J-R) Qx + B u, \\
 	y &= CQx.
	\end{split}
\end{equation}
on the imaginary axis. We make use of Theorem~\ref{thm:Gal_interpolate}
to obtain a reduced order system satisfying the interpolation conditions (\ref{eq:tangent_interpolate})
while retaining the structure in (\ref{eq:PHS}). We, in particular, employ right tangential interpolation
for a given $\widehat{s} \in {\mathbb C}$ and $\widehat{b} \in {\mathbb C}^m$, and
choose ${\mathcal V}_k$ as any subspace satisfying (\ref{eq:subspace_incl}).
Let us also define $	W_k	\;\;	:=	\;\;	QV_k (V_k^H Q V_k)^{-1}$,
${\mathcal W}_k := {\rm Im} (W_k)$, so that
\[
	 W_k^H V_k	\;\; = \;\;	I_k
		\quad	{\rm and}	\quad
	( W_k V_k^H )^2		\;\; = \;\; W_k V_k^H,
\]
i.e., $W_k V_k^H$ is an oblique projector onto ${\rm Im}( Q V_k)$.

The matrices $\widetilde{A}_k, \widetilde{B}_k, \widetilde{C}_k$ of the
reduced system  (\ref{eq:red_sys1}), (\ref{eq:red_sys2}) for these choices of $V_k$ and $W_k$ then satisfy
\begin{equation}\label{eq:reduced_A}
	\begin{split}
		\widetilde{A}_k 	& \;\; = \;\;	W_k^H \widetilde{A} V_k	=	W_k^H (J-R) Q V_k		=	W_k^H (J-R) W_k V_k^H Q V_k	\\					
					& \;\; = \;\; (J_k - R_k) Q_k,
	\end{split}
\end{equation}
where $J_k := W_k^H J W_k=-J_k^H$,  $R_k := W_k^H R W_k=R_k^H\geq 0$, and $Q_k := V_k^H Q V_k=Q_k^H>0$.
Additionally, $
	\widetilde{C}_k  \;\; = \;\; \widetilde{C} V_k = CQ V_k  = C W_k V_k^H Q V_k = C_k Q_k$,
where $C_k : = C W_k$, and $\widetilde{B}_k \;\; = \;\; W_k^H B  \;\; =: \;\; B_k$.

This construction leads to the following  result of \cite{Gugercin_2012}.
\begin{theorem}
\label{interpol2}
Consider a linear system of the form \eqref{eq:PHS}
with transfer function $G(s) := CQ(s I_n - (J - R)Q)^{-1}B$.
Furthermore, for a given point $\widehat{s}\in\C$ and a given tangent direction $\widehat{b}\in\C^m$,
suppose that $V_k$ is a matrix with orthonormal columns such that
\begin{equation*}
( \widehat{s}I_n - (J - R)Q  )^{-(\ell-1)} (\widehat{s}I_n - (J - R)Q )^{-1}B\widehat{b} \in{\rm Im}(V_k) \quad \text{for} \;\; \ell = 1,\ldots,N.
\end{equation*}
Define $\: W_k := QV_k(V_k^HQV_k)^{-1} \:$ and set
\begin{equation}\label{eq:DH_subspaces}
	\begin{split}
	J_k: = & W_k^H J W_k, \;\; Q_k: = V_k^H Q V_k, \;\; R_k := W_k^H R W_k  \\
	B_k: =& W_k^H B, \;\; C_k: = CW_k.
	\end{split}
\end{equation}
Then the resulting reduced order model
\begin{equation}\label{eq:RPHS}
	\begin{split}
		\dot x_k& \;\; = \;\; (J_k-R_k) Q_k x_k + B_k u,  \\
 		y_k& \;\; = \;\; C_k Q_kx_k
	\end{split}
\end{equation}
is a DH system with transfer function
\begin{equation}
\label{eq:rph}
G_k(s)  \;\; := \;\; C_kQ_k(s I_k - (J_k - R_k)Q_k)^{-1}B_k
\end{equation}
that satisfies
\begin{equation}
G^{(j)}(\hat{s})\hat{b} \;\; = \;\; G^{(j)}_k(\hat{s})\hat{b} \quad \text{for} \quad  j = 0\ldots, N-1,
\end{equation}
where $G^{(j)}(\hat{s})$ denotes the $j$-th derivative of $\; G(s)$ at the point $\hat{s}$.
\end{theorem}

Based on Theorem~\ref{interpol2} we obtain Algorithm \ref{alg:dti} for the computation of $r(R; B,C) = r(J; B,C)$.
\begin{algorithm}[ht]
\label{algor1}
 \begin{algorithmic}[1]
\REQUIRE{Matrices $B \in \mathbb{C}^{n\times m}$, $C \in \mathbb{C}^{p\times n}$, $ J, R, Q \in \mathbb{C}^{n\times n}$.}
\ENSURE{The sequence of frequencies $\{ \omega_k \}$.}
\STATE Choose initial interpolation points $\omega_{1}, \dots , \omega_{q} \in {\mathbb R}.$
	\STATE $V_q \gets  {\rm orth} \begin{bmatrix} D(\ri\omega_1)^{-1}B  & D(\ri\omega_1)^{-2}B
									&	\dots		&	D(\ri\omega_q)^{-1}B  & D(\ri\omega_q)^{-2}B 	\end{bmatrix}$	\label{defn_init_subspaces0} \\
	\hskip 19ex $\quad \text{and}\quad \quad
	W_q \gets QV_q(V_q^HQV_q)^{-1} $. \label{defn_init_subspaces}
\FOR{$k = q,\,q+1,\,\dots$}
	\STATE  Form $G_{k}$ as in \eqref{eq:rph} for the choices of $J_k, R_k, Q_k, B_k, C_k$ in (\ref{eq:DH_subspaces}). \label{reduced_transfer}
	\STATE
	$\; \displaystyle \omega_{k+1} \gets \: \argmax_{\omega \in {\mathbb R}} \sigma_{\max}(G_{k} (\ri \omega))$. \label{solve_reduced}
	\STATE $\widehat{V}_{k+1} \gets  \begin{bmatrix} D(\ri\omega_{k+1})^{-1}B  & D(\ri\omega_{k+1})^{-2}B. \end{bmatrix}$
	\label{defn_later_subspaces}
	\STATE $V_{k+1} \gets \operatorname{orth}\left(\begin{bmatrix} V_{k} & \widehat{V}_{k+1} \end{bmatrix}\right)
		\quad \text{and}\quad W_{k+1} \gets Q V_{k+1}(V_{k+1}^H Q V_{k+1})^{-1}$.
\ENDFOR
\end{algorithmic}
\caption{$\;$ DH structure preserving subspace method for the
computation of the stability radii $r(R; B,C)$ and $r(J; B,C)$ for large-scale systems.}
\label{alg:dti}
\end{algorithm}
According to Theorem~\ref{interpol2}, for a given $\widehat{s} \in {\mathbb C}$, setting
\[
	V_k :=  \begin{bmatrix} D(\widehat{s})^{-1}B  & D(\widehat{s})^{-2}B \end{bmatrix},\ 
	W_k := QV_k(V_k^HQV_k)^{-1}
\]
where
\begin{equation}\label{eq:defn_Ds}
	D(\widehat{s}): = (\widehat{s} I_n - (J - R)Q),
\end{equation}
we obtain $G(\widehat{s}) = G_k(\widehat{s})$ and $G'(\widehat{s}) = G'_k(\widehat{s})$ and thus  the Hermite interpolation conditions
\begin{equation}\label{eq:Hermite_inter_smax}
	\sigma_{\max}(G(\widehat{s})) = \sigma_{\max}(G_k(\widehat{s})),\	
\sigma'_{\max}(G(\widehat{s})) = \sigma'_{\max}(G_k(\widehat{s}))
\end{equation}
are satisfied, which suggest the use of the reduced system in the greedy subspace framework outlined in Algorithm~\ref{alg:dti}.

In line 5 of every iteration, the subspace framework computes the ${\mathcal H}_\infty$-norm of a reduced system, in particular it computes the point ${\rm i} \omega_\ast$ on the imaginary axis where this ${\mathcal H}_\infty$-norm
is attained. Then the current left and right subspaces are expanded in a way so that the resulting reduced system
still has DH structure and its transfer function Hermite interpolates the original transfer function
at ${\rm i} \omega_\ast$. Since the Hermite interpolation conditions~(\ref{eq:Hermite_inter_smax}) are satisfied
at $\widehat{s} = {\rm i} \omega_1, \dots, {\rm i} \omega_{k}$ at the end of iteration $k$, the rate-of-convergence
analysis in \cite{Aliyev2017} applies to deduce a superlinear rate-of-convergence for the sequence $\{ \omega_k \}$.

The computationally most expensive part of Algorithm \ref{alg:dti} is  in lines \ref{defn_init_subspaces} and \ref{defn_later_subspaces}, where many linear systems with possibly many right hand sides have to be solved. If this is done with a direct solver, then for each value $\widehat{\omega} \in {\mathbb R}$  one $LU$ factorization
of the matrix $D(\ri \widehat{\omega})$ has to be performed.
For large values of $n$, the computation time is usually dominated by these $LU$ factorizations. In contrast to this,
the solution of the reduced problem in line \ref{solve_reduced} can be achieved (for small systems) by means of the
efficient algorithm in \cite{Boyd1990, Bruinsma1990}.

\subsubsection{Computation of $r(Q; B,C)$}
To compute the stability radius $r(Q; B,C)$ in the large-scale setting, we employ left tangential interpolations
(i.e., part (ii) of Theorem \ref{thm:Gal_interpolate}).

In this case $r(Q; B,C)$ is the reciprocal of the ${\mathcal H}_\infty$-norm of the transfer function
$G(s) := C (sI - (J-R)Q)^{-1} (J-R)B$ corresponding to the system
\begin{equation}\label{eq:PHS2}
	\dot x = (J-R) Qx + (J-R)B u, \quad
 	y(t) = Cx.
\end{equation}
To obtain a reduced system which has the same structure as (\ref{eq:PHS2}) and has a transfer function $G_k(s)$ that satisfies $\widehat{c}^{\: H} G(\widehat{s})  = \widehat{c}^{\: H}G_k(\widehat{s})$ for a given point $\widehat{s} \in {\mathbb C}$ and a direction $\widehat{c} \in {\mathbb C}^p$,
let us choose $W_k$ so as to satisfy the condition in (\ref{eq:subspace_inclusion_left})
for $\widetilde{A} := (J-R)Q$, and $\widetilde{C} := C$. Furthermore, we set
\[
	V_k := (J - R)^H W_k (W_k^H (J - R)^H W_k)^{-1},
\]
as well as ${\mathcal W}_k := {\rm Im}(W_k)$, ${\mathcal V}_k := {\rm Im}(V_k)$. The matrix $V_k$ is  chosen to satisfy
\[
	V_k^H W_k \;\; = \;\; I_k		\quad	{\rm and}	\quad		
						(V_k W_k^H)^2	\;\; = \;\; V_k W_k^H,
\]
so that $V_k W_k^H$ is an oblique projector onto ${\rm Im} ( (J-R)^H W_k )$.

In (\ref{eq:PHS2}), setting $\widetilde{A} := (J-R)Q$, $\widetilde{B} := (J-R)B$, $\widetilde{C} = C$,
let us investigate the matrices $\widetilde{A}_k, \widetilde{B}_k, \widetilde{C}_k$ of the corresponding reduced
system defined by (\ref{eq:red_sys1}), (\ref{eq:red_sys2}). Specifically, we have that
\[
	\widetilde{A}_k \;\; = \;\; W_k^H \widetilde{A} V_k	\; = \; W_k^H (J-R) Q V_k	\; = \;		W_k^H (J-R) W_k V_k^H Q V_k,
\]
where, in the third equality, we employ $V_k W_k^H (J - R)^H W_k = (J- R)^H W_k$,
or equivalently  $W_k^H (J- R) W_k V_k^H = W_k^H (J-R)$.
Hence, defining $J_k := W_k^H J W_k=-J_k^H$, $R_k := W_k^H R W_k=R_k^H\geq 0$, $Q_k := V_k^H Q V_k=Q_k^H>0$,
we obtain a DH system with $\widetilde{A}_k \;\; = \;\; (J_k - R_k) Q_k$.
We also have
\[
		\widetilde{B}_k   \;\; = \;\; W_k^H \widetilde{B}  = W_k^H (J-R) B  = W_k^H (J - R) W_k V_k^H B   =  (J_k - R_k) B_k
\]
with $B_k := V_k^H B$, and $\widetilde{C}_k  \;\; = \;\; C V_k \;\; =: \;\; C_k$.
These constructions lead to the following analogue of Theorem \ref{interpol2}.
\begin{theorem}
\label{interpol3}
Consider a linear system of the form \eqref{eq:PHS2}
with the transfer function $G(s) := C(s I_n - (J - R)Q)^{-1}(J-R)B$.
For a given point $\widehat{s}\in\C$ and a direction $\widehat{c}\in\C^m$,
suppose that $W_k$ is a matrix with orthonormal columns such that
\begin{equation*}
\left( \widehat{c}^{\: H} C ( \widehat{s}I_n - (J - R)Q  )^{-1} (\widehat{s}I_n - (J - R)Q )^{-(\ell-1)} \right)^H \in{\rm Im}(W_k)
\end{equation*}
for $\ell= 1, \dots, N$.
Letting $\: V_k := (J - R)^H W_k (W_k^H (J - R)^H W_k)^{-1} \:$, and
\begin{equation}\label{eq:DH_subspaces2}
	\begin{split}
	J_k: = & W_k^H J W_k \;\; Q_k: = V_k^H Q V_k, \;\; R_k := W_k^H R W_k  \\
	B_k: =& V_k^H B, \;\; C_k: = CV_k,
	\end{split}
\end{equation}
the resulting reduced order system
\begin{equation}\label{eq:RPHS2}
		\dot x_k \;\; = \;\; (J_k-R_k) Q_k x_k + (J_k - R_k) B_k u,  \quad
 		y_k \;\; = \;\; C_k x_k
\end{equation}
is such that $\dot x_k = (J_k-R_k) Q_k x_k$ is dissipative Hamiltonian.
Furthermore, the transfer function
\begin{equation}
\label{eq:rph2}
G_k(s)  \;\; := \;\; C_k (s I_k - (J_k - R_k)Q_k)^{-1} (J_k - R_k) B_k
\end{equation}
of (\ref{eq:RPHS2}) satisfies
\begin{equation}
\widehat{c}^{\: H} G^{(\ell)}(\hat{s})\hat{b} \;\; = \;\; \widehat{c}^{\: H} G^{(\ell)}_k(\hat{s}) \quad \text{for} \quad  \ell = 0\ldots, N-1.
\end{equation}
\end{theorem}
Theorem \ref{interpol3} shows that at a given $\widehat{s} \in {\mathbb C}$, the
Hermite interpolation properties $G(\widehat{s}) = G_k(\widehat{s})$ and $G'(\widehat{s}) = G_k'(\widehat{s})$ and, in particular,
$\sigma_{\max} (G(\widehat{s}))$ $=$ $\sigma_{\max}(G_k(\widehat{s}))$ and $\sigma_{\max}'(G(\widehat{s})) = \sigma_{\max}'(G_k(\widehat{s}))$)
can be achieved, while preserving the structure, with the choices
\begin{eqnarray*}
	W_k \; &=& \;
		\left[
			\begin{array}{cc}
				(C D(\widehat{s})^{-1})^H	&	(C D(\widehat{s})^{-2})^H
			\end{array}
		\right],\\
	V_k	\;	&=&	\;	(J - R)^H W_k (W_k^H (J - R)^H W_k)^{-1},
\end{eqnarray*}
where $D(\widehat{s})$ is as in (\ref{eq:defn_Ds}). This in turn gives rise to Algorithm \ref{alg:dti2}.
\begin{algorithm}[ht]
 \begin{algorithmic}[1]
\REQUIRE{Matrices $B \in \mathbb{C}^{n\times m}$, $C \in \mathbb{C}^{p\times n}$, $ J, R, Q \in \mathbb{C}^{n\times n}$.}
\ENSURE{The sequence $\{ \omega_k \}$.}
\STATE Choose initial interpolation points $\omega_{1}, \dots, \omega_{j} \in {\mathbb R}$.
	\STATE
		$W_j \gets
					{\rm orth}
					\left[ 						
						\left( C D(\ri\omega_1)^{-1} \right)^H
							\;\; \left( C D(\ri\omega_1)^{-2} \right)^H
									\;\;	\dots		\right.$ \\
		\hskip 35ex $	\left. \;\;	\left( C D(\ri\omega_j)^{-1} \right)^H
											\;\;  \left( C D(\ri\omega_j)^{-2} \right)^H
							\right]$, \\
		$\quad\quad\quad\quad V_j\gets (J - R)^H W_j (W_j^H (J - R)^H W_j)^{-1} $. \label{defn_init_subspaces2}
\FOR{$k = j,\, j+1,\,\dots$}
	\STATE  Form $G_{k}$ as in \eqref{eq:rph2} for the choices of $J_k$, $R_k$, $Q_k$, $B_k$, $C_k$ in (\ref{eq:DH_subspaces2}). \label{reduced_transfer2}
	\STATE
	$\; \displaystyle \omega_{k+1} \gets \argmax_{\omega \in {\mathbb R}} \sigma_{\max}(G_{k} (\ri \omega))$. \label{solve_reduced2}
	\STATE $\widehat{W}_{k+1} \gets
	\begin{bmatrix} \left( C D(\ri\omega_{k+1})^{-1} \right)^H & \left( C D(\ri\omega_{k+1})^{-2} \right)^H \end{bmatrix}$. \label{defn_later_subspaces2}
	\STATE $W_{k+1} \gets \operatorname{orth}\left(\begin{bmatrix} W_{k} & \widehat{W}_{k+1} \end{bmatrix}\right)
			\quad \text{and}	$	\\	
				\hskip 15ex 	$V_{
k+1} \gets (J - R)^H W_{k+1} (W_{k+1}^H (J - R)^H W_{k+1})^{-1}$.
\ENDFOR
\end{algorithmic}
\caption{$\;$ DH structure preserving subspace method for the
computation of the stability radius $r(Q; B,C)$ of a large scale DH system.}
\label{alg:dti2}
\end{algorithm}

At every iteration of this algorithm, the ${\mathcal H}_\infty$-norm is computed for a reduced problem of
the form $(\ref{eq:RPHS2})$, in particular, the optimal frequency where this ${\mathcal H}_\infty$-norm
is attained is retrieved. Then the subspaces are updated so that the Hermite interpolation properties hold
also at this optimal frequency at the largest singular values of the full and reduced problem, respectively.
Once again the sequence $\{ \omega_k \}$ by Algorithm \ref{alg:dti2} is guaranteed to converge at a super-linear rate, which can be attributed to the Hermite interpolation properties holding between  the largest singular values of the full and
reduced transfer functions.

\subsection{Numerical Experiments}\label{sec:DH_numexps}

In this subsection we illustrate the performance of MATLAB implementations of
Algorithms \ref{alg:dti} and \ref{alg:dti2} via some numerical examples. We first discuss some implementation details and then present numerical results on two sets of random synthetic
examples in Section \ref{sec:numexp_syn}, and data from a FE
model  of a brake disk in Section~\ref{sec:numexp_brake}.

\subsubsection{Implementation Details and Test Setup}\label{sec:test_setup}

Algorithms \ref{alg:dti} and \ref{alg:dti2} are terminated
when at least one of the following three conditions
is fulfilled:
\begin{enumerate}
	\item The relative distance between $\omega_{k}$ and $\omega_{k-1}$ is less than a prescribed tolerance
	for some $k> j$, i.e.,
	\[
		\left| \omega_k- \omega_{k-1} \right| < \varepsilon \cdot \frac{1}{2} (\omega_k + \omega_{k-1}).
	\]
	\item Letting $f_k := \max_{\omega \in {\mathbb R} \cup \infty} \sigma_{\max}(G_k({\rm i} \omega))$, two
	consecutive iterates $f_k, f_{k-1}$ are close enough in a relative sense, i.e.,
	\[
		\left| f_k - f_{k-1} \right| < \varepsilon \cdot \frac{1}{2} (f_k + f_{k-1}).
	\]
	\item The number of iterations exceeds a specified integer, i.e.,
	$k> k_{\max}$.
\end{enumerate}
In all numerical examples that we present,  we set $\varepsilon = 10^{-6}$ and $k_{\max} = 100$.

In general, Algorithms \ref{alg:dti} and \ref{alg:dti2} converge only locally. The choice of the initial interpolation points affects
the maximizers that the subspace frameworks converge to, in particular, whether convergence to a global
maximizer occurs. The initial interpolation points are chosen based on the following procedure.

First, we discretize the interval $[ \lambda_{\min}^{\Im} , \lambda_{\max}^{\Im} ]$
into $\rho$ equally spaced points, say $\omega_{0,1}, \dots, \omega_{0,\rho}$, including the end-points
$\lambda_{\min}^{\Im}, \lambda_{\max}^{\Im}$, where $\rho$ is specified by the user and
$\lambda_{\min}^{\Im}, \lambda_{\max}^{\Im}$ denote the imaginary parts of the eigenvalues of $(J-R)Q$ with
the smallest and largest imaginary part, respectively. Then we approximate the eigenvalues $z_1, \dots, z_\rho$ of $(J-R)Q$
closest to ${\rm i} \omega_{0,1}, \dots, {\rm i} \omega_{0,\rho}$, and permute them into
$z_{j_1}, \dots, z_{j_\rho}$ where $\{ j_1, \dots, j_\rho \} = \{ 1, \dots, \rho \}$ so that
$\sigma_{\max}(G({\rm i} \Im z_{j_1} )) \geq \dots \geq \sigma_{\max}(G({\rm i} \Im z_{j_\rho} ))$.
The interpolation points $\omega_1, \dots, \omega_\ell$ employed initially are then chosen as the imaginary parts
of $z_{j_1}, \dots, z_{j_\ell}$, where again $\ell\leq \rho$ is specified by the user.

\subsubsection{Results on Synthetic Examples}\label{sec:numexp_syn}
We now present results for two families of linear DH systems with random coefficient matrices; the first family consists of dense systems of order $800$,
whereas the second family consists of sparse systems of order $5000$.

\medskip

\noindent
\textbf{Dense Random examples.}
In the dense family the coefficient matrices $J$, $Q$, $R$  are formed by the MATLAB commands
\begin{verbatim}
	>> J = randn(800);	J = (J - J')/2;
	>> Q = randn(800);	Q = (Q + Q')/2;	mineig = min(eig(Q));
	>> if (mineig < 10^-4) Q = Q + (-mineig + 5*rand)*eye(n); end
	>> p = round(80*rand);
	>> Rp = randn(p); Rp = (Rp + Rp')/2; mineig = min(eig(Rp));
	>> if (mineig < 10^-4) Rp = Rp + (-mineig + 5*rand)*eye(p); end
	>> R = [Rp zeros(p,800-p); zeros(800-p,p) zeros(800-p,800-p)];
	>> X = randn(800); [U,~] = qr(X); R = U'*R*U;
\end{verbatim}
The restriction matrices $B$ and $C$ are chosen as $800\times 2$ and $2\times 800$
random matrices created by the MATLAB command \texttt{randn}. To compute $r(R; B,C) = r(J; B,C)$, as well as $r(Q;B,C)$, we ran
\begin{enumerate}
	\item[(1)] the Boyd-Balakrishnan (BB) algorithm \cite{Boyd1990},
	\item[(2)] the subspace framework that does not preserve the DH structure
	\cite[Algorithm 1]{Aliyev2017} described in Subsection \ref{sec:sf_descriptor}, and
	\item[(3)] the subspace frameworks that preserve structure, i.e., Algorithms \ref{alg:dti} and \ref{alg:dti2},
	introduced in Subection \ref{sec:sf_PH}
\end{enumerate}
on $100$ such random examples.
The spectrum of a typical $(J-R)Q$ of size $800$ generated in this way is depicted
in Figure \ref{fig:spectra_randomPH} on the left.
	\begin{figure}
	\hskip -2.5ex
		\begin{tabular}{cc}
			\includegraphics[width=6.cm]{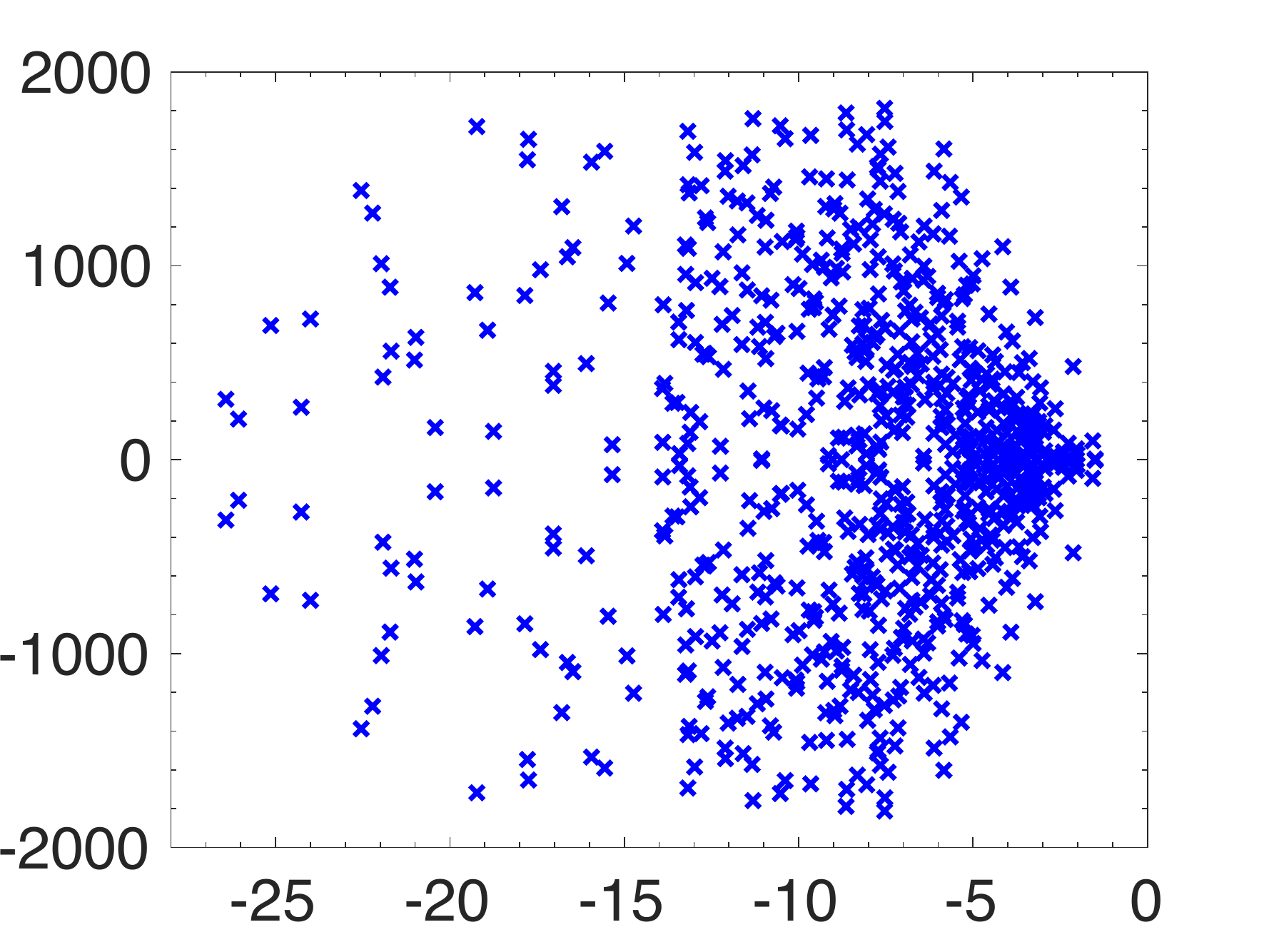} 	&	\includegraphics[width=6.cm]{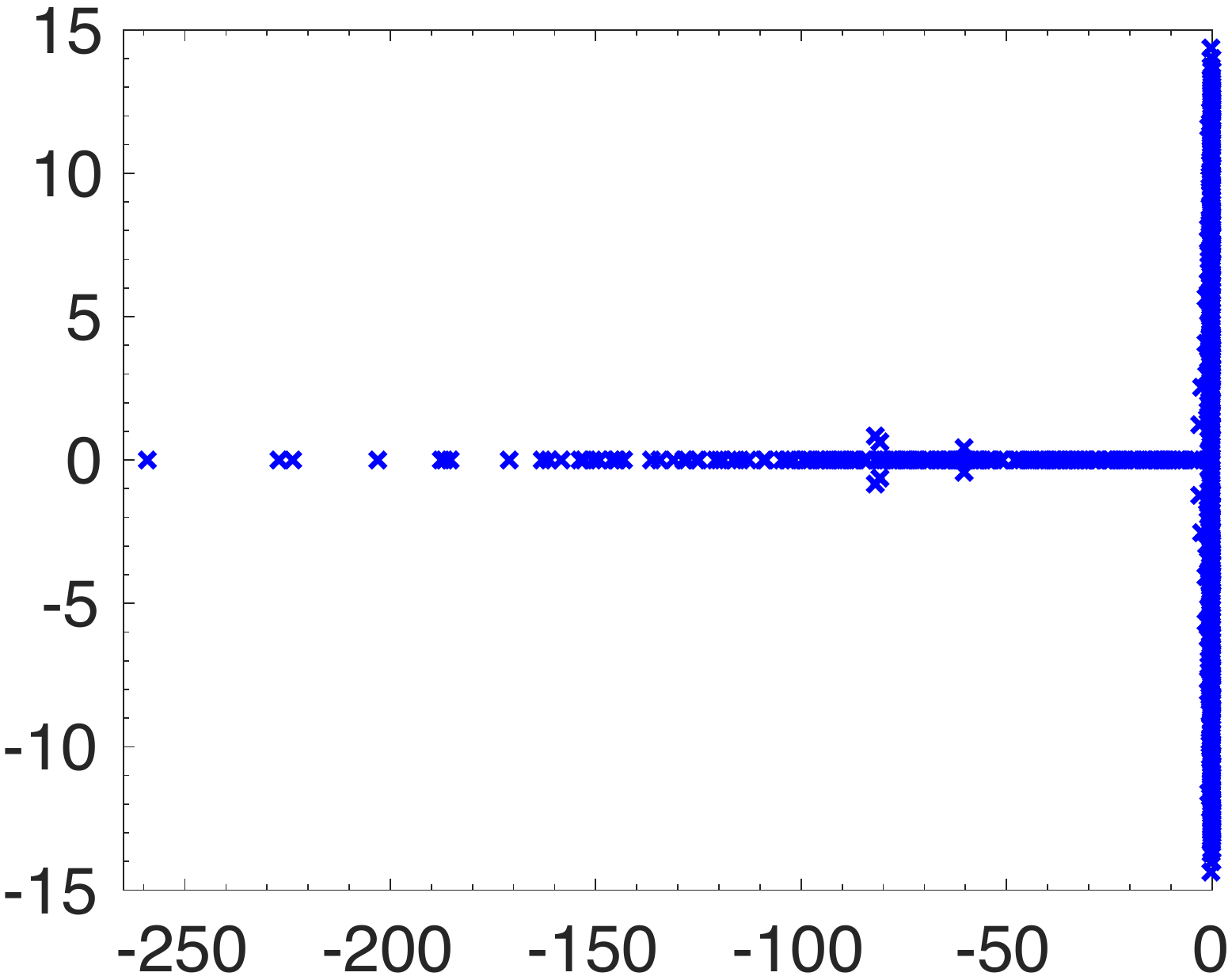}	
		\end{tabular}
		\caption{The spectra of $A = (J-R)Q$ for a dense random $J, R, Q \in {\mathbb R}^{800\times 800}$ (left),
		and a sparse random $J, R, Q \in {\mathbb R}^{5000\times 5000}$ (right). The MATLAB commands
		yielding these $J, R, Q$ are specified in Section \ref{sec:numexp_syn}.}
		\label{fig:spectra_randomPH}
	\end{figure}

The progress of Algorithm \ref{alg:dti2}, as well as Algorithm 1 in \cite{Aliyev2017},
to compute $r(Q; B,C)$ for this example is presented
in  Figure~\ref{fig:progress_alg2}, which includes on the top left a plot of
$f(\omega) :=  \sigma_{\max}(C ({\rm i} \omega I - (J-R) Q)^{-1} (J-R)B)$
for $\omega \in [-2000,0]$ along with the converged maximizers by the respective Algorithms. 
Algorithm \ref{alg:dti2} converges to the global maximizer $\omega_{\ast,1} = -731.9774$
with $f(\omega_{\ast,1}) = 32.321399$, while Algorithm 1 in \cite{Aliyev2017} converges to the local maximizer
$\omega_{\ast,2} = -1602.1187$ with $f(\omega_{\ast,2}) = 29.028197$. The globally optimal peak
$(\omega_{\ast,1}, f(\omega_{\ast,1}))$ and the locally optimal peak $(\omega_{\ast,2}, f(\omega_{\ast,2}))$
are marked in the plot with a square and a circle, respectively.

The remaining five plots in Figure \ref{fig:progress_alg2}
illustrate the progress of Algorithm \ref{alg:dti2}. In each one of these plots, the black curve is a plot of the reduced function
$f_k(\omega) := \max_{\omega \in {\mathbb R} \cup \infty} \sigma_{\max}(C_k ({\rm i} \omega I - (J_k-R_k) Q_k)^{-1} (J_k-R_k)B_k)$
with respect to $\omega$, and the circle marks the global maximizer of this reduced function.
The top right shows the initial reduced function in black interpolating the full function at ten points, and the other four
show the reduced function after iterations $1$-$4$ from middle-left to bottom-right. Observe that, at every iteration, the refined
reduced function interpolates the full function at the maximizer of the previous reduced function in addition
to the earlier interpolation points. We also list the iterates of Algorithm \ref{alg:dti2} in Table \ref{table:dense_iterates_alg2}
indicating a quick converge. The algorithm terminates after performing six subspace iterations.

The results of Algorithms \ref{alg:dti} and \ref{alg:dti2}
for the first $10$ random examples are presented in Tables~\ref{table:dense_comparison} and~\ref{table:dense_comparisonQ}, respectively. Results from \cite[Algorithm 1]{Aliyev2017}
and the BB Algorithm \cite{Boyd1990} are also included in these tables for
comparison purposes. For the computation of $r(J; B,C) = r(R; B,C)$,
the new structure-preserving Algorithm \ref{alg:dti} and \cite[Algorithm 1]{Aliyev2017} perform equally well on these first $10$ examples. They both return the globally optimal solutions in $9$ out of
$10$ examples, perform similar number of subspace iterations and require similar amount of cpu-time.
	\begin{figure}
		\begin{tabular}{cc}
			\includegraphics[width=5.6cm]{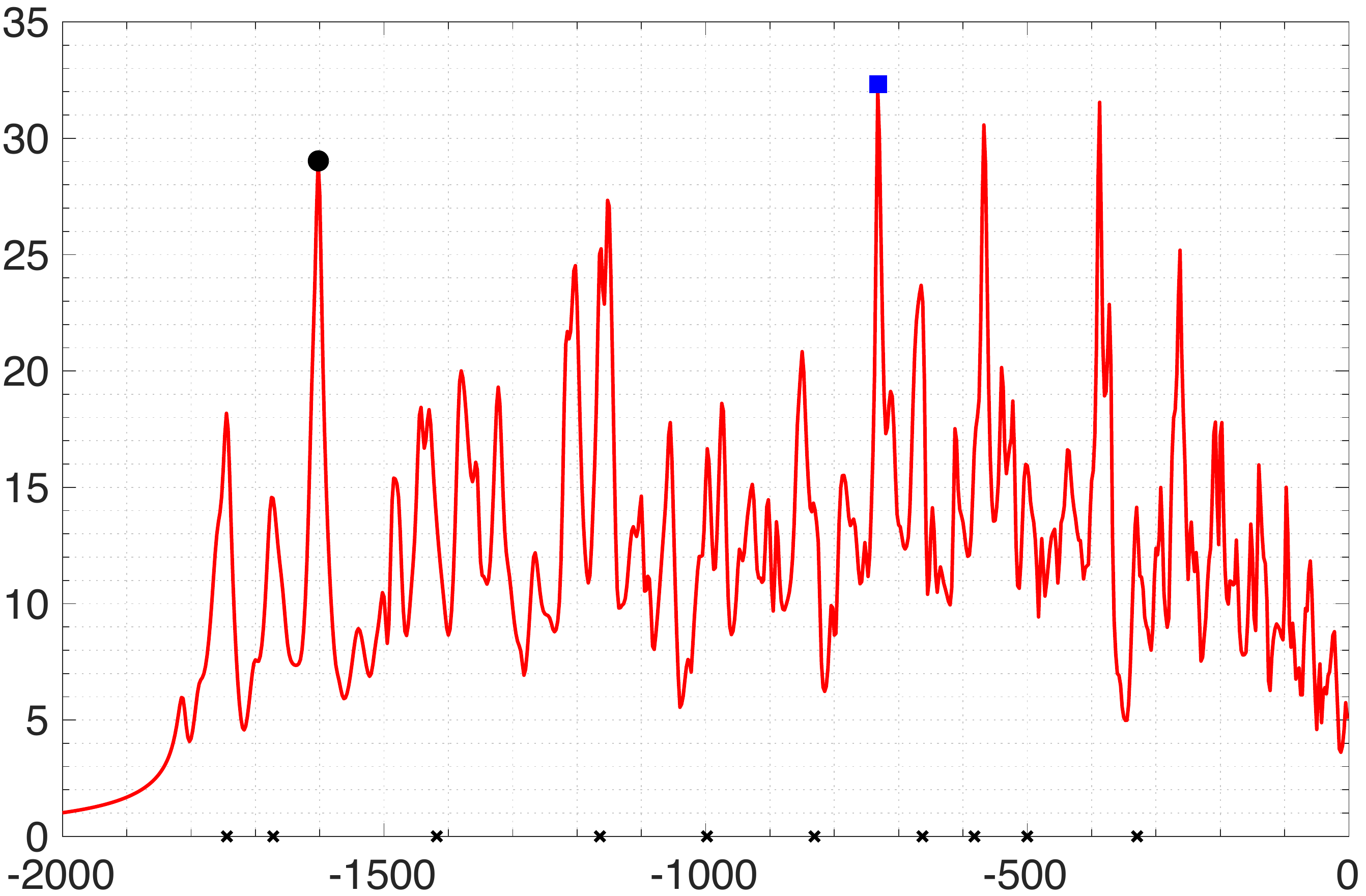} 	&	\includegraphics[width=5.6cm]{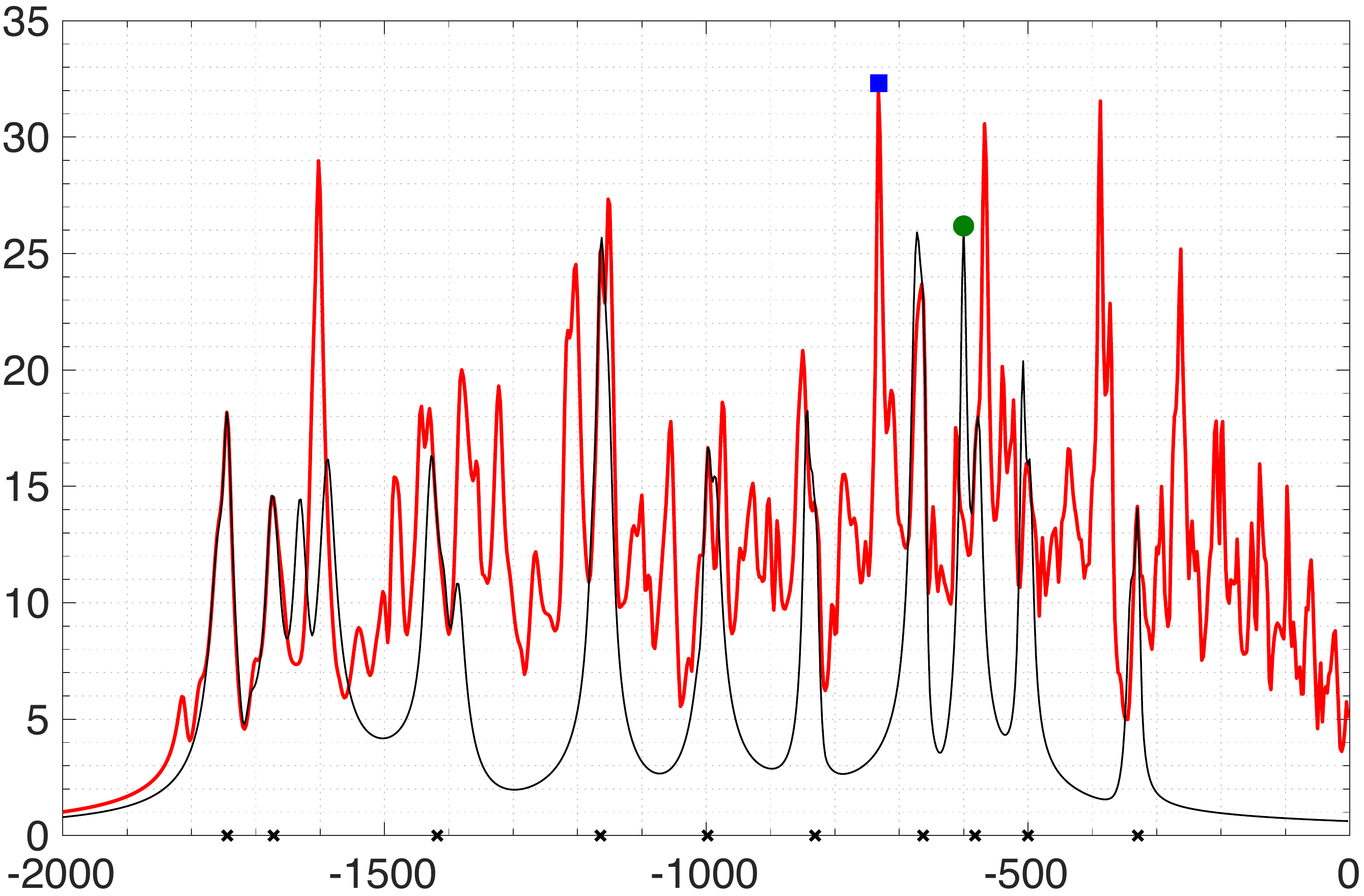}		\\
			\includegraphics[width=5.6cm]{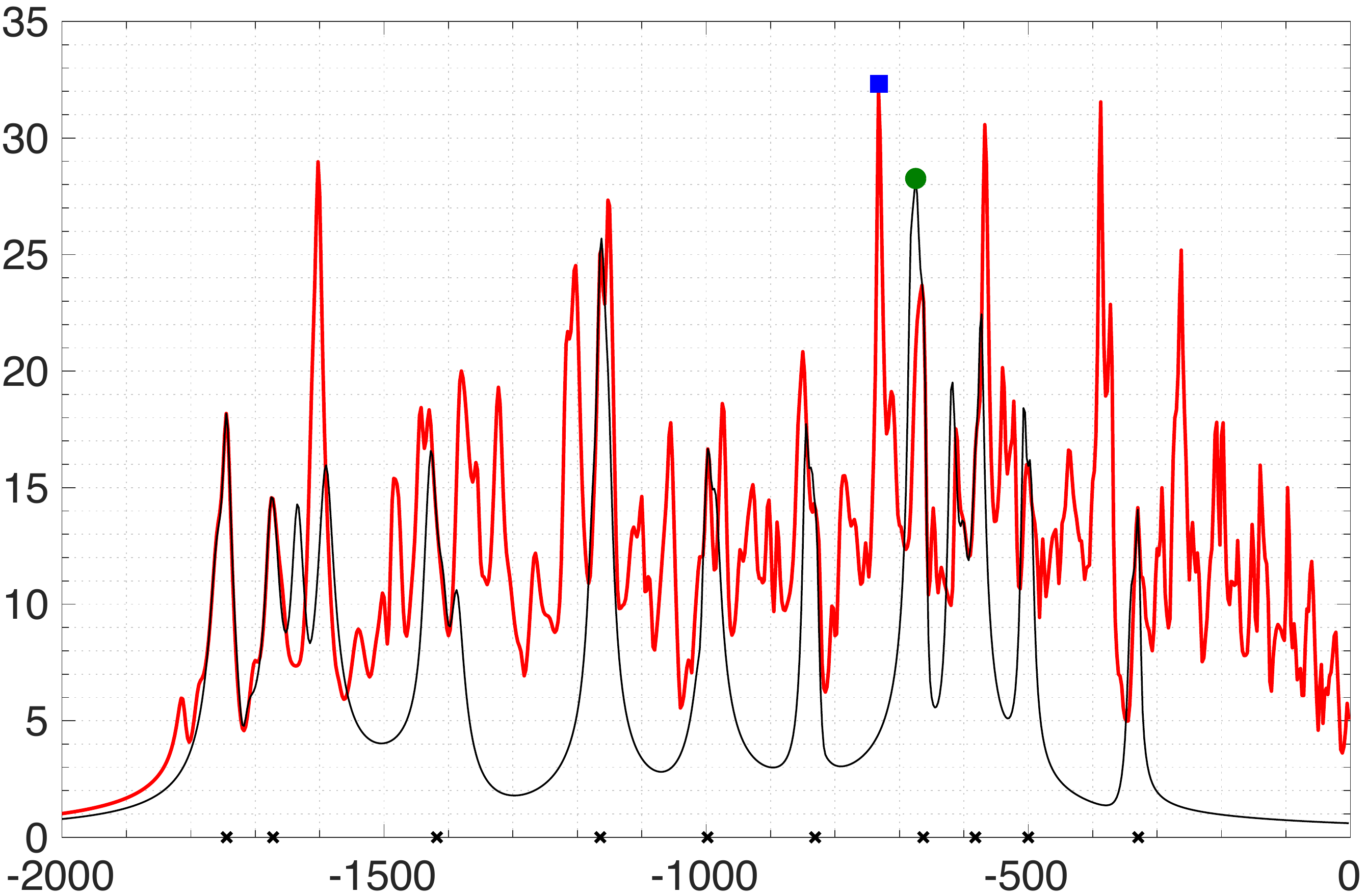} 	&	\includegraphics[width=5.6cm]{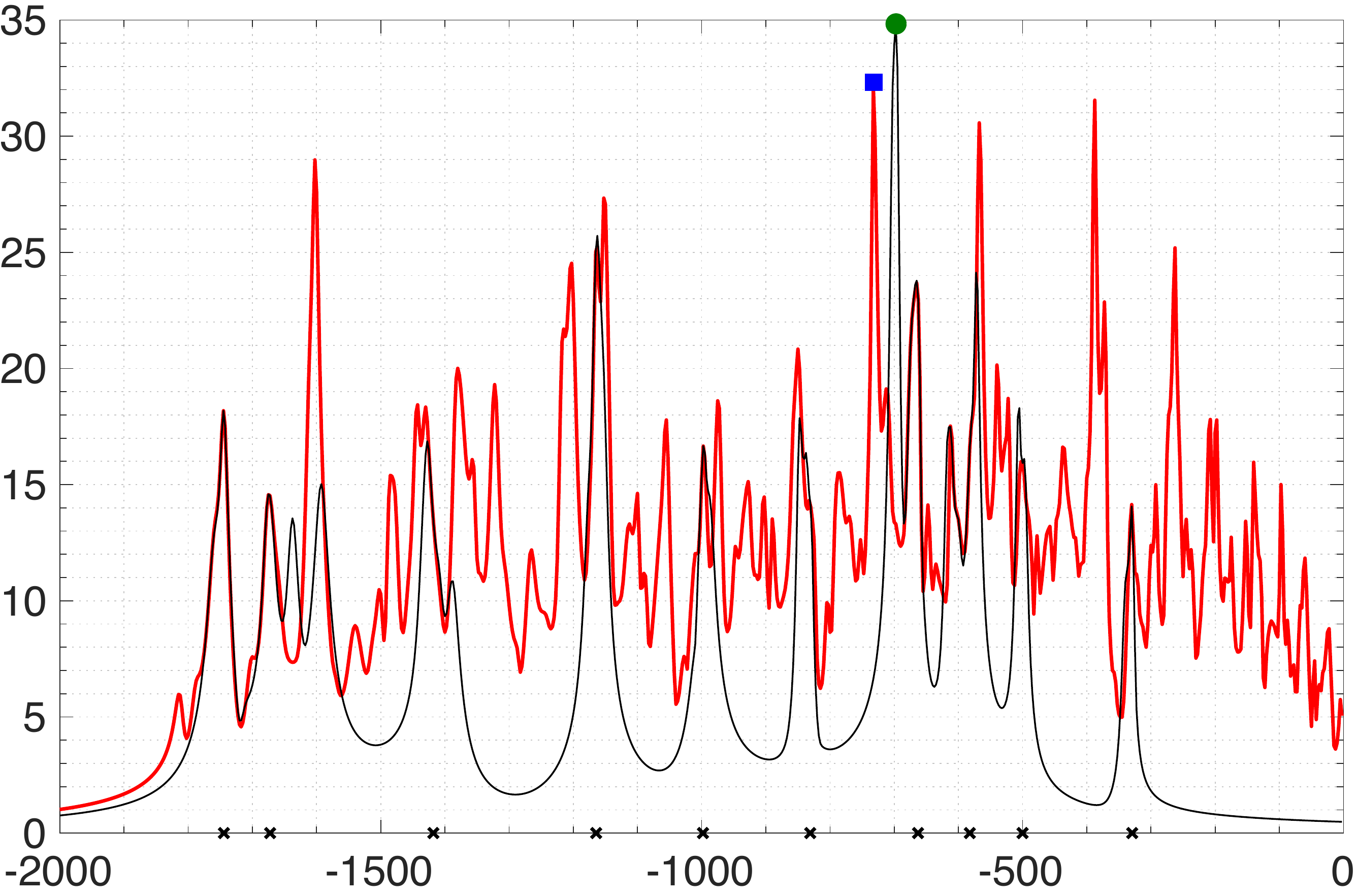}		\\
			\includegraphics[width=5.6cm]{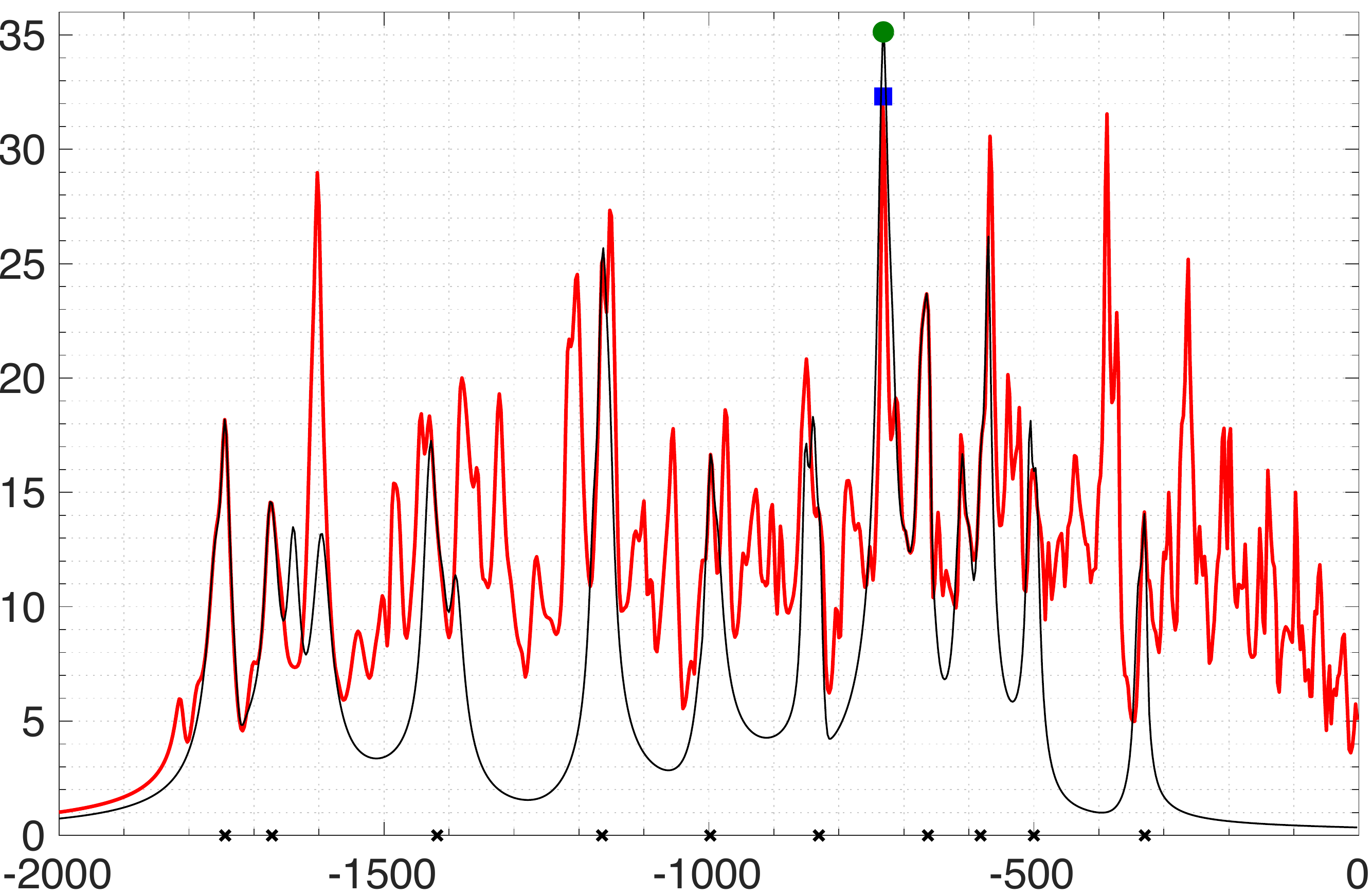} 	&	\includegraphics[width=5.6cm]{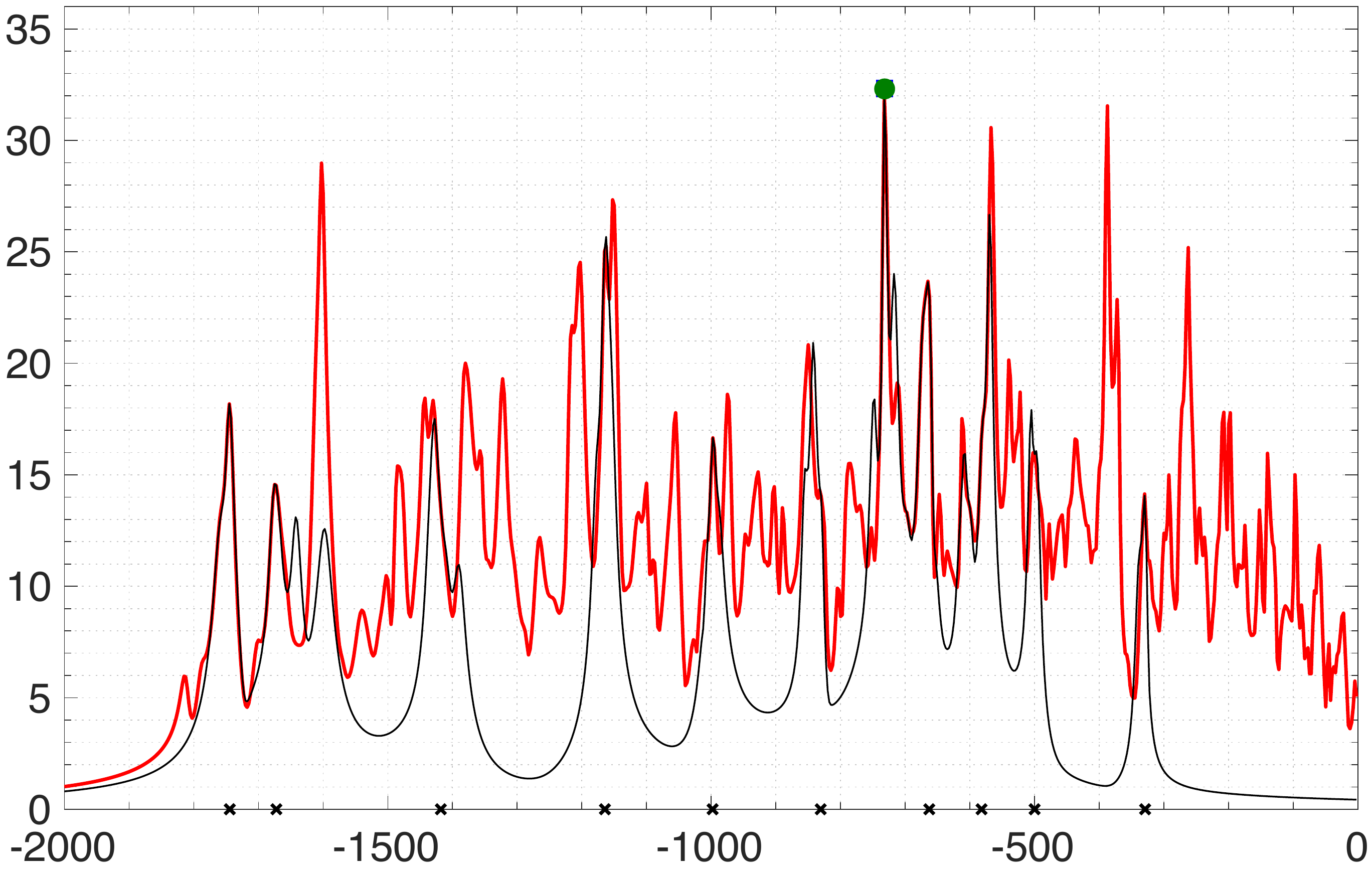}		\\
		\end{tabular}
		\caption{\textbf{(Top Left)} The plot of $\sigma_{\max}(C({\rm i} \omega - (J-R)Q)^{-1} (J-R)B)$ as a function of $\omega \in [-2000,0]$
		along with the maxima computed by Algorithm \ref{alg:dti2} and \cite[Algorithm 1]{Aliyev2017} marked with the square and
		circle, respectively, for a dense random example of order $800$. \textbf{(Top Right)} The black curve is the initial reduced function
		for Algorithm \ref{alg:dti2} interpolating the full function at $10$ points, whereas the circle is the global maximum of this
		reduced function. \textbf{(Middle Left - Bottom Right)} Plots of the reduced functions after iterations $1$-$4$ of Algorithm \ref{alg:dti2}
		displayed with black curves along with the maximizers of the reduced functions marked with circles.}
		\label{fig:progress_alg2}
	\end{figure}

A more decisive conclusion can be drawn when we consider all of the $100$ random examples.
The left-hand columns in Figure \ref{fig:dense_functions} depict the ratios
$(f_{\rm BB} - f_{\rm SF}) / ((f_{\rm BB} + f_{\rm SF})/2)$, where $f_{BB}$ are the globally
maximal values of $\sigma_{\max}(CQ ({\rm i} \omega I - (J-R)Q)^{-1} B)$ over $\omega$
returned by the BB algorithm and $f_{\rm SF}$ are the values returned by the
subspace frameworks, specifically by Algorithm \ref{alg:dti} on the top and by \cite[Algorithm 1]{Aliyev2017}
at the bottom. The results by Algorithm \ref{alg:dti} match with the ones by the BB algorithm
$81$ times out of $100$, while the results by Algorithm \cite[Algorithm 1]{Aliyev2017}
match with the ones by the BB algorithm $67$ times out of $100$.
(In the examples where the results by the subspace frameworks differ from those by
the BB algorithm, the subspace frameworks converge to
local maximizers that are not global maximizers.)

On these $100$ random examples Algorithm \ref{alg:dti} performs slightly fewer
iterations, on average $17.6$, whereas \cite[Algorithm 1]{Aliyev2017} on average performs $20.3$ iterations.
On the other hand, the total run-time on average is better for \cite[Algorithm 1]{Aliyev2017} compared
with Algorithm \ref{alg:dti} here,  $21.3\, s$  vs $30.9\, s$. We observe this behavior on
various other DH systems; Algorithm~\ref{alg:dti} seems to be more robust for the
computation of $r(J; B,C) = r(R; B,C)$ in converging to the globally maximal value
of $\sigma_{\max}(CQ ({\rm i}  \omega I - (J-R)Q)^{-1} B)$ compared with \cite[Algorithm 1]{Aliyev2017}, however, this is at the expense of slightly more computation time.

On the other hand, for the computation of $r(Q; B, C)$, Table \ref{table:dense_comparisonQ} indicates
that Algorithm \ref{alg:dti2} returns exactly the same globally maximal values (up to tolerances)
as the BB algorithm for all of the first $10$ examples except one, whereas
application of \cite[Algorithm 1]{Aliyev2017} results in locally maximal solutions
that are not globally maximal $4$ times. Fewer number of subspace iterations
in favor of Algorithm \ref{alg:dti2} are also apparent from the table. Once again,
the plots of the ratios $(f_{\rm BB} - f_{\rm SF}) / ((f_{\rm BB} + f_{\rm SF})/2)$ are
shown in Figure \ref{fig:dense_functions} on the right-hand column for all $100$ examples
with $f_{\rm SF}$ now representing the values returned by Algorithm \ref{alg:dti2} on the top
and by \cite[Algorithm 1]{Aliyev2017} at the bottom. Algorithm \ref{alg:dti2} and \cite[Algorithm 1]{Aliyev2017}
return locally optimal solutions that are not globally optimal $21$ and $27$ times,
respectively. In this case the difference between the number of subspace iterations
for these $100$ examples is more pronounced in favor of Algorithm \ref{alg:dti2};
indeed the number of subspace iterations is on average $7.2$  for Algorithm \ref{alg:dti2} and and $17.0$ for
\cite[Algorithm 1]{Aliyev2017}. This difference in the number of iterations is also reflected
in the average run-times which are $13.2\, s$ and $19\, s$ for Algorithm \ref{alg:dti2}
and \cite[Algorithm 1]{Aliyev2017}, respectively.

\begin{table}
\begin{center}
\begin{tabular}{|c||c|c|}

\hline
  $k$  & $\omega_{k+1}$ &  $\sigma_{\max}(G_k(\omega_{k+1}))$ \\ [0.5ex]
\hline 
\hline
10  & -600.705819      &    26.182525 \\
11  & -674.769938      &     28.262865 \\
12  & -697.139310      &    34.834307 \\
13  &  -731.573363      &    35.133647 \\
14  &  -731.942586	   &	 32.309246 \\
15  &  -731.977386	  &	32.321399 \\
16  &  -731.977385 	&	32.321399 \\

\hline

\end{tabular}
\end{center}
\caption{ Iterates of Algorithm~\ref{alg:dti2} to compute $r(Q; B, C)$
on a DH system with dense random $J, R, Q \in {\mathbb R}^{800\times 800}$ and
random restriction matrices $B \in {\mathbb R}^{800\times 2}, C \in {\mathbb R}^{2\times 800}$.
The algorithm is initiated with $10$ interpolation points and terminates after $6$
iterations with $32$ dimensional subspaces. }
\label{table:dense_iterates_alg2}
\end{table}

\begin{table}
\begin{center}
\begin{tabular}{|c||ccc|cc|cc|}
\hline
  &  \multicolumn{3}{c}{$\max_{\omega} \sigma_{\max}(CQ ({\rm i} \omega I - (J-R)Q)^{-1} B)$}  &
  			\multicolumn{2}{|c|}{$\#$ iterations}	&
			\multicolumn{2}{|c|}{run-time} \\
 [0.5ex]
   Ex. &  Alg.~\ref{alg:dti} & \cite[Alg. 1]{Aliyev2017}  & BB Alg. \cite{Boyd1990}
  			&  Alg.~\ref{alg:dti} & \cite{Aliyev2017}
					&	Alg.~\ref{alg:dti} & \cite{Aliyev2017}	\\ [0.5ex]  \hline	\hline

1     &   32.559659     &      32.559659   &     32.559659     & 	9	&	9	&	14	&	13	\\

2       & 46.703932     &      46.703932     &   46.703932    &	15	&	12	&	16	&	12.4	\\

3      &  26.227029     &      \textbf{24.023572}  &    26.227029 	&	7	&	12	&	12.9	  &	15.1	\\

4   	&  \textbf{62.748090}      &     108.030409       &  108.030409 &	 17	&	41	&	17.8	 &	27.2	\\

5   	&  35.974956      &	 35.974957      &   35.974956 	&	9	&	14	&	13.4		&	13.9	\\

6   	&  53.522033       &   53.522033     &   53.522033 	&	6	&	3	&	11.4		&	10.2	\\

7    	&  31.739000      &     31.739000   &    31.739000 	&	4	&	12	&	11.6		&	13.8	\\

8     	&  76.958658      &   76.958658     &   76.958658 	&	35	&	8	&	43.2		&	11	\\

9     	&  37.007241 	&   37.007241    &   37.007241  &	6	&	2	&	13	&	11.5		\\

10    &  155.642871      &   155.642871     &   155.642871  &	2	&	2	&	8.6	&	8.5	\\

\hline

\end{tabular}
\end{center}
\caption{Run-time (in $s$) comparison of Algorithm \ref{alg:dti} and
		\cite[Algorithm 1]{Aliyev2017}
		to compute $r(R; B, C) = r(J; B, C)$ for $10$ dense random examples
		of order $800$. The third column  refers to the number of
		subspace iterations.}
\label{table:dense_comparison}
\end{table}

\begin{table}
\hskip -1.3ex
\begin{tabular}{|c||ccc|cc|cc|}
\hline
  &  \multicolumn{3}{c}{$\max \: \sigma_{\max}(C({\rm i} \omega I - (J-R)Q)^{-1} (J-R)B)$}  &
  			\multicolumn{2}{|c|}{$\#$ iterations}	&
			\multicolumn{2}{|c|}{run-time} \\
 [0.5ex]
  $\#$ &  Alg.~\ref{alg:dti} & \cite[Alg. 1]{Aliyev2017}  & BB Alg. \cite{Boyd1990}
  			&  Alg.~\ref{alg:dti} & \cite{Aliyev2017}
					&	Alg.~\ref{alg:dti} & \cite{Aliyev2017}	\\ [0.5ex]  \hline	\hline
					
1	&	9.809182			&	9.809182		&	9.809182		&	3	&	26	&	11.1	&	19.5	\\

2	&	22.386670		&	22.386670	&	22.386670	&	5	&	26	&	10.2	&	18.1	\\

3	&	8.364927			&	8.364927		&	8.364927		&	3	&	8	&	11	&	13.3	\\

4	&	32.321399		&	\textbf{29.028197}	&	32.321399	&	6	&	37	&	10.5	&	25.2	\\

5	&	15.071678		&	15.071678	&	15.071678	&	7	&	15	&	12.7	&	14	\\

6	&	21.641484		&	21.641484	&	21.641484	&	4	&	8	&	10.2	&	11.6	\\

7	&	12.858494		&	\textbf{12.763161}	&	12.858494	&	6	&	4	&	11.8	&	11.2	\\

8	&	31.901305		&	\textbf{27.996873}	&	31.901305	&	8	&	12	&	11.5	&	12.5	\\

9	&	9.228945			&	9.228945		&	9.228945		&	3	&	8	&	11.3	&	13	\\

10	&	\textbf{47.697528}	&	\textbf{47.697528}	&	71.534252	&	10	&	8	&	11.8	&	10.2	\\
					
\hline

\end{tabular}
\caption{Run-time (in $s$) comparison of Algorithm \ref{alg:dti2} and
		\cite[Algorithm 1]{Aliyev2017} for the computation of $r(Q; B, C)$ on  $10$ dense random examples
		of order $800$.}
\label{table:dense_comparisonQ}
\end{table}

\begin{figure}[ht]
		\hskip -3ex
		\begin{tabular}{cc}
			\includegraphics[width=6cm]{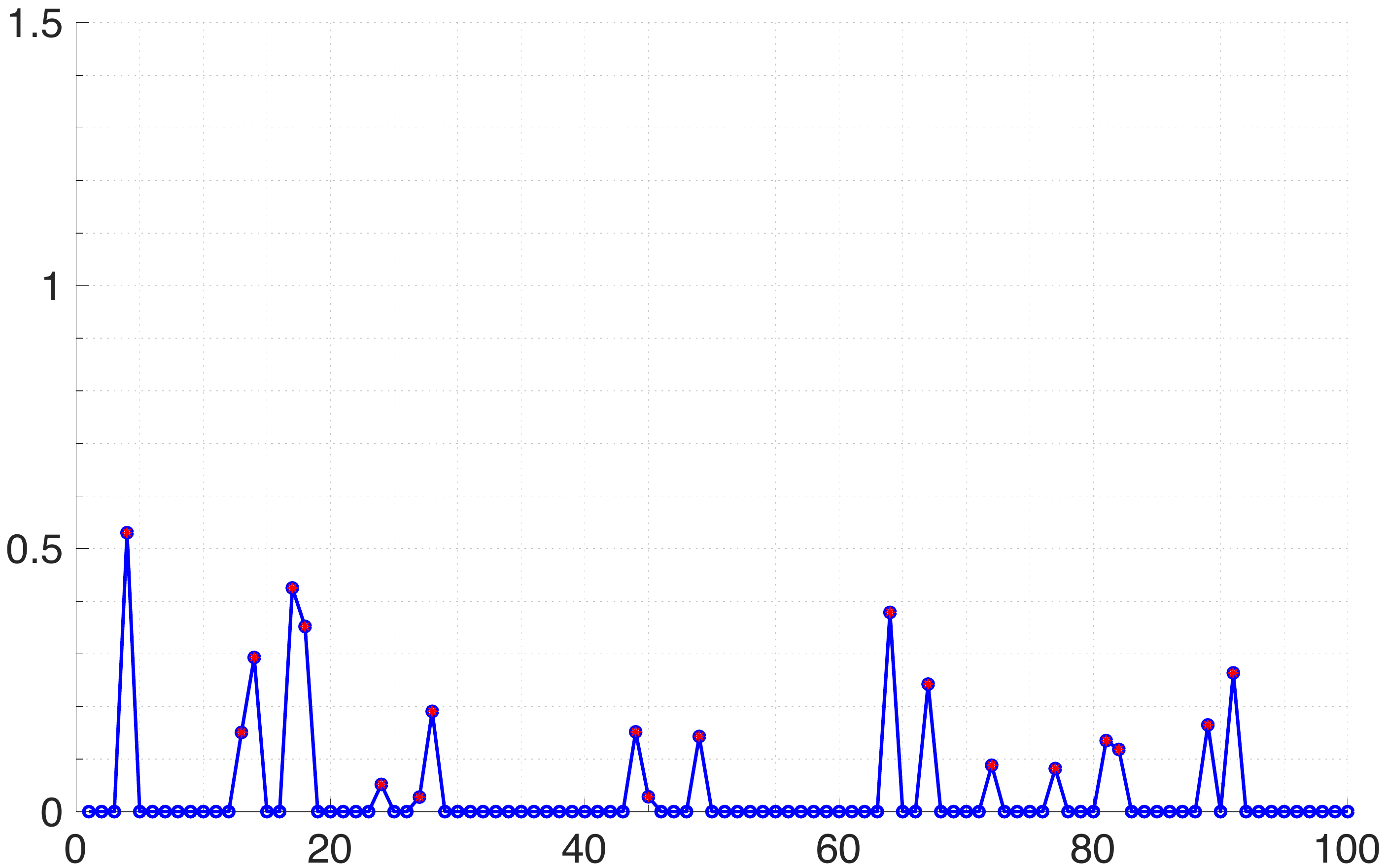} 	&	\includegraphics[width=6cm]{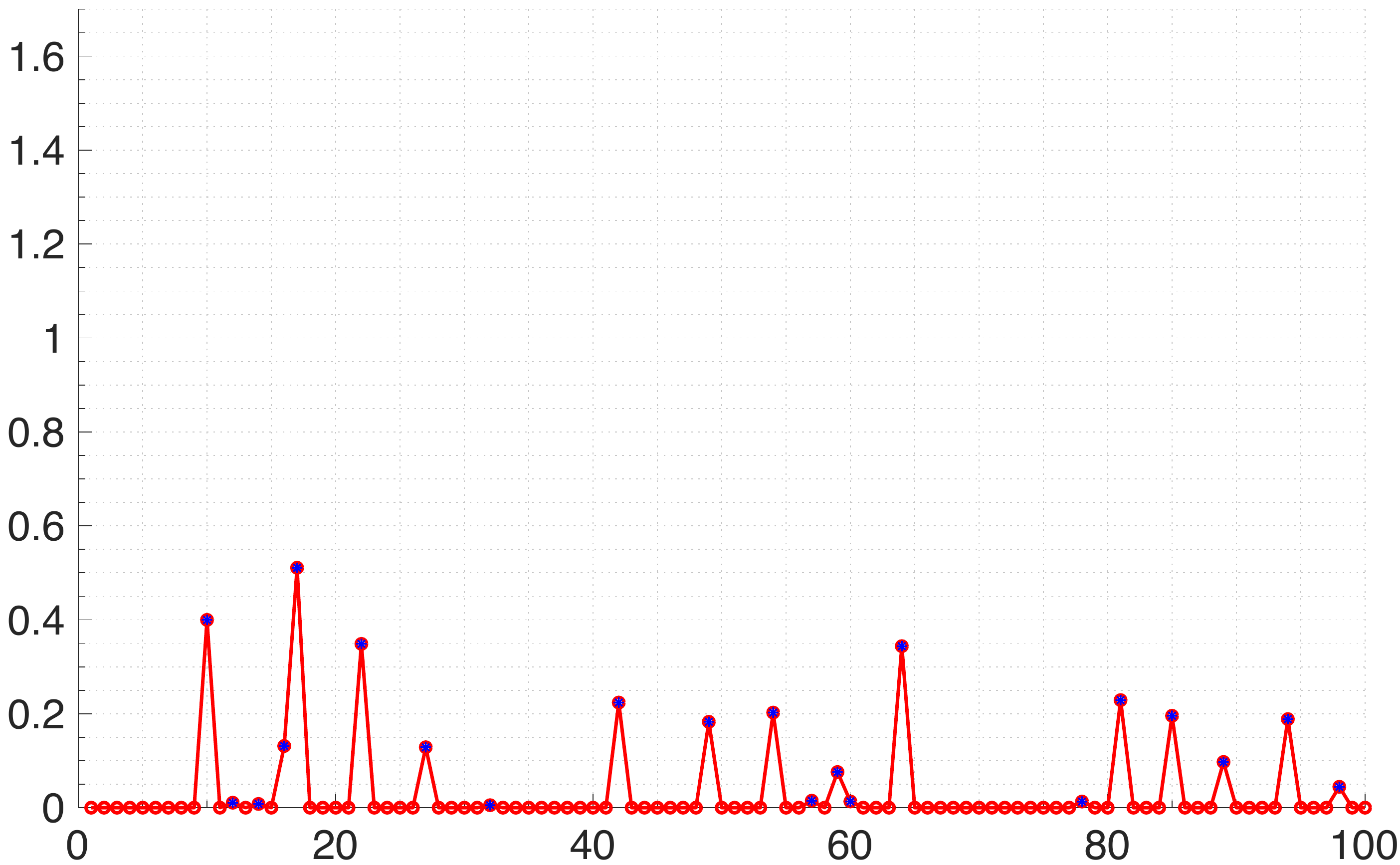} 	\\
			\includegraphics[width=6cm]{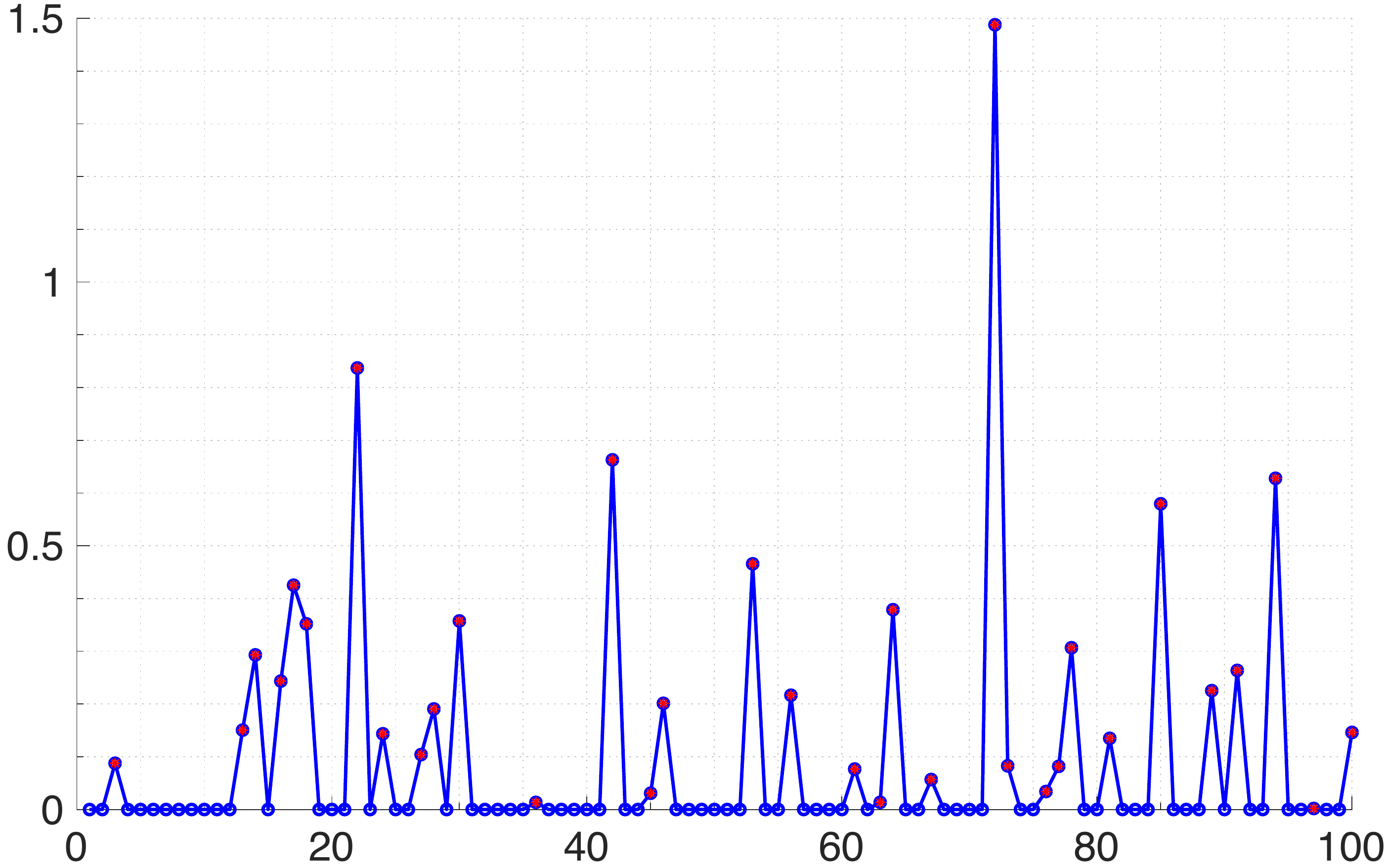}	&	 \includegraphics[width=6cm]{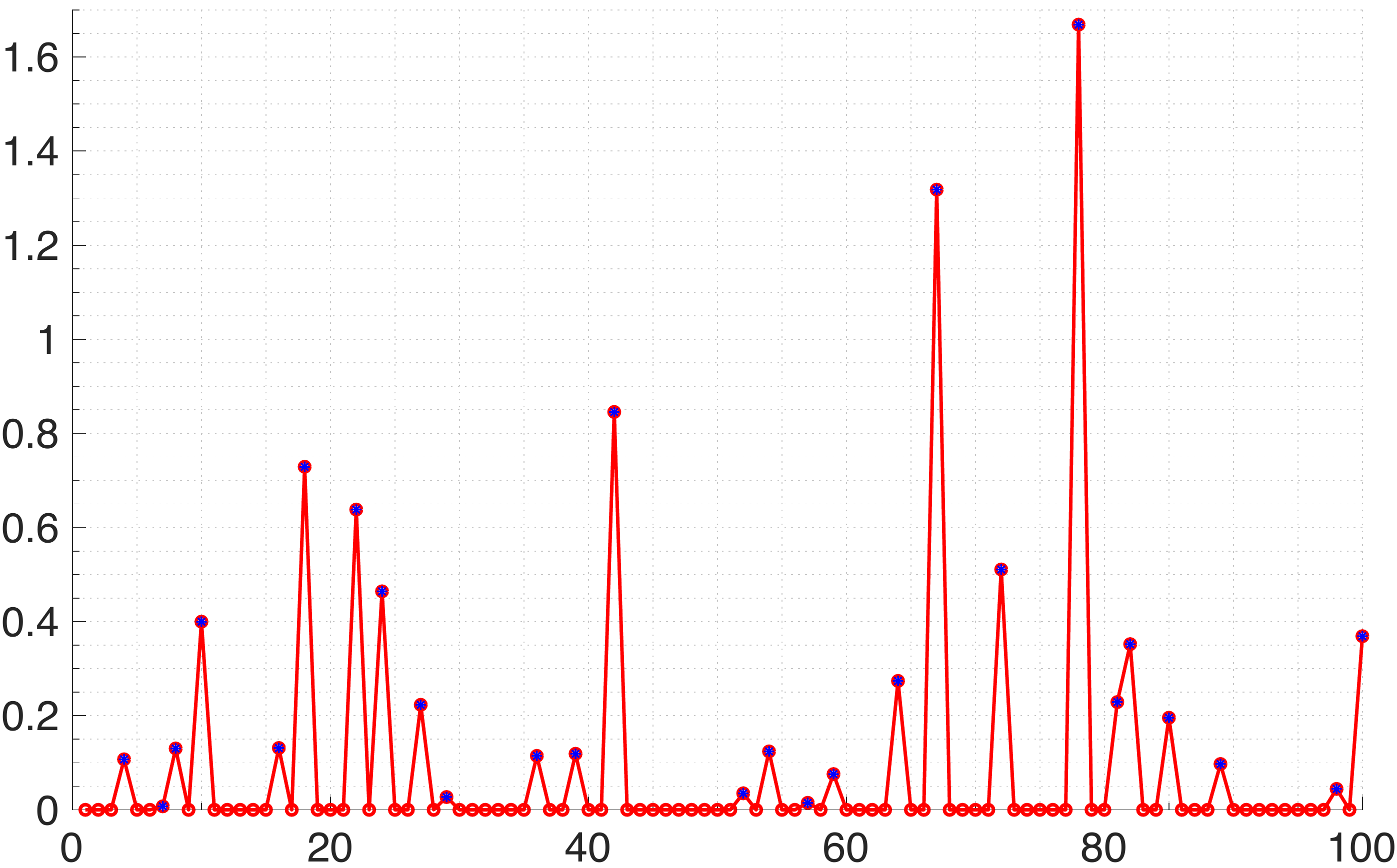}	 \\
		\end{tabular}
		\caption{\textbf{(Left Column)} Ratios $(f_{BB} - f_{SF}) / ((f_{BB} + f_{SF})/2)$ for $100$ dense random DH examples
		of order $800$, where $f_{BB}$ and $f_{SF}$ denote the maximal values of 
		$\sigma_{\max}(CQ ({\rm i} \omega I - (J-R)Q)^{-1}B)$
		over $\omega$ computed by the BB algorithm and the subspace framework (i.e., Algorithm \ref{alg:dti}
		for the plot on the top, \cite[Algorithm 1]{Aliyev2017} for the plot at the bottom).
		\textbf{(Right Column)} Same as the left column on the same $100$ dense random examples of order $800$ 		except now this concerns a comparison of the maximization of $\sigma_{\max}(C ({\rm i} \omega I - (J-R)Q)^{-1}(J-R)B)$ using Algorithm \ref{alg:dti2} in the top plot and \cite[Algorithm 1]{Aliyev2017} in  the bottom plot.}
		\label{fig:dense_functions}
\end{figure}

\medskip

\noindent
\textbf{Sparse Random Examples.}
The $5000\times 5000$ sparse matrices $J, Q, R$ are constrained to be banded with bandwidth $10$. The matrix $J$ is generated as in the dense family randomly
using the \texttt{randn} command, but the entries that fall outside of the bandwidth $10$ are set equal to zero. The matrix $Q>0$ is created using the commands
\begin{verbatim}
>> A = sprandn(n,n,1/n);  >> Q = (A + A')/2,
\end{verbatim}
followed by setting the entries outside the bandwidth $10$ again to zero. Finally, the
following commands ensure that $Q>0$.
\begin{verbatim}
>> mineig = eigs(Q,1,'smallestreal');
>> if (mineig<10^-4) Q=Q+(-mineig+5*rand)*speye(n); end
\end{verbatim}
To form $R\geq 0$, first a diagonal matrix $D$ of random rank
not exceeding $500$ is generated by the commands
\begin{verbatim}
>> p = round(500*rand);   D = sparse(5000,5000);   h = n/p;
>> for j=1:p  k = floor(j*h); D(k,k) = 5*rand;  end
\end{verbatim}
Then we set \texttt{R = sparse(X'*D*X)} for a square random matrix $X$ with bandwidth $5$.
The matrices $B$, $C$ are random, and of size $5000 \times 2$, $2 \times 5000$, respectively.
The spectrum of a typical such sparse  matrix $(J-R)Q$ is displayed
in Figure \ref{fig:spectra_randomPH} on the right-hand side.

We again apply Algorithms \ref{alg:dti} and \ref{alg:dti2} to $100$ such random sparse examples.
Since  the matrices are too large to apply the BB algorithm, we
compare the structure-preserving algorithms directly with the unstructured algorithm
\cite[Algorithm 1]{Aliyev2017}. The retrieved estimates of
$\max_{\omega \in {\mathbb R} \cup \infty} \sigma_{\max}(G({\rm i} \omega))$ for
$G({\rm i}\omega) = CQ (\ri \omega I - (J-R)Q)^{-1}B$ and $G({\rm i}\omega) = C (\ri \omega I - (J-R)Q)^{-1}(J-R)B$
are compared on the top and at the bottom, respectively, in Figure \ref{fig:sparse_functions}.
Specifically, the ratio
$
		\frac{2 (f_{\rm ST} - f_{\rm UN})}{f_{\rm ST} + f_{\rm UN}}
$
is plotted for each random example with $f_{\rm UN}$ denoting the estimate by the unstructured
algorithm \cite[Algorithm 1]{Aliyev2017}, and $f_{\rm ST}$ denoting the estimate by the structured
algorithm, i.e., Algorithm \ref{alg:dti} for the top plot, Algorithm \ref{alg:dti2} for the bottom plot.
According to the top plot, which concerns the computation of $r(J; B, C) = R(R; B, C)$, the two
algorithms return exactly the same results (up to tolerances) for all but $6$ examples;
the structured algorithm returns better estimates for $4$ of these $6$ examples, while
the unstructured algorithm returns better estimates for the other two. The structured algorithm
appears to be even more robust for the computation of $r(Q; B, C)$ in terms of avoiding locally
optimal solutions away from global solutions; as displayed at the
bottom, the structured algorithm returns a better estimate for $40$ of the $100$ examples,
the unstructured algorithm returns the better estimate for $6$ examples, and the results
match exactly up to the tolerances for the remaining $54$ examples.

The structured algorithms perform typically fewer iterations as compared to the
unstructured algorithm. Indeed the average value of the number of subspace iterations performed on
these $100$ examples is $6.3$  for the structured and $9.8$ for the unstructured algorithm for the computation of $r(J; B,C) = r(R; B, C)$, while these average values are $9.7$ and $12.9$ for the computation
of $r(Q; B, C)$. On the other hand, the unstructured algorithm is slightly superior when run-times are taken into account. The average run-times are $14.3\, s$ for the structured and
 $12.7\, s$ for the unstructured algorithm for the computation of $r(J; B, C) = r(R; B, C)$, whereas these figures are
$16.5\, s$ and $14\, s$ for the computation of $r(Q; B, C)$.

We also list the computed maximal values of $\sigma_{\max}(CQ({\rm i} \omega I - (J-R)Q)^{-1}B)$
and $\sigma_{\max}(C({\rm i} \omega I - (J-R)Q)^{-1}(J-R)B)$ over $\omega$
for the first $10$ of these sparse examples in Tables~\ref{table:sparse_iterates_alg} and
\ref{table:sparse_iterates_algQ}. Included in these tables are also the number of subspace
iterations, as well as the run-time required by the structured and the unstructured algorithm.

\begin{figure}
\begin{center}
		\begin{tabular}{cc}
			\includegraphics[width=8cm]{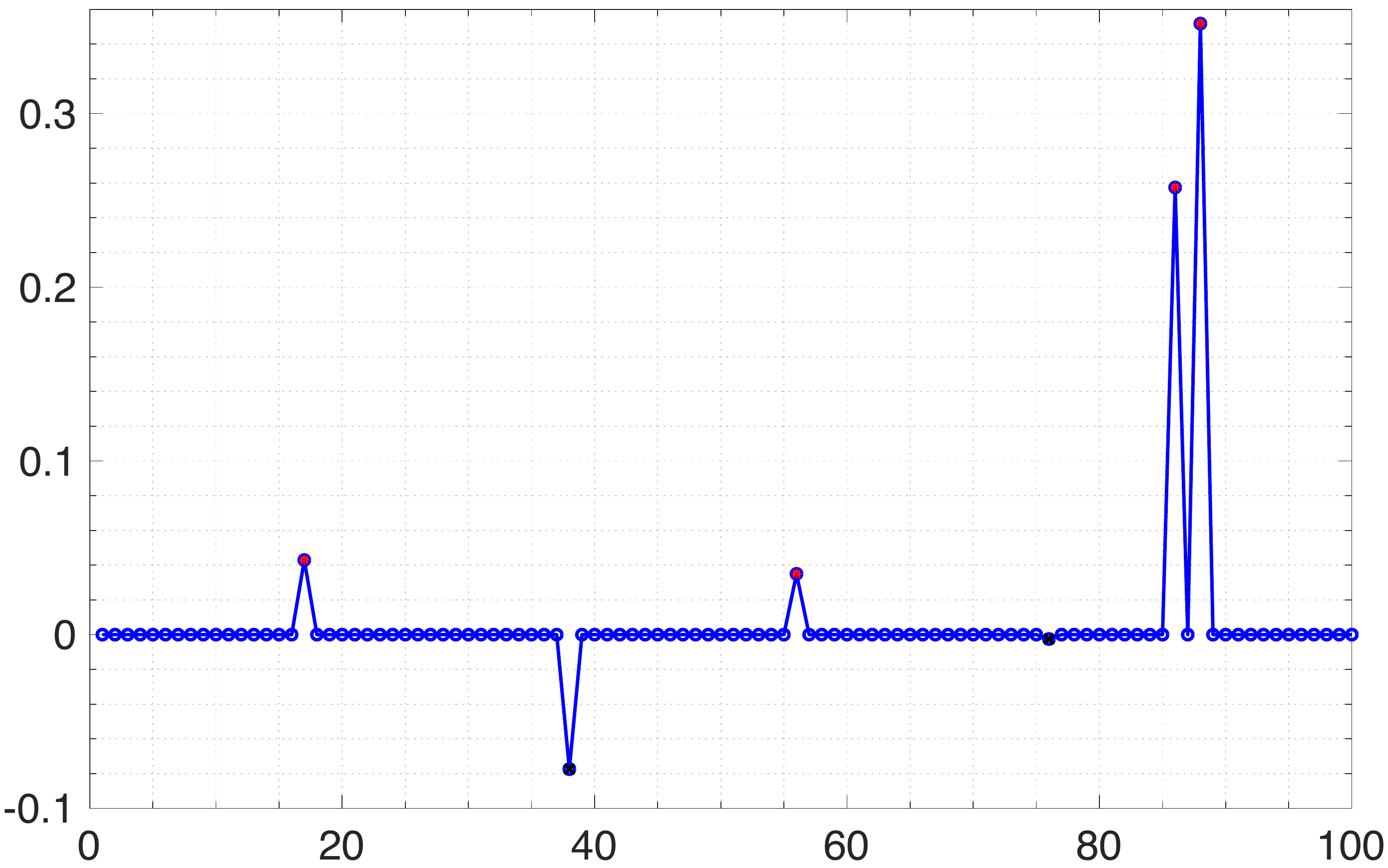} 	\\
			 \includegraphics[width=8cm]{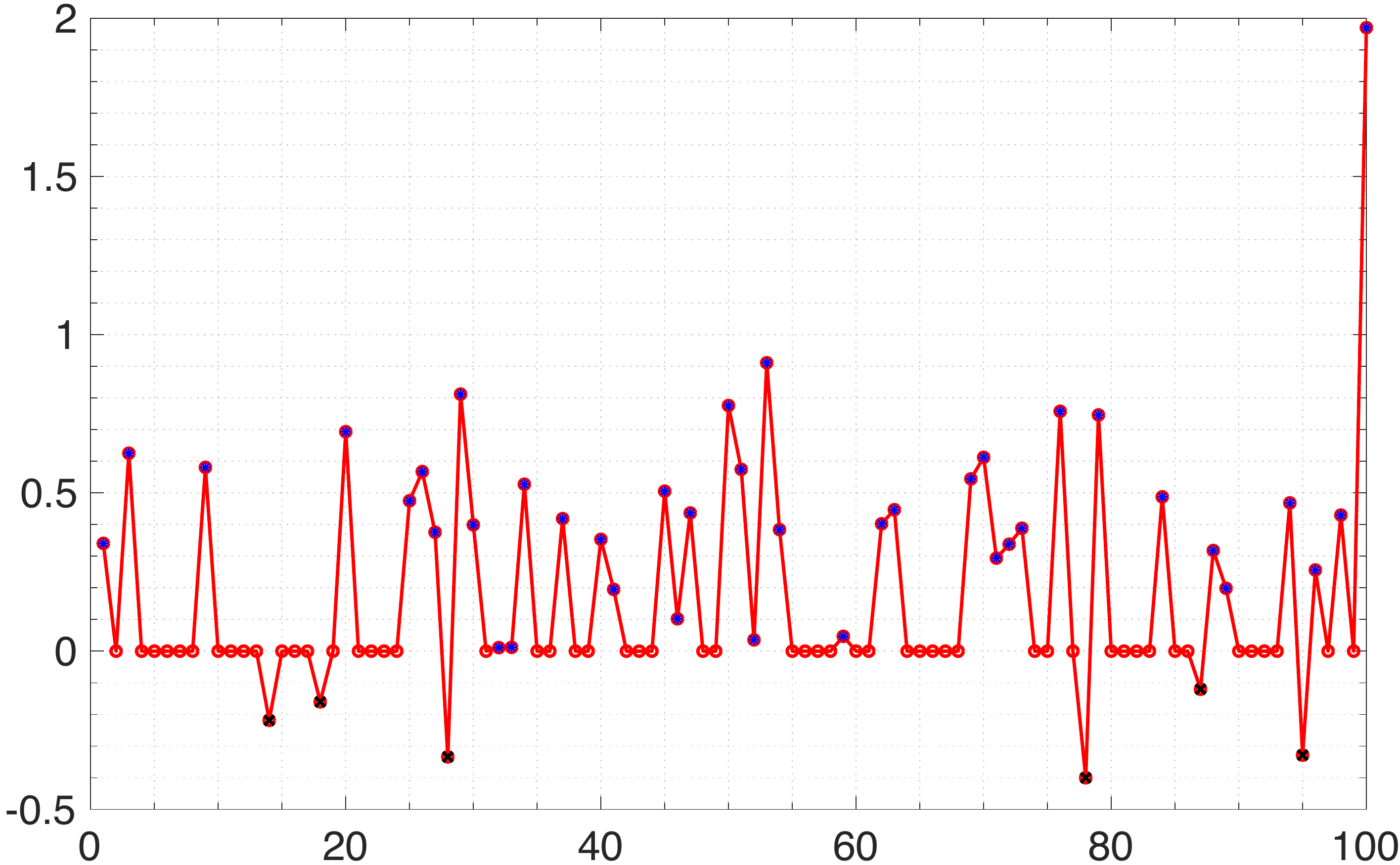}	 \\
		\end{tabular}
\end{center}
		\caption{ \textbf{(Top)} Plot of the ratios $(f_{\rm ST} - f_{\rm UN})/((f_{\rm ST} + f_{\rm UN})/2)$
		on $100$ sparse random examples of order $5000$, where $f_{\rm ST}, f_{\rm UN}$ represent
		the computed maximal value of $\sigma_{\max}(CQ({\rm i} \omega I - (J-R)Q)^{-1} B)$
		over $\omega$ by Algorithm \ref{alg:dti} and \cite[Algorithm 1]{Aliyev2017}, respectively.
		\textbf{(Bottom)} Similar to the top plot, only now for the computed maximal values $f_{\rm ST}$,
		$f_{\rm UN}$ of $\sigma_{\max}(C({\rm i} \omega I - (J-R)Q)^{-1} (J-R)B)$ over $\omega$
		by Algorithm \ref{alg:dti2} and \cite[Algorithm 1]{Aliyev2017}, respectively. }
		\label{fig:sparse_functions}
\end{figure}

\begin{table}
\begin{center}
\begin{tabular}{|c||cc|cc|cc|}
\hline
  &  \multicolumn{2}{c}{$\max_{\omega} \sigma_{\max}(CQ({\rm i} \omega I - (J-R)Q)^{-1} B)$}  &
  			\multicolumn{2}{|c|}{$\#$ iterations}	&
			\multicolumn{2}{|c|}{run-time} \\
 [0.5ex]
  $\#$ &  Alg.~\ref{alg:dti} & \cite[Alg. 1]{Aliyev2017}
  			&  Alg.~\ref{alg:dti} & \cite{Aliyev2017}
					&	Alg.~\ref{alg:dti} & \cite{Aliyev2017}	\\ [0.5ex]  \hline	\hline

1     &   $1.393086\times 10^3$    &     $1.393086\times 10^3$       & 	2	&	3	&	10.7	&	10.1	\\

2      & $8.323309\times 10^2$     &      $8.323309\times 10^2$       &	 6	&	11	&	12.3	&	11.9	\\

3      &  $1.289416 \times 10^3$     &      $1.289416 \times 10^3$   	&	4	&	3	&	11.2	  &	10.6	\\

4   	&  $8.850355 \times 10^2$     &     $8.850355 \times 10^2$      &	 5	&	16	&	13.2	 &	14	\\

5   	&  $6.891467 \times 10^2$      &  $6.891467 \times 10^2$      &   	26	&	46	&	30.3		&	31.4	\\

6   	&  $5.652337\times 10^6$      &   $5.652337\times 10^6$     &  	1	&	1	&	6.6		&	6.1	\\

7    	&  $8.834190\times 10^2$      &     $8.834190\times 10^2$   &    1	&	1	&	9.7		&	9.6	\\

8     	&  $3.402375\times 10^3$     &   $3.402375\times 10^3$     &   	1	&	2	&	9.1		&	9	\\

9     	&  $8.240097\times 10^2$ 	&   $8.240097\times 10^2$    &   	2	&	3	&	12.3	&	12.1		\\

10    &  $1.256781\times 10^3$      &   $1.256781\times 10^3$     &   	2	&	10	&	10.9	&	11.3	\\

\hline

\end{tabular}
\end{center}
\caption{Comparison of Algorithms \ref{alg:dti}  and \cite[Algorithm 1]{Aliyev2017} to compute
$r(R; B,C) = r(J; B, C)$ on sparse random DH systems of order $5000$
 of bandwidth $10$. The MATLAB commands to generate these random
sparse examples are explained in Section \ref{sec:numexp_syn}. The third column lists the number
of subspace iterations, while the run-times (in $s$) are listed in the last column.}
\label{table:sparse_iterates_alg}
\end{table}

\begin{table}
\hskip -2ex
\begin{tabular}{|c||cc|cc|cc|}
\hline
  &  \multicolumn{2}{c}{$\max_{\omega} \sigma_{\max}(C({\rm i} \omega I - (J-R)Q)^{-1} (J-R)B)$}  &
  			\multicolumn{2}{|c|}{$\#$ iterations}	&
			\multicolumn{2}{|c|}{run-time} \\
 [0.5ex]
  $\#$ &  Alg.~\ref{alg:dti2} & \cite[Alg. 1]{Aliyev2017}
  			&  Alg.~\ref{alg:dti2} & \cite{Aliyev2017}
					&	Alg.~\ref{alg:dti2} & \cite{Aliyev2017}	\\ [0.5ex]  \hline	\hline

1     &     $1.320177\times 10^3$  &     $\mathbf{9.360406\times 10^2}$     & 	15	&	39	&	18.3	&	26.7	\\

2       & 	$8.284036\times 10^2$     &     $8.284036\times 10^2$       &	2	&	13	&	11.1	&	12.7	\\

3      &  $9.288427\times 10^2$    &     $\mathbf{4.863413 \times 10^2}$  	&	31	&	11	&	46.1	  &	11.8	\\

4   	&  $6.583171\times 10^2$     &     $6.583171 \times 10^2$      &	 5	&	17	&	13.1	 &	14.6	\\

5   	&   $7.722736\times 10^2$     &	  $7.722736\times 10^2$     &   	3	&	9	&	11.3		&	11.4	\\

6   	&  $2.260647\times 10^6$       &   $2.260647\times 10^6$     &  	1	&	1	&	6.4		&	6.1	\\

7    	&  $1.660164\times 10^3$     &     $1.660164\times 10^3$   &    	4	&	12	&	10.6		&	11.1	\\

8     	&  $1.515824\times 10^3$      &     $1.515824\times 10^3$   &   	8	&	9	&	11.8		&	10	\\

9     	&  $4.610935\times 10^2$ 	&    $\mathbf{2.535771\times 10^2}$  &   13	&	4	&	17.9	&	12.1		\\

10    &  $1.011299\times 10^3$      &   $1.011299\times 10^3$      &   	3	&	7	&	10.8	&	10.6	\\

\hline

\end{tabular}
\caption{Comparison of Algorithms \ref{alg:dti2} and \cite[Algorithm 1]{Aliyev2017} to
compute $r(Q; B, C)$. The display is analogous to Table \ref{table:sparse_iterates_alg}, in particular the numerical experiments
are carried out exactly on the same $10$ sparse random examples of order $5000$ employed for Table \ref{table:sparse_iterates_alg}. }
\label{table:sparse_iterates_algQ}
\end{table}

\subsubsection{The FE model of a Disk Brake}\label{sec:numexp_brake}
The only large-scale computation required by Algorithm \ref{alg:dti} is the solution of the linear systems
\begin{equation}\label{eq:large_lin_sys}
	D({\rm i} \omega) \: X = B	\quad	{\rm and}	\quad	D({\rm i} \omega)^2 \: Y = B
\end{equation}
at a given $\omega \in {\mathbb R}$ in lines \ref{defn_init_subspaces0} and \ref{defn_later_subspaces},
where $D({\rm i} \omega) = {\rm i} \omega I  - (J - R)Q$. For the DH system resulting from a FE model of a disk brake in (\ref{insta}) and (\ref{eq:insta_matrices}), the mass matrix $M$ and the stiffness matrix
$K(\Omega)$ are available from the FE modeling. In other words, we have the sparse matrix $Q^{-1}$, but not $Q$, which turns out to be dense. Trying to invert $Q^{-1}$ and/or solve a linear system
with the coefficient matrix $Q$  is computationally very expensive and would require full matrix storage.

This difficulty can be avoided by exploiting that
\[
	( {\rm i} \omega I  - (J - R)Q )^{-1} B	\;\;\;	=	\;\;\;	Q^{-1} ( {\rm i} \omega Q^{-1}  - (J - R) )^{-1} B.
\]
Hence, to compute $X, Y$ as in (\ref{eq:large_lin_sys}), we proceed as follows.
\begin{enumerate}
	\item[\bf (1)] We first solve $( {\rm i} \omega Q^{-1}  - (J - R) ) \widehat{X} = B$ for $\widehat{X}$,
	and set $X = Q^{-1} \widehat{X}$.
	\item[\bf (2)] Then we solve $( {\rm i} \omega Q^{-1}  - (J - R) ) \widehat{Y} = X$ for $\widehat{Y}$,
	and set $Y = Q^{-1} \widehat{Y}$.
\end{enumerate}
A second observation that further speeds up the computation
is the particular structure of the coefficient matrix $\{ {\rm i} \omega Q^{-1}  - (J - R) \}$ 
with $N = 0$. Setting $\widetilde{M}({\rm i} \omega; \Omega) := {\rm i} \omega	M +  D(\Omega) + G(\Omega)$, we have
\[
	{\rm i} \omega Q^{-1}  - (J - R)	\;\;	=	\;\;		
			\left[
				\begin{array}{cc}
					\widetilde{M}({\rm i} \omega; \Omega)		&	K(\Omega)	\\
					-K(\Omega)							&		{\rm i} \omega K(\Omega)
				\end{array}
			\right].
\]
Hence, to solve $\left\{ {\rm i} \omega Q^{-1}  - (J - R) \right\} Z = W$, for a given
$W
	=
	\left[
		\begin{array}{cc}
			W_1^T		&		W_2^T
		\end{array}
	\right]^T$
and the unknown
$Z
	=
	\left[
		\begin{array}{cc}
			Z_1^T		&		Z_2^T
		\end{array}
	\right]^T$
with $W_1, W_2, Z_1, Z_2$ having all equal number of rows, we perform a column block permutation and then eliminate the lower left block to obtain
\[
	\left[
	\begin{array}{cc}
		K(\Omega)		&		\widetilde{M}({\rm i} \omega; \Omega)	\\
			0		&		-K(\Omega) - {\rm i} \omega \widetilde{M}({\rm i} \omega; \Omega)
	\end{array}
	\right]
	\left[
	\begin{array}{c}
		Z_2	\\
		Z_1
	\end{array}
	\right]
			\;\;	=	\;\;
	\left[
	\begin{array}{c}
		W_1	\\
		W_2 - {\rm i} \omega W_1
	\end{array}
	\right],	
\]
which in turn yields
\[
	(-K(\Omega) - {\rm i} \omega \widetilde{M}({\rm i} \omega; \Omega) ) Z_1	=	W_2 - {\rm i} \omega W_1,\
	K(\Omega) Z_2 = W_1 - \widetilde{M}({\rm i} \omega; \Omega) Z_1.
\]
At every subspace iteration, the highest costs arise from the computation of the $LU$ factorizations of the sparse matrices $K(\Omega)$ and
$K(\Omega) + {\rm i} \omega \widetilde{M}({\rm i} \omega; \Omega)$.

The main cost for Algorithm \ref{alg:dti2} is the solution of the linear systems
\[
	D({\rm i} \omega)^H X \;\; = \;\; C^H
		\quad	{\rm and}		\quad
	\left[ D({\rm i} \omega)^H \right]^2 Y \;\; = \;\; C^H
\]
at a given $\omega \in {\mathbb R}$. This can be treated similarly by exploiting that
\[
	D({\rm i} \omega)^{-H}  C^H
			\;\;  =  \;\;
	( {\rm i} \omega Q^{-1} - (J-R))^{-H} (CQ^{-1})^H.
\]

We have applied Algorithm \ref{alg:dti} to compute the unstructured stability radius $r(R; B, B^T)$
for the DH system of the form (\ref{insta}), (\ref{eq:insta_matrices}) resulting from the FE brake model
with $N = 0$, where $G(\Omega), K(\Omega), D(\Omega), M \in {\mathbb R}^{4669 \times 4669}$
so that $J, R, Q \in {\mathbb R}^{9338\times 9338}$.

The plot of the computed $r(R; B, B^T)$ vs the rotation speed $\Omega$ is presented
in Figure \ref{fig:brake_squeal_unstructured} at lower frequencies (i.e., $\Omega \in [2.5, 100]$)
on the top, and at higher frequencies (i.e., $\Omega \in [900, 1700]$)  at the bottom.
For smaller frequencies, the stability radius initially decreases with respect to $\Omega$,
but around $\Omega = 1100$ the stability radius suddenly increases.
The non-smooth nature of
the stability radius with respect to $\Omega$ is apparent from the figure. One should note, in particular,
the sharp turns near $\Omega = 1120$ and $\Omega = 1590$; this non-smoothness is
due to the fact that $\sigma_{\max}(B^TQ ({\rm i} \omega I - (J-R)Q)^{-1} B)$ has multiple
global maximizers. This means that two distinct points on the imaginary axis can be attained with
perturbations of minimal norm.

The computed values of $r(R; B, B^T)$ are listed in Table~\ref{tab:brake_squeal_unstructured}
for some values of $\Omega$. In this table, for each $\Omega$, the  value $\omega_\ast$, where the singular value function
$\sigma_{\max}(B^TQ ({\rm i} \omega I - (J-R)Q)^{-1} B)$ is maximized globally is displayed, the number of subspace iterations, the
run-time (in $s$) and the subspace dimension at termination are included as well.
In all cases, $2$ or $3$ subspace iterations are sufficient to achieve the prescribed accuracy tolerance. This leads to considerably smaller reduced systems of size $72 \times 72$ or $78\times 78$
compared with the original problem of size $9338\times 9338$. The $\omega$ value maximizing
$\sigma_{\max}(B^TQ ({\rm i} \omega I - (J-R)Q)^{-1} B)$ differs substantially, depending on whether the frequency $\Omega$ is small or large.

The resulting reduced problems at termination
capture the full problem remarkably well around the global maximizer.
This is depicted in Figure~\ref{fig:brake_squeal_unstructured_sval},
where, for $\Omega = 1000$, the singular value function $f(\omega) = \sigma_{\max}(B^TQ ({\rm i} \omega I - (J-R)Q)^{-1} B)$
for the full problem (solid curve) and 
$f_k(\omega) = \sigma_{\max}(B^T_kQ_k ({\rm i} \omega I - (J_k-R_k)Q_k)^{-1} B_k)$ for the reduced problem at termination
(dashed curve) are plotted near the global maximizer $\omega_\ast = -178880.9$. It is fairly difficult to
distinguish these two curves from each other. 

\begin{table}
	\begin{tabular}{|c||ccccc|}
	\hline
			$\Omega$		&	$r(R; B, B^T)$	&	$\omega_\ast$			&	iterations		&	run-time 	&	dimension	 \\
	\hline
	\hline
			 2.5			&	0.01066		&	$-1.938\times 10^5$		&	2			&	22.0		&	72	\\		
			 5			&	0.01038		&	$-1.938\times 10^5$		&	2			&	21.7		&	72	\\
			 10			&	0.01026		&	$-1.938\times 10^5$		&	3			&	24.9		&	78	\\
			 50			&	0.00999		&	$-1.938\times 10^5$		&	2			&	22.0		&	72	\\
			 100			&	0.00988		&	$-1.938\times 10^5$		&	2			&	22.1		&	72	\\	
			 1000		&	0.00809		&	$-1.789\times 10^5$		&	2			&	21.4		&	72	\\
			 1050		&	0.00789		&	$-1.789\times 10^5$		&	2			&	21.8		&	72	\\
			 1100		&	0.00834		&	$-1.789\times 10^5$		&	3			&	24.8		&	78	\\
			 1116		&	0.01091		&	$-1.789\times 10^5$		&	3			&	25.6		&	78	\\
			 1150		&	0.00344		&	$-1.742\times 10^5$		&	2			&	21.5		&	72	\\
			 1200		&	0.00407		&	$-1.742\times 10^5$		&	2			&	21.8		&	72	\\
			 1250		&	0.00471		&	$-1.742\times 10^5$		&	2			&	21.8		&	72	\\
			 1300		&	0.00516		&	$-1.742\times 10^5$		&	2			&	21.2		&	72	\\
	\hline
	\end{tabular}
\caption{Computed stability radii $r(R; B, B^T)$ by Algorithm \ref{alg:dti} for several  $\Omega$ values
for the DH system of order $9338$ originating from the FE brake model. The other columns display $\omega_\ast$ corresponding to $\: \argmax_{\omega} \sigma_{\max}(B^TQ ({\rm i} \omega I - (J-R)Q)^{-1} B) \:$, the number of subspace iterations, the total run-time (in $s$)
and the subspace dimension at termination. }
\label{tab:brake_squeal_unstructured}
\end{table}

\begin{figure}
	\begin{center}
		\begin{tabular}{c}
			\includegraphics[width=8cm]{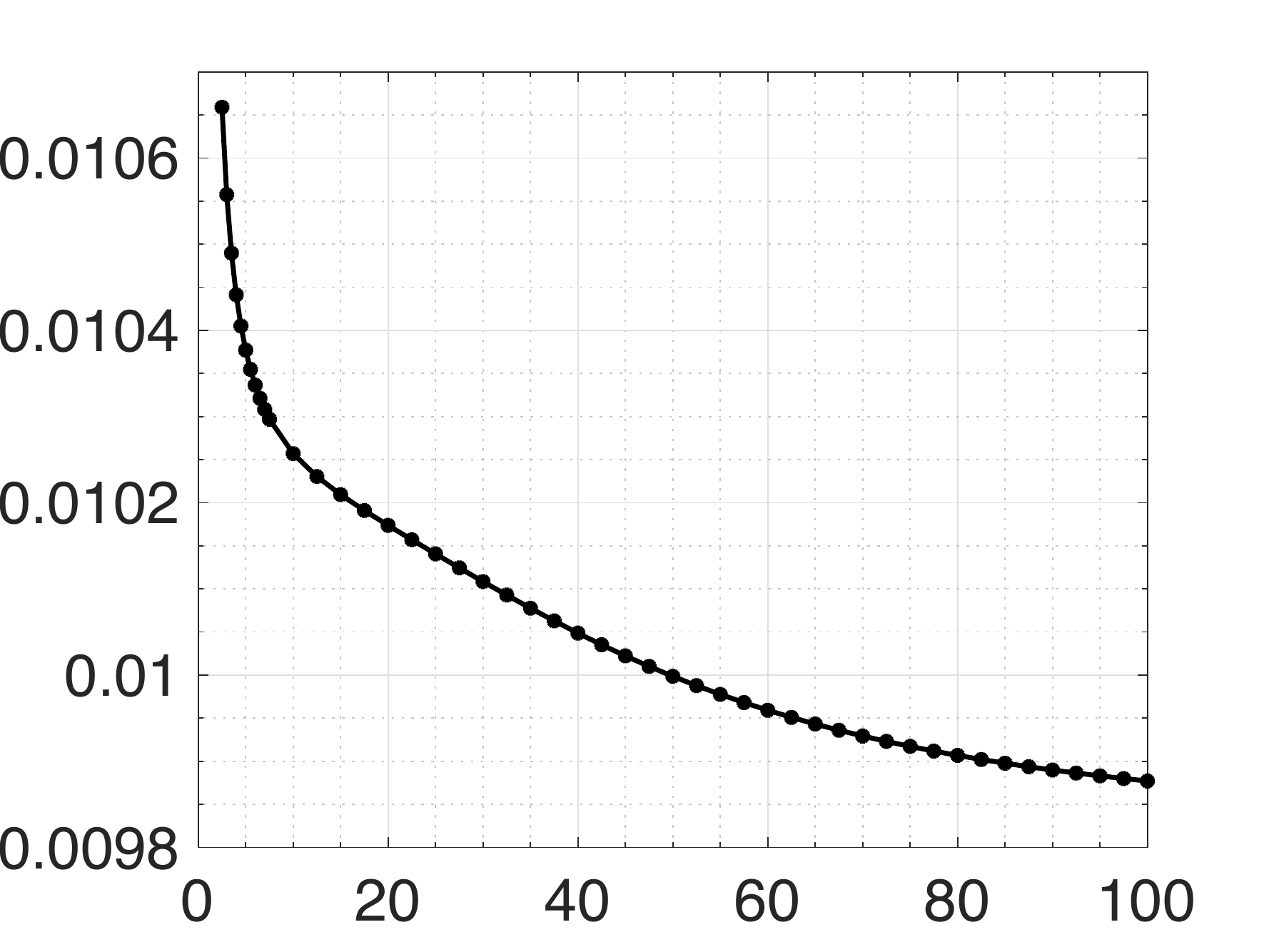} 	\\
			 \includegraphics[width=8cm]{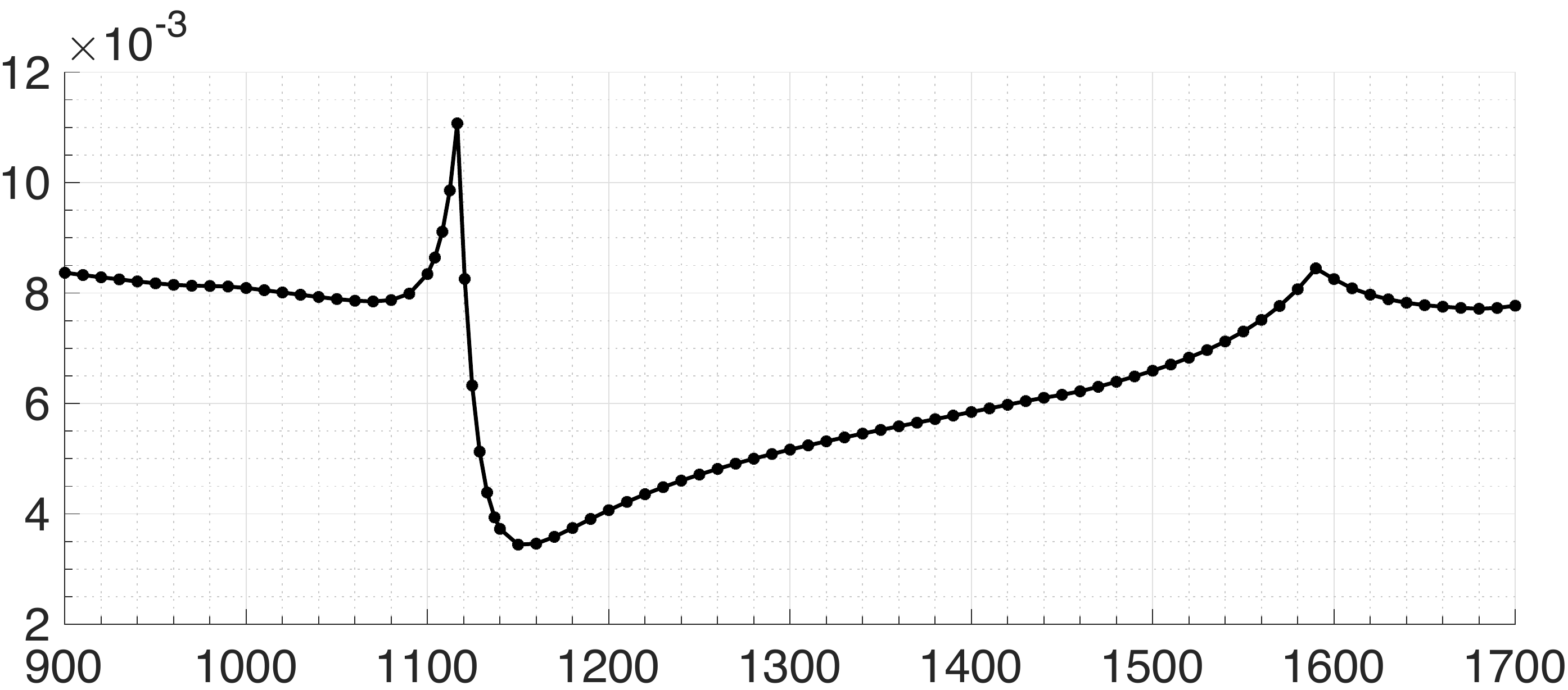}	
		\end{tabular}
	\end{center}
		\caption{ Plot of the stability radius $r(R; B, B^T)$ for the DH system (\ref{insta}), (\ref{eq:insta_matrices})
		resulting from the FE model of a disk-brake as a function of the rotation speed $\Omega$
		for $\Omega \in [2.5, 100]$ (top plot), and $\Omega \in [900, 1700]$ (bottom plot).
		The order of the DH system under consideration in this plot is $9338$. }
		\label{fig:brake_squeal_unstructured}
\end{figure}

\begin{figure}
\begin{center}
		\begin{tabular}{c}
		\includegraphics[width=8cm]{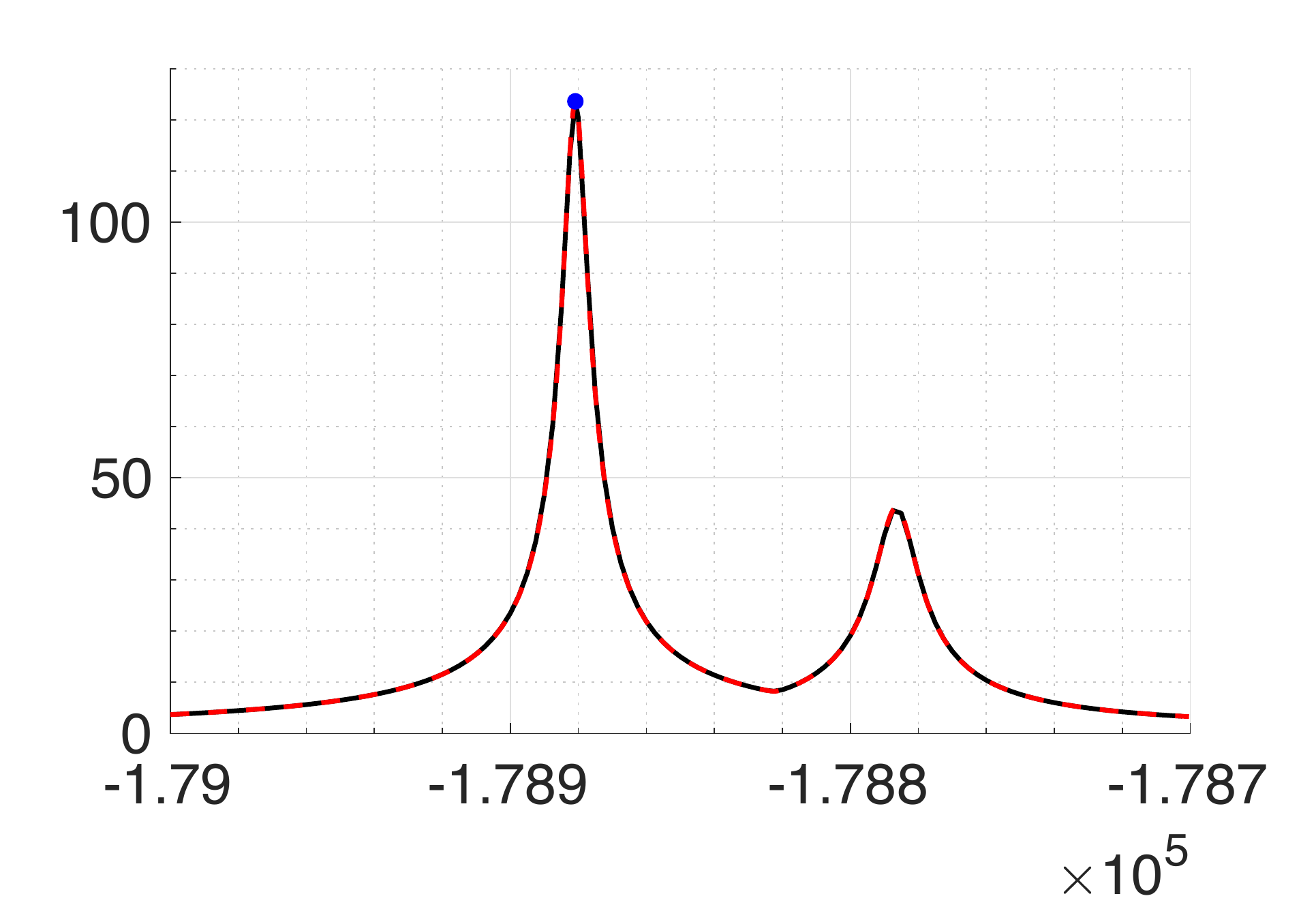}
		\end{tabular}
\end{center}
		\caption{ Plot of the singular value functions $f(\omega) = \sigma_{\max}(B^TQ ({\rm i} \omega I - (J-R)Q)^{-1} B)$ (solid curve)
		and $f_k(\omega) = \sigma_{\max}(B^T_kQ_k ({\rm i} \omega I - (J_k-R_k)Q_k)^{-1} B_k)$ (dashed curve) 
		at termination of Algorithm \ref{alg:dti}  near the global maximizer $\omega_\ast = -178880.9$
		for the DH system of order $9338$ arising from the FE disk-brake model with $\Omega = 1000$. The circle marks $(\omega_\ast, f(\omega_\ast))$.}
		\label{fig:brake_squeal_unstructured_sval}
\end{figure}
\section{Computation of the Structured Stability Radius}\label{sec:structured}
In the last section we have studied stability radii for dissipative Hamiltonian systems where the restriction matrices, however, allowed unstructured perturbations in the system coefficients. In this section we put additional constraints on the perturbations, in particular we
require that the perturbations are structured themselves. We discuss only perturbations in the dissipation matrix, since this is usually the most uncertain part of the system, due to the fact that modeling damping or friction very exactly is usually extremely difficult. We deal with the computation of $r^{\rm Herm}(R; B)$ defined as in (\ref{eq:structured_radii_defn}).
We first describe a numerical technique for small-scale problems in Section \ref{sec:small_scale_st} and then develop a subspace
framework that converges superlinearly with respect to the subspace dimension
in Section \ref{sec:large_scale_st}. Both techniques
use the eigenvalue optimization characterization of $r^{\rm Herm}(R; B)$ in
Theorem \ref{thm_str_Herm}.

\subsection{Small-Scale Problems}\label{sec:small_scale_st}
\subsubsection{Inner Maximization Problems}\label{sec:small_scale_inner}
The eigenvalue optimization characterization of $r^{\rm Herm}(R; B)$ is a min-max problem, where the inner 
maximization problem is concave, indeed it can alternatively be expressed as a
semi-definite program (SDP). Formally, for a given $\omega \in {\mathbb R}$, and $H_0({\rm i} \omega), H_1({\rm i} \omega)$
representing the Hermitian matrices defined in Theorem \ref{thm_str_Herm}, we have
\begin{equation}\label{eq:inner_max}
	\begin{split}
	\widetilde{\eta}^{\rm Herm}(R; B, {\rm i}\omega)
			& \;	=	\;
	\sup_{t \in {\mathbb R}} \: \lambda_{\min} (H_0({\rm i} \omega) + t H_1 ({\rm i} \omega))	\\
			& \;	=	\;
	\sup\{   z   \; | \;    z,t \in {\mathbb R} \;\; {\rm s.t.} \;\;  H_0({\rm i} \omega) + t H_1 ({\rm i} \omega)  - zI \geq 0		\},
	\end{split}
\end{equation}
where the characterization in the second line is a linear convex SDP. Here $\widetilde{\eta}^{\rm Herm}(R; B, {\rm i}\omega)$
is related to a structured backward error for the eigenvalue ${\rm i} \omega$, specifically it corresponds to the
square of the distance (see \cite[Definition 3.2 and Theorem 4.9]{MehMS16a})
\begin{equation}\label{eq:backward_error}
	\begin{split}
	\eta^{\rm Herm}(R; B, {\rm i}\omega)
			\; := \;	\hskip 49ex	\\
	\hskip 13ex
	 \inf\{ \|\Delta \|_2 \; | \; \Delta = \Delta^H, \;\;	  {\rm i} \omega \in \Lambda \big( (J - R)Q - (B\Delta B^H)Q \big) \}.
	 \end{split}
\end{equation}
We have that $\eta^{\rm Herm}(R; B, {\rm i}\omega)$ is finite if and only if the suprema in (\ref{eq:inner_max})
are attained, which happens  if and only if $H_1({\rm i}\omega)$ is indefinite, i.e., $H_1({\rm i}\omega)$ has both negative and positive eigenvalues.

The most widely used techniques to solve a linear convex SDP  are different forms of interior-point methods.
Implementations of some of these interior-point methods are made available through the package \texttt{cvx} \cite{Grant2008,cvx}.
Hence, one option is to use \texttt{cvx}
to compute $\widetilde{\eta}^{\rm Herm}(R; B, {\rm i}\omega)$ directly. An alternative, and also theoretically well understood approach, is to employ the software package \texttt{eigopt} \cite{Mengi2014} for
the eigenvalue optimization problem in the first characterization in (\ref{eq:inner_max}). This second approach forms piece-wise quadratic functions that lie globally above the eigenvalue function, and maximizes these piece-wise quadratic functions instead of the eigenvalue function. Each piece-wise quadratic function is defined as the minimum
of several other quadratic functions, all of which have the same curvature $\gamma$
(which must be a global upper bound on the second derivative of the eigenvalue function at all points where the eigenvalue function is differentiable).
Any slightly positive real number for the curvature $\gamma$ serves the
purpose (e.g., $\gamma = 10^{-6}$), since the smallest eigenvalue function in (\ref{eq:inner_max}) is a concave function of $t$.

In our experience, \texttt{eigopt} performs the computation of $\widetilde{\eta}^{\rm Herm}(R; B, {\rm i}\omega)$ significantly faster than \texttt{cvx}. The only downside is that an interval containing the optimal $t$ for (\ref{eq:inner_max})
must be supplied to \texttt{eigopt}, whereas such an interval is not needed by \texttt{cvx}.

\subsubsection{Outer Minimization Problems}\label{sec:small_scale_outer}
The minimum of $\widetilde{\eta}^{\rm Herm}(R; B, {\rm i}\omega)$ with respect to $\omega \in {\mathbb R}$ yields the
distance $r^{\rm Herm}(R; B)$, and the minimizing $\omega \in {\mathbb R}$ yields the point ${\rm i} \omega$
that first becomes an eigenvalue on the imaginary axis under the smallest perturbation possible.
This is a non-convex optimization problem, indeed the objective $\widetilde{\eta}^{\rm Herm}(R; B, {\rm i}\omega)$
may even blow up at some $\omega$.

We again resort to \texttt{eigopt} for the minimization of $\widetilde{\eta}^{\rm Herm}(R; B, {\rm i}\omega)$.
For the sake of completeness, a formal description is provided in Algorithm \ref{eigopt} below, where we use the abbreviations
\[
	\widetilde{\eta}_j \; := \; \widetilde{\eta}^{\rm Herm}(R; B, {\rm i}\omega_j)
		\quad	{\rm and}		\quad
	\widetilde{\eta}^{\;'}_j \; := \; \frac{d\, \widetilde{\eta}^{\rm Herm}(R; B, {\rm i}\omega_j)}{d \omega}.
\]
Introducing $t(\omega) := \argmax_{t \in {\mathbb R}} \: \lambda_{\min}(H_0({\rm i} \omega) + t H_1({\rm i} \omega))$,
the algorithm approximates the smallest eigenvalue function
\[
	\lambda_{\min}(H_0({\rm i} \omega) + t(\omega) H_1({\rm i} \omega))
			\; = \;
	\widetilde{\eta}^{\rm Herm}(R; B, {\rm i}\omega)
\]
with the piece-wise quadratic function
\begin{eqnarray*}
	Q_k(\omega) &:=&	\max \{  q_j(\omega) \; | \; j = 0, \dots, k \},	\\
			q_j(\omega) \; &:=& \; \widetilde{\eta}_j  +
							\widetilde{\eta}_j^{\;'} ( \omega - \omega_j )  +  (\gamma/2) ( \omega - \omega_j )^2
\end{eqnarray*}
at iteration $k$. It computes the global minimizer $\omega_{k+1}$ of $Q_k(\omega)$, and refines the piece-wise quadratic
function $Q_k(\omega)$ with the addition of one more quadratic piece, namely
$q_{k+1}(\omega) \; := \; \widetilde{\eta}_{k+1}  +
							\widetilde{\eta}_{k+1}^{\;'} ( \omega - \omega_{k+1} )  +  (\gamma/2) ( \omega - \omega_{k+1} )^2$.
Here, $\gamma$ is supposed to be a lower bound for the second derivative $\lambda_{\min}''(H_0({\rm i} \omega) + t(\omega) H_1({\rm i} \omega))$
for all $\omega$ sufficiently close to the global minimizer of $\widetilde{\eta}^{\rm Herm}(R; B, {\rm i}\omega)$.
In theory, it can be shown that for all $\gamma$ small enough, every convergent subsequence of the sequence
$\{ \omega_k\}$ converges to a global minimizer of $\widetilde{\eta}^{\rm Herm}(R; B, {\rm i}\omega)$.

At step $k$ of the algorithm $\widetilde{\eta}^{\rm Herm}(R; B, {\rm i}\omega)$ and its derivative need
to be computed at $\omega_{k+1}$. We rely on one of the two approaches (\texttt{cvx} or \texttt{eigopt})
described in Section \ref{sec:small_scale_inner} for the computation of $\widetilde{\eta}^{\rm Herm}(R; B, {\rm i}\omega_{k+1})$,
and employ a finite difference formula to approximate its derivative.

\begin{algorithm}
	\begin{algorithmic}[1]
		\REQUIRE{Matrices $B \in {\mathbb C}^{n\times m}$, $J, R, Q \in {\mathbb C}^{n\times n}$,
		a negative real number $\gamma$, and a closed interval $\widetilde \Omega \subseteq {\mathbb R}$
		that contains the global minimizer of $\widetilde{\eta}^{\rm Herm}(R; B, {\rm i}\omega)$ over
		$\omega \in {\mathbb R}$.}
		\ENSURE{The sequence $\{ \omega_k \}$.}
		\STATE{$\omega_0 \gets$ an initial point in $\widetilde \Omega$}
		\STATE Compute $\widetilde{\eta}_0 := \widetilde{\eta}^{\rm Herm}(R; B, {\rm i}\omega_0)$ and
					$\widetilde{\eta}_0^{\;'} := d\, \widetilde{\eta}^{\rm Herm}(R; B, {\rm i}\omega_0) / d\omega$
		\FOR{$k=0,1,\dots$}
		\STATE{$Q_k(\omega)\gets\max\left\{q_j(\omega)\;|\; j=0,\dots,k\right\}$, where \\
				\hskip 22ex
					$q_j(\omega) \; := \; \widetilde{\eta}_j  +
							\widetilde{\eta}_j^{\;'} ( \omega - \omega_j )  +  (\gamma/2) ( \omega - \omega_j )^2$}
		\STATE{$\omega_{k+1} \gets \argmin_{\omega\in\widetilde\Omega} Q_k(\omega)$}
		\STATE Compute
\begin{eqnarray*}\widetilde{\eta}_{k+1} &:=& \widetilde{\eta}^{\rm Herm}(R; B, {\rm i}\omega_{k+1}),\\
					\widetilde{\eta}_{k+1}^{\;'} &:=& d\, \widetilde{\eta}^{\rm Herm}(R; B, {\rm i}\omega_{k+1}) / d\omega
\end{eqnarray*}
		\ENDFOR
	\end{algorithmic}
	\caption{Small-Scale Computation of $r^{\rm Herm}(R; B)$.}
	\label{eigopt}
\end{algorithm}

\subsection{Large-Scale Problems}\label{sec:large_scale_st}
The characterization via eigenvalue optimization in Theorem \ref{thm_str_Herm} is in terms of
the matrix-valued functions $H_0({\rm i} \omega), H_1({\rm i} \omega)$, which are of small size
 provided that $B$ has few columns. The large-scale nature of this characterization is hidden
 in the matrix-valued function $W({\rm i} \omega) := (J-R)Q - {\rm i} \omega I$ defined in Theorem~\ref{thm_str_Herm}. Note that, in particular, both $H_0({\rm i} \omega)$
 and $H_1({\rm i} \omega)$ are defined in terms of $W({\rm i} \omega)^{-1} B$.
 This is also reflected in Algorithm \ref{eigopt} when $J, R, Q$ are large;
 at iteration $k$ of the algorithm, the matrices $H_0({\rm i} \omega_{k+1})$, $H_1({\rm i} \omega_{k+1})$
need to be formed for the computation of $\widetilde{\eta}^{\rm Herm}(R; B, {\rm i}\omega_{k+1})$,
which in turn requires the solution of the linear system $W({\rm i} \omega_{k+1}) Z = B$.

To cope with the large-scale setting, we benefit from structure preserving two-sided projections
similar to those described in Section \ref{sec:sf_PH}. In particular, for a given subspace ${\mathcal V}_k$
and a matrix $V_k$ whose columns form an orthonormal basis for ${\mathcal V}_k$, we set
\begin{equation}\label{eq:defn_Wr}
		W_k	\;\;	:=	\;\;	QV_k (V^H_k Q V_k)^{-1}.
\end{equation}
%
Furthermore, we define the projected matrices
\begin{equation}\label{eq:projected_matrices}
		\begin{split}
			J_k   \;\; & := \;\; W_k^H J W_k,	\;\;		R_k \;\; := \;\; W_k^H R W_k,	\\
			Q_k  \;\; & := \;\; V_k^H Q V_k,		\;\;\;\:		B_k \;\; := \;\; W_k^H B.
		\end{split}
\end{equation}
Recall also the identities
\begin{equation}\label{eq:identities}
		W_k^H V_k	=	I		\quad	{\rm and}		\quad	(W_k V_k^H)^2	=	W_k V_k^H,
\end{equation}
the latter of which means  that $W_k V_k^H$ is an oblique projector onto ${\rm Im}(Q V_k)$.

Although these identities are still available, however, we no longer have a tool such as
Theorem~\ref{thm:Gal_interpolate} that we could depend on to establish interpolation results.
This is because there is no apparent transfer function, as there is indeed no apparent linear port-Hamiltonian
system that can be tied to the eigenvalue optimization characterization. But the following simple observation turns out to be very useful.
\begin{lemma}\label{thm:reduced_Wlambda}
Consider a DH model~(\ref{DH}) and a reduced model $\dot x_k= (J_k-R_k)Q_k x_k$ with coefficients as in (\ref{eq:projected_matrices}).
With $W(\lambda) = (J - R) Q  -  \lambda I$ and $W_k(\lambda) := (J_k - R_k) Q_k  -  \lambda I$, we then have
\[
	W_k(\lambda)	\;\;	=	\;\;  W_k^H \: W(\lambda) \: V_k	\quad\quad	\mbox{\rm for all}\ \lambda \in {\mathbb C}.
\]
\end{lemma}
\begin{proof}
 From the definition of $J_k, R_k, Q_k$ in (\ref{eq:projected_matrices}) we obtain
\[
	\begin{split}
	W_k(\lambda) 	& \;\; = \;\; (W_k^H (J - R) W_k) V_k^H Q V_k	-	\lambda I		\\
				& \;\; = \;\;	W_k^H (J - R)  Q V_k	-	\lambda W_k^H V_k	\\
				& \;\; = \;\;	W_k^H \left\{ (J-R) Q  -  \lambda I  \right\}  V_k  \;\; = \;\;  W_k^H W(\lambda) V_k,		
	\end{split}
\]
where we have employed  the identities in (\ref{eq:identities}).
\end{proof}

Our reduced problems are expressed in terms of the reduced versions of $H_0(\lambda)$, $L(\lambda)$, and $H_1(\lambda)$ defined via
\[
	H_{k,0}(\lambda) := L_k(\lambda)^{-1} L_k(\lambda)^{-H},
\]
with $L_k(\lambda)$ denoting a lower triangular Cholesky factor of
\[
	\widetilde{H}_{k,0}(\lambda) := B^H_k W_k(\lambda)^{-H} Q_k B_k B^H_k Q_k W_k(\lambda)^{-1} B_k,
\]
and $H_{k,1}(\lambda) := \ri (\widetilde{H}_{k,1}(\lambda)  -  \widetilde{H}_{k,1}(\lambda)^H)$ with
\[
	\widetilde{H}_{k,1}(\lambda) := L_k(\lambda)^{-1} B^H_k W_k(\lambda)^{-H} Q_k B_k L_k(\lambda)^{-H}.
\]
Note that to ensure the uniqueness of $L(\lambda)$ and $L_k(\lambda)$, we define them as the Cholesky factors of
$\widetilde{H}_{0}(\lambda)$ and $\widetilde{H}_{k,0}(\lambda)$ with real and positive entries along the diagonal.
Our goal is to come up with reduced counterparts of $\widetilde{\eta}^{\rm Herm}(R; B, {\rm i}\omega)$,
$\eta^{\rm Herm}(R; B, {\rm i}\omega)$ that Hermite-interpolate the full functions at prescribed points.
For a given subspace ${\mathcal V}_k$, a matrix $V_k$ whose columns form an orthonormal
basis for ${\mathcal V_k}$, and for $W_k$ as in (\ref{eq:defn_Wr}), we introduce
\[
	\begin{split}
			\widetilde{\eta}^{\rm Herm}_k(R; B, {\rm i}\omega)
			 \;	& :=	\;
			\sup_{t \in {\mathbb R}} \: \lambda_{\min} (H_{k,0}({\rm i} \omega) + t H_{k,1} ({\rm i} \omega))	 \\
			\eta^{\rm Herm}_k(R; B, {\rm i}\omega)
			\; 	& := 	\;	
	 	\inf\{ \|\Delta \|_2 \; | \; \Delta = \Delta^H, \;\;	 {\rm i} \omega \in \Lambda \big( (J_k - R_k)Q_k - (B_k\Delta B^{H}_k)Q_k \big) \}.
	\end{split}
\]
Recall that $\eta^{\rm Herm}(R; B, {\rm i}\omega) = \widetilde{\eta}^{\rm Herm}(R; B, {\rm i}\omega)^{1/2}$, and a similar relation holds
for the reduced problems, i.e., $\eta^{\rm Herm}_k(R; B, {\rm i}\omega) = \widetilde{\eta}^{\rm Herm}_k(R; B, {\rm i}\omega)^{1/2}$.

We start our analysis by establishing that the quantities $\eta^{\rm Herm}_k(R; B, {\rm i}\omega)$,
$\widetilde{\eta}^{\rm Herm}_k(R; B, {\rm i}\omega)$ are independent of the choice of basis for the subspace
${\mathcal V}_k$. For this proof we introduce the notation $\eta^{\rm Herm}_{V_k,W_k}(R; B, {\rm i}\omega))$,
$\widetilde{\eta}^{\rm Herm}_{V_k,W_k}(R; B, {\rm i}\omega)$
to emphasize the particular choices of basis $V_k, W_k$ for the  subspaces ${\mathcal V_k}, {\mathcal W}_k$ used in the
definitions of $\eta^{\rm Herm}_k(R; B, {\rm i}\omega), \widetilde{\eta}^{\rm Herm}_k(R; B, {\rm i}\omega)$.
Similarly, we indicate the spaces and the bases with the help of the notations $J_{W_k}$, $R_{W_k}$, $Q_{V_k}$, and $ B_{W_k}$.
%
\begin{lemma}\label{thm:eta_basis_independent}
Let the columns of $V_k$ and $\widetilde{V}_k$ form orthonormal bases for the subspace ${\mathcal V}_k$, and let
$W_k := Q V_k (V_k^H Q V_k)^{-1}$, $\widetilde{W}_k := Q \widetilde{V}_k (\widetilde{V}_k^H Q \widetilde{V}_k)^{-1}$.
Then
\[
	\begin{split}
	\eta^{\rm Herm}_{V_k, W_k}(R; B, {\rm i}\omega)	&  = 	\eta^{\rm Herm}_{\widetilde{V}_k, \widetilde{W}_k}(R; B, {\rm i}\omega), \\
	\widetilde{\eta}^{\rm Herm}_{V_k, W_k}(R; B, {\rm i}\omega)	&  = 	\widetilde{\eta}^{\rm Herm}_{\widetilde{V}_k, \widetilde{W}_k}(R; B, {\rm i}\omega)
	\end{split}
\]
for all $\omega \in {\mathbb R}$.
\end{lemma}
\begin{proof}
It suffices to prove
$\eta^{\rm Herm}_{V_k, W_k}(R; B, {\rm i}\omega)	  = 	\eta^{\rm Herm}_{\widetilde{V}_k, \widetilde{W}_k}(R; B, {\rm i}\omega)$,
the other equality follows from this equality immediately.
Since the columns of $V_k$ and $\widetilde V_k$ form bases for the same space, there exists a unitary matrix $P$ such that $\widetilde{V}_k = V_k P$. Furthermore,
by definition,
\[
	\widetilde{W}_k  =  Q (V_k P) (P^H V_k^H Q V_k P)^{-1}	=	Q V_k P P^H (V_k^H Q V_k)^{-1} P	=	W_k P.
\]
The assertion then follows from the following set of equivalences:
\[
	\begin{split}
		{\rm i} \omega \in \Lambda \big( (J_{W_k} - R_{W_k}) Q_{V_k} - (B_{W_k} \Delta B^H_{W_k})Q_{V_k} \big)
			\;\;	\Longleftrightarrow	\;\;		\\
{\rm det}(W_k^H (J - R)Q V_k  - W_k^H B \Delta B^H Q V_k  - {\rm i} \omega W_k^H V_k	) = 0 \;\;	\Longleftrightarrow	\;\;		\\{\rm det}(W_k^H \left( (J - R)Q  -  B \Delta B^H Q  - {\rm i} \omega I \right) V_k	) = 0 \;\;	\Longleftrightarrow	\;\;		\\
{\rm det}(P^H W_k^H \left( (J - R)Q  -  B \Delta B^H Q  - {\rm i} \omega I \right) V_k P ) = 0 \;\;	\Longleftrightarrow	\;\;		\\
{\rm det}(\widetilde{W}_k^H (J - R)Q \widetilde{V}_k  - \widetilde{W}_k^H B \Delta B^H Q \widetilde{V}_k  - {\rm i} \omega \widetilde{W}_k^H \widetilde{V}_k	) = 0 \;\;	\Longleftrightarrow	\;\;		\\
		{\rm i} \omega \in \Lambda \big( (J_{\widetilde{W}_k} - R_{\widetilde{W}_k}) Q_{\widetilde{V}_k} - (B_{\widetilde{W}_k} \Delta B^H_{\widetilde{W}_k})Q_{\widetilde{V}_k} \big).	\quad\quad\quad
	\end{split}
\]
\end{proof}
In the following we will develop a subspace framework including an Hermite interpolation property for DH systems. For this we first show an auxiliary interpolation result
for $B^{H} Q W(\lambda)^{-1} B$, the matrix through which $H_0(\lambda), H_1(\lambda)$ are defined.
\begin{theorem}\label{thm:interpolate_matrixfuns}
Consider a DH model~(\ref{DH}) and a reduced model $\dot x_k= (J_k-R_k)Q_k x_k$ with coefficients as in (\ref{eq:projected_matrices}), and let $W(\lambda) = (J - R) Q  -  \lambda I$ and $W_k(\lambda) := (J_k - R_k) Q_k  -  \lambda I$.
For a given $\widehat{\lambda} \in {\mathbb C}$ such that $W(\widehat{\lambda})$ and $W_k(\widehat{\lambda})$
are invertible, the following assertions hold:
\begin{enumerate}
	\item[\bf (i)] If $\: {\rm Im}(W(\widehat{\lambda})^{-1} B) \subseteq {\mathcal V}_k$, then
				$B^H Q W(\widehat{\lambda})^{-1} B \; = \; B^H_k Q_k W_k(\widehat{\lambda})^{-1} B_k$.
	\item[\bf (ii)] Additionally, if $\: {\rm Im}(W(\widehat{\lambda})^{-2} B) \subseteq {\mathcal V}_k$ and
	the orthonormal basis $V_k$ for ${\mathcal V}_k$ is such that
	$\:
			V_k
				=
			\left[
				\begin{array}{cc}
					\widetilde{V}_k	&	\widehat{V}_k
				\end{array}
			\right]
	$	
	where the columns of $\:\: \widetilde{V}_k$ form an orthonormal basis for ${\rm Im}(W(\widehat{\lambda})^{-1} B)$, then
	$\: B^H Q W(\widehat{\lambda})^{-2} B \; = \;$ $B^H_k Q_k W_k(\widehat{\lambda})^{-2} B_k$.
\end{enumerate}
\end{theorem}
\begin{proof}
\textbf{(i)} If ${\rm Im}(W(\widehat{\lambda})^{-1} B) \subseteq {\mathcal V}_k$ then, since
$W_k V_k^H$ is a projector onto ${\rm Im}(Q V_k)$, we obtain
\[
	\begin{split}
	B^H Q W(\widehat{\lambda})^{-1} B
		\; = \;
	B^H Q V_k V_k^H W(\widehat{\lambda})^{-1} B
		& \; = \;
	B^H W_k V_k^H Q V_k V_k^H W(\widehat{\lambda})^{-1} B	\\
		& \; = \;
	B^H_k Q_k V_k^H W(\widehat{\lambda})^{-1} B.
	\end{split}
\]
To show that
	$V_k^H W(\widehat{\lambda})^{-1} B	=	W_k(\widehat{\lambda})^{-1} B_k$, let $Z := W(\widehat{\lambda})^{-1} B$, and $Z_k$ be such that $V_k Z_k = Z$. (There exists a unique $Z_k$ with this property, because ${\rm Im}(Z) \subseteq {\mathcal V}_k$.) Then
$W(\widehat{\lambda}) Z = B$ implies that $	W(\widehat{\lambda}) V_k Z_k = B$, and thus $	W^H_k W(\widehat{\lambda}) V_k Z_k = W^H_k B$. Hence,  by Lemma~\ref{thm:reduced_Wlambda} we see that $Z_k  = W_k(\widehat{\lambda})^{-1} B_k$,
implying that
\[
	V_k^{H} W(\widehat{\lambda})^{-1} B	\;	=	\;
	V_k^{H} Z	\;	=	\;
	V_k^{H} (V_k Z_k)	\;	=	\;
	W_k(\widehat{\lambda})^{-1} B_k,
\]
as asserted.

\textbf{(ii)} Following the steps at the beginning of the proof of part (i), we have
\[
	B^H Q W(\widehat{\lambda})^{-2} B	\;\; = \;\;	B^H_k Q_k V_k^H W(\widehat{\lambda})^{-2} B.
\]
To show that $V_k^H W(\widehat{\lambda})^{-2} B = W_k(\widehat{\lambda})^{-2} B_k$, we exploit that
\begin{equation}\label{eq:intermed_secpow_interpolate}
	V_k^H W(\widehat{\lambda})^{-2} B	\;\; = \;\;	
		(V_k^H W(\widehat{\lambda})^{-1} 	\widetilde{V}_k) (\widetilde{V}_k^H W(\widehat{\lambda})^{-1}B).
\end{equation}
Now define $Z := W(\widehat{\lambda})^{-1} \widetilde{V}_k$ and $Z_k$ such that $V_k Z_k = Z$
(once again such a $Z_k$ exists uniquely, because ${\rm Im}(Z) \subseteq {\mathcal V}_k$) so that $W(\widehat{\lambda}) Z = \widetilde{V}_k$. Then $
	W(\widehat{\lambda}) V_k Z_k = \widetilde{V}_k$ and hence $	W^H_k W(\widehat{\lambda}) V_k Z_k = I_{\widetilde{k},m}$,
where $I_{\widetilde{k},m}$ is the matrix consisting of the first  $m$ columns of the $\widetilde{k}\times \widetilde{k}$ identity matrix 
with $\widetilde{k} := {\rm dim} \: {\mathcal V}_k > m$. This implies that
$Z_k = W_k(\widehat{\lambda})^{-1} I_{\widetilde{k},m}$, so the following can be deduced about
the term inside the first parenthesis on the right-hand side of (\ref{eq:intermed_secpow_interpolate}):
\[
	V_k^H W(\widehat{\lambda})^{-1} 	\widetilde{V}_k
			\; = \;
	V_k^H Z
			\; = \;
	V_k^H (V_k Z_k)
			\; = \;
	W_k(\widehat{\lambda})^{-1} I_{\widetilde{k},m}.
\]
As for the term inside the second parenthesis on the right-hand side of (\ref{eq:intermed_secpow_interpolate}),
we make use of the following observation:
\begin{equation}\label{eq:intermed2_secpow_interpolate}
	V_k^H W(\widehat{\lambda})^{-1}B
			\; = \;
	\left[
		\begin{array}{c}
			\widetilde{V}_k^H 	\\
			\widehat{V}_k^H
		\end{array}
	\right]
	W(\widehat{\lambda})^{-1}B
			\; = \;
	\left[
		\begin{array}{c}
		\widetilde{V}_k^H W(\widehat{\lambda})^{-1}B	\\
		0
		\end{array}
	\right],
\end{equation}
where the last equality follows, since the columns of $\widetilde{V}_k$ form an orthonormal basis
for ${\rm Im}(W(\widehat{\lambda})^{-1}B)$. Putting these observations together
in (\ref{eq:intermed_secpow_interpolate}), we obtain
\[
	\begin{split}
	V_k^H W_k (\widehat{\lambda})^{-2}B_k \; & = \;
	(V_k^H W(\widehat{\lambda})^{-1} 	\widetilde{V}_k) (\widetilde{V}_k^H W(\widehat{\lambda})^{-1}B)	\\
	\; & = \;	
	W_k(\widehat{\lambda})^{-1} I_{\widetilde{k},m} (\widetilde{V}_k^H W(\widehat{\lambda})^{-1}B)	\\
	\; & = \;
	W_k(\widehat{\lambda})^{-1} I_{\widetilde{k}} (V_k^H W(\widehat{\lambda})^{-1}B)	\\
	\; & = \;
	W_k(\widehat{\lambda})^{-1} I_{\widetilde{k}} ( W_k (\widehat{\lambda})^{-1}B_k )	
	\; = \;	 W_k (\widehat{\lambda})^{-2}B_k,
	\end{split}
\]
where in the third equality we exploit (\ref{eq:intermed2_secpow_interpolate}), and in the fourth
equality we employ that $V_k^H W(\widehat{\lambda})^{-1}B = W_k (\widehat{\lambda})^{-1}B_k$ which is proven in part (i).
\end{proof}

After these preparations we can prove our main interpolation result.
\begin{theorem}\label{thm:main_Hermite_interpolate}
Consider a DH model~(\ref{DH}) and a reduced model $\dot x_k= (J_k-R_k)Q_k x_k$ with coefficients as in (\ref{eq:projected_matrices}), and let $W(\lambda) = (J - R) Q  -  \lambda I$ and $W_k(\lambda) := (J_k - R_k) Q_k  -  \lambda I$. Suppose that the subspace ${\mathcal V}_k$ is such that
\begin{equation}\label{eq:subspace_inclusion}
	{\rm Im}( W({\rm i} \widehat{\omega})^{-1} B ), {\rm Im}( W({\rm i} \widehat{\omega})^{-2} B ) \subseteq {\mathcal V}_k,
\end{equation}
and $W_k$ is defined as in (\ref{eq:defn_Wr}) in terms of a matrix $V_k$ whose columns form an orthonormal
basis for ${\mathcal V}_k$.
\begin{enumerate}
\item[\bf (i)]
The quantity $\widetilde{\eta}^{\rm Herm}(R; B, {\rm i} \widehat{\omega})$ is finite if and only if
$\widetilde{\eta}^{\rm Herm}_k(R; B, {\rm i} \widehat{\omega})$ is finite. If $\widetilde{\eta}^{\rm Herm}(R; B, {\rm i} \widehat{\omega})$
is finite, then
\begin{equation}\label{eq:main_interpolation_kes}
	\widetilde{\eta}^{\rm Herm}(R; B, {\rm i} \widehat{\omega})
						\; = \;
					\widetilde{\eta}^{\rm Herm}_k(R; B, {\rm i} \widehat{\omega}).
\end{equation}
\item[\bf (ii)]
Moreover, if $\widetilde{\eta}^{\rm Herm}(R; B, {\rm i} \omega)$ and $\widetilde{\eta}^{\rm Herm}_k(R; B, {\rm i} \omega)$
are differentiable at $\widehat{\omega}$, then we have
\begin{equation}\label{eq:main_interpolation_kes2}
	\frac{d\, \widetilde{\eta}^{\rm Herm}(R; B, {\rm i} \widehat{\omega})}{ d \omega}
						\; = \;
					\frac{d\, \widetilde{\eta}^{\rm Herm}_k(R; B, {\rm i} \widehat{\omega})}{d \omega}.
\end{equation}
\end{enumerate}
\end{theorem}
\begin{proof}
Without loss of generality, we may assume that the matrix $V_k$ is such that
$
	V_k
		=
	\left[
		\begin{array}{cc}
		\widetilde{V}_k	  &	\widehat{V}_k	
		\end{array}
	\right],
$
with the columns of $\widetilde{V}_k$ forming an orthonormal basis for ${\rm Im}( W({\rm i} \widehat{\omega})^{-1} B )$.
It suffices to prove the claims for this particular choice of orthonormal basis,
because it is established in Lemma \ref{thm:eta_basis_independent} that the function
$\widetilde{\eta}^{\rm Herm}_k(R; B, {\rm i} \omega)$, hence its derivative, are
independent of the choice of $V_k$ as long as its columns form an orthonormal basis for ${\mathcal V}_k$.

\textbf{(i)} By the definitions of $\widetilde{H}_0 (\lambda), \widetilde{H}_{k,0} (\lambda)$, and part (i) of
Theorem \ref{thm:interpolate_matrixfuns}, we have
\[
	\begin{split}
	\widetilde{H}_0 ({\rm i} \widehat{\omega})
			& \; = \;
		(B^{H} Q W({\rm i} \widehat{\omega})^{-1} B)^H (B^{H} Q W({\rm i} \widehat{\omega})^{-1} B)	\\
			& \; = \;
		(B^{H}_k Q_k W_k({\rm i} \widehat{\omega})^{-1} B_k)^H (B^{H}_k Q_k W_k({\rm i} \widehat{\omega})^{-1} B_k)
			\; = \;
	\widetilde{H}_{k,0} ({\rm i} \widehat{\omega}).
	\end{split}
\]
This also implies that $L({\rm i} \widehat{\omega}) = L_k({\rm i} \widehat{\omega})$ due to the uniqueness
of the Cholesky factors of $\widetilde{H}_0 ({\rm i} \widehat{\omega})$, $\widetilde{H}_{k,0} ({\rm i} \widehat{\omega})$. Therefore,
\[
	H_0({\rm i} \widehat{\omega}) = L({\rm i} \widehat{\omega})^{-1} L({\rm i} \widehat{\omega})^{-H}
							=	L_k({\rm i} \widehat{\omega})^{-1} L_k({\rm i} \widehat{\omega})^{-H}
							=	H_{k,0} ({\rm i} \widehat{\omega}).
\]
Furthermore,
\[
	\begin{split}
	\widetilde{H}_1 ({\rm i} \widehat{\omega})
				& \; = \;
		L({\rm i} \widehat{\omega})^{-1} (B^{H} Q W({\rm i} \widehat{\omega})^{-1} B)^H L({\rm i} \widehat{\omega})^{-H}	\\
				& \; = \;
L_k({\rm i} \widehat{\omega})^{-1} (B^{H}_k Q_k W_k({\rm i} \widehat{\omega})^{-1} B_k)^H L_k({\rm i} \widehat{\omega})^{-H}	
				\; = \;
		\widetilde{H}_{k,1} ({\rm i} \widehat{\omega})	
	\end{split}
\]
and
\[
	H_1 ({\rm i} \widehat{\omega})
			=
	{\rm i}
	\left(
			\widetilde{H}_1({\rm i} \widehat{\omega})
					-
			\widetilde{H}_1({\rm i} \widehat{\omega})^H
	\right)
			=
	{\rm i}
	\left(
			\widetilde{H}_{k,1}({\rm i} \widehat{\omega})
					-
			\widetilde{H}_{k,1}({\rm i} \widehat{\omega})^H
	\right)
			=
	H_{k,1} ({\rm i} \widehat{\omega}).
\]
Since $H_1 ({\rm i} \widehat{\omega}) = H_{k,1} ({\rm i} \widehat{\omega})$, it follows that
$\widetilde{\eta}^{\rm Herm}(R; B, {\rm i} \widehat{\omega})$ is finite if and only if
$\widetilde{\eta}^{\rm Herm}_k(R; B, {\rm i} \widehat{\omega})$ is finite. Additionally,
if $\widetilde{\eta}^{\rm Herm}(R; B, {\rm i} \widehat{\omega})$ is finite, then
\[
	\begin{split}
	\widetilde{\eta}^{\rm Herm}(R; B, {\rm i} \widehat{\omega})
			& \;	=	\;
	\max_{t \in {\mathbb R}} \: \lambda_{\min} (H_0({\rm i} \widehat{\omega}) + t H_1 ({\rm i} \widehat{\omega}))	\\
			& \; =	\;
	\max_{t \in {\mathbb R}} \: \lambda_{\min} (H_{k,0}({\rm i} \widehat{\omega}) + t H_{k,1} ({\rm i} \widehat{\omega}))
			\; =	\;
	\widetilde{\eta}^{\rm Herm}_k(R; B, {\rm i} \widehat{\omega}),
	\end{split}
\]
completing the proof of (\ref{eq:main_interpolation_kes}).

\textbf{(ii)}
Now let us suppose that $\widetilde{\eta}^{\rm Herm}(R; B, {\rm i} \omega)$ and $\widetilde{\eta}^{\rm Herm}_k(R; B, {\rm i} \omega)$
are differentiable at $\widehat{\omega}$.
To prove the interpolation property in the derivatives, we benefit from the analytical expressions \cite{Lancaster1964}
\begin{equation}\label{eq:eta_derivative}
	\begin{split}
	& \frac{d\, \widetilde{\eta}^{\rm Herm}(R; B, {\rm i} \widehat{\omega})}{ d \omega}
			\;	=	\;
v^H	\left( \frac{d H_0({\rm i} \widehat{\omega}) }{d \omega} + \widehat{t} \: \frac{d H_1 ({\rm i} \widehat{\omega})}{d \omega} \right)	v,	\\
	& \frac{d\, \widetilde{\eta}^{\rm Herm}_k(R; B, {\rm i} \widehat{\omega})}{ d \omega}
			\;	=	\;
	v^H	\left(  \frac{d H_{k,0}' ({\rm i} \widehat{\omega}) }{d \omega} + \widehat{t} \: \frac{d H_{k,1}' ({\rm i} \widehat{\omega})}{ d \omega} \right)	v,
	\end{split}
\end{equation}
where
\[
		\widehat{t} := \argmax_t \lambda_{\min} (H_0({\rm i} \widehat{\omega}) + t H_1 ({\rm i} \widehat{\omega}))	
				= \argmax_t \lambda_{\min} (H_{k,0}({\rm i} \widehat{\omega}) + t H_{k,1} ({\rm i} \widehat{\omega}))
\]
and $v$ is a unit eigenvector corresponding to
$\lambda_{\min} (H_0({\rm i} \widehat{\omega}) + \widehat{t} H_1 ({\rm i} \widehat{\omega}))$
					$=$
	$\lambda_{\min} ($ $H_{k,0}({\rm i} \widehat{\omega}) + \widehat{t} H_{k,1} ({\rm i} \widehat{\omega}))$.
Thus, it suffices to prove that $d H_0({\rm i} \widehat{\omega}) / d \omega  = d H_{k,0}({\rm i} \widehat{\omega}) / d\omega$
and $d H_1({\rm i} \widehat{\omega}) / d\omega  =  d H_{k,1}'({\rm i} \widehat{\omega}) / d\omega$ in order to show
(\ref{eq:main_interpolation_kes2}). It follows from parts (i) and (ii) of
Theorem \ref{thm:interpolate_matrixfuns} that
\[
	\begin{split}
	d \widetilde{H}_0'({\rm i} \widehat{\omega}) /  d\omega
			& =
	-{\rm i} (B^{H} Q W({\rm i} \widehat{\omega})^{-1} B)^H (B^{H} Q W({\rm i} \widehat{\omega})^{-2} B)	\hskip 10ex	\\
&	\hskip 3ex + \; {\rm i} (B^{H} Q W({\rm i} \widehat{\omega})^{-2} B)^H (B^{H} Q W({\rm i} \widehat{\omega})^{-1} B)		\\
& =  -{\rm i} (B^{H}_k Q_k W_k({\rm i} \widehat{\omega})^{-1} B_k)^H (B^{H}_k Q_k W_k({\rm i} \widehat{\omega})^{-2} B_k)	\\
&\hskip 3ex + \; {\rm i} (B^{H}_k Q_k W_k({\rm i} \widehat{\omega})^{-2} B_k)^H (B^{H}_k Q_k W_k({\rm i} \widehat{\omega})^{-1} B_k)
			=
	d \widetilde{H}_{k,0}'({\rm i} \widehat{\omega}) / d \omega.
	\end{split}
\]
Now let us determine the derivatives of the Cholesky factors, which at a given $\omega$ satisfy
\[
	\widetilde{H}_0({\rm i} \omega) =  L(\omega) L({\rm i} \omega)^H		
	\quad
	\text{and}
	\quad
	\widetilde{H}_{k,0}({\rm i} \omega) =  L_k({\rm i} \omega) L_k({\rm i} \omega)^H .
\]		
Differentiating these two equations, and setting the derivatives equal to each other at $\widehat{\omega}$ yield
\[
	\begin{split}
	L({\rm i} \widehat{\omega}) \left( \frac{d L({\rm i} \widehat{\omega})}{d \omega} \right)^H 	+	
					\frac{d L({\rm i} \widehat{\omega})}{d \omega} L({\rm i} \widehat{\omega})^H
				&	\;	=	\;
	L_k({\rm i} \widehat{\omega}) \left( \frac{d L_k({\rm i} \widehat{\omega})}{d \omega} \right)^H 	+	
					\frac{ L_k({\rm i} \widehat{\omega})}{d \omega} L_k({\rm i} \widehat{\omega})^H		\\
				&	\;	=	\;
	L({\rm i} \widehat{\omega}) \left( \frac{d L_k({\rm i} \widehat{\omega})}{d \omega} \right)^H 	+	
					\frac{d L_k({\rm i} \widehat{\omega})}{d \omega} L({\rm i} \widehat{\omega})^H,	
	\end{split}
\]
where we have used that  $L({\rm i} \widehat{\omega}) = L_k({\rm i} \widehat{\omega})$ as established in
part (i). Thus, both $d L({\rm i} \widehat{\omega}) / d\omega$ and $d L_k({\rm i} \widehat{\omega}) / d\omega$ are
lower triangular solutions of the matrix equation
\[
	d \widetilde{H}_0({\rm i} \widehat{\omega}) / d\omega = L({\rm i} \widehat{\omega}) X^H 	+	X L({\rm i} \widehat{\omega})^H.
\]
This linear matrix equation has a unique lower triangular solution,
so $dL({\rm i} \widehat{\omega}) / d\omega = dL_k({\rm i} \widehat{\omega}) / d\omega$. Now, by the definitions of
$H_0 ({\rm i} \omega), H_{k,0} ({\rm i} \omega)$, we have
\[
	L({\rm i} \omega) H_0 ({\rm i} \omega) L({\rm i} \omega)^H		\;\;	=	\;\;		I
										\;\;	=	\;\;	L_k({\rm i} \omega) H_{k,0} ({\rm i} \omega) L_k({\rm i}\omega)^H.	
\]
Differentiating this equation at $\omega = \widehat{\omega}$ yields
\[
	\begin{split}
	 \frac{d L({\rm i} \widehat{\omega})}{d \omega} H_0({\rm i} \widehat{\omega})  L({\rm i} \widehat{\omega})^H
				\; + \;
	L({\rm i} \widehat{\omega}) \frac{d H_0({\rm i} \widehat{\omega})}{d \omega}  L({\rm i} \widehat{\omega})^H 	\hskip 30ex \\
				\; + \;\;
	L({\rm i} \widehat{\omega}) H_0({\rm i} \widehat{\omega})  \left( \frac{d L({\rm i} \widehat{\omega})}{d \omega} \right)^H
			\;\;	=	\;\;			
	 \frac{d L_k({\rm i} \widehat{\omega})}{d \omega} H_{k,0}({\rm i} \widehat{\omega})  L_k({\rm i} \widehat{\omega})^H
				\;\; + \; 	\hskip 5ex \\
	\hskip 8ex L_k({\rm i} \widehat{\omega}) \frac{d H_{k,0}({\rm i} \widehat{\omega})}{d \omega}  L_k({\rm i} \widehat{\omega})^H
				\; + \;
	L_k({\rm i} \widehat{\omega}) H_{k,0}({\rm i} \widehat{\omega})  \left( \frac{d L_k({\rm i} \widehat{\omega})}{d \omega} \right)^H.	
	\end{split}
\]
Using that $L({\rm i} \widehat{\omega}) = L_k({\rm i} \widehat{\omega})$,
$dL({\rm i} \widehat{\omega}) / d\omega =  dL_k({\rm i} \widehat{\omega}) / d\omega$, and
$H_0({\rm i} \widehat{\omega}) = H_{k,0}({\rm i} \widehat{\omega})$, we deduce that
$dH_0({\rm i} \widehat{\omega}) / d\omega = dH_{k,0}({\rm i} \widehat{\omega}) / d\omega$.
Next we focus on the derivatives
$d \widetilde{H}_1({\rm i} \omega) / d\omega, d \widetilde{H}_{k,1}({\rm i} \omega) / d\omega$.
In particular, we use that
\[
	L({\rm i} \omega) \widetilde{H}_1({\rm i} \omega)  L({\rm i} \omega)^H
			\;\;	=	\;\;
	 (B^H Q W({\rm i}\omega)^{-1} B)^{H}.
\]
Differentiating both sides of the last equation at $\omega = \widehat{\omega}$ gives rise to
\[
	\begin{split}
	\frac{d \left\{ L({\rm i} \omega) \widetilde{H}_1({\rm i}\omega)
					L({\rm i} \omega)^H \right\}}{d\omega} \bigg|_{\omega = \widehat{\omega}}
				\;\; & = \;\;	
	{\rm i} (B^H Q W({\rm i} \widehat{\omega})^{-2} B)^{H}	 \\
				 \;\; & = \;\;	
	{\rm i} (B^H_k Q_k W_k({\rm i} \widehat{\omega})^{-2} B_k)^{H}	\\
				\;\; & = \;\;	
\frac{d \left\{ L_k({\rm i} \omega) \widetilde{H}_{k,1}({\rm i} \omega)  L_k({\rm i} \omega)^H \right\}}{d\omega} \bigg|_{\omega = \widehat{\omega}}	\;\;\;	,
	\end{split}	
\]
which in turn implies that
\[
	\begin{split}
	 \frac{d L({\rm i} \widehat{\omega})}{d \omega} \widetilde{H}_1({\rm i} \widehat{\omega})  L({\rm i} \widehat{\omega})^H
				\; + \;
	L({\rm i} \widehat{\omega}) \frac{d \widetilde{H}_1({\rm i} \widehat{\omega})}{d \omega}  L({\rm i} \widehat{\omega})^H	
				\hskip 28ex \\
				\; + \;\;
	L({\rm i} \widehat{\omega}) \widetilde{H}_1({\rm i} \widehat{\omega})  \left( \frac{d L({\rm i} \widehat{\omega})}{d \omega} \right)^H
			\;\;	=	\;\;			
	 \frac{d L_k({\rm i} \widehat{\omega})}{d \omega} \widetilde{H}_{k,1}({\rm i} \widehat{\omega})  L_k({\rm i} \widehat{\omega})^H
				\;\; + \;		\hskip 5ex \\
	L_k({\rm i} \widehat{\omega}) \frac{d \widetilde{H}_{k,1}({\rm i} \widehat{\omega})}{d \omega}  L_k({\rm i} \widehat{\omega})^H
				\; + \;
	L_k({\rm i} \widehat{\omega}) \widetilde{H}_{k,1}({\rm i} \widehat{\omega}) \left( \frac{d L_k({\rm i} \widehat{\omega})}{d \omega} \right)^H .	
	\end{split}
\]
Once again exploiting that $L({\rm i} \widehat{\omega}) = L_k({\rm i} \widehat{\omega})$,
$dL({\rm i} \widehat{\omega}) / d\omega =  dL_k({\rm i} \widehat{\omega}) / d\omega$, as well as
$\widetilde{H}_1({\rm i} \widehat{\omega})	= \widetilde{H}_{k,1}({\rm i} \widehat{\omega})$
in the last equation, we obtain
	$d \widetilde{H}_1({\rm i} \widehat{\omega}) / d\omega	= d \widetilde{H}_{k,1}({\rm i} \widehat{\omega}) / d\omega$
which implies that
\begin{equation*}
	\begin{split}
	\frac{d H_1 ({\rm i} \widehat{\omega})}{d\omega}
			& 	\;\;	=	\;\;
	{\rm i}
		\left\{
			\frac{d \widetilde{H}_1({\rm i} \widehat{\omega})}{d\omega}
					-
			\left( \frac{d \widetilde{H}_1({\rm i} \widehat{\omega})}{d \omega} \right)^H		
		\right\}	\\
			&	\;\;	=	\;\;
	{\rm i}
		\left\{
			\frac{d \widetilde{H}_{k,1}({\rm i} \widehat{\omega})}{d \omega}
					-
			\left( \frac{d \widetilde{H}_{k,1}({\rm i} \widehat{\omega})}{d \omega} \right)^H
		\right\}
			\;\; = \;\;
	\frac{d H_{k,1} ({\rm i} \widehat{\omega})}{d \omega},
	\end{split}
\end{equation*}
and  the proof of (\ref{eq:main_interpolation_kes2}) is complete.
\end{proof}
\begin{remark}\label{rem:rem2}{\rm
The function $\widetilde{\eta}^{\rm Herm}(R; B, {\rm i} \omega)$ is differentiable at $\widehat{\omega}$
whenever $\widetilde{\eta}^{\rm Herm}(R;$ $B, {\rm i} \widehat{\omega})$ is finite
(equivalently $H_1({\rm i} \widehat{\omega})$ is indefinite),
$\lambda_{\min}(H_0({\rm i} \widehat{\omega}) + \widehat{t} H_1({\rm i} \widehat{\omega}))$ is simple
where $\widehat{t} := \argmax_{t \in {\mathbb R}} \lambda_{\min}(H_0({\rm i} \widehat{\omega}) + t H_1({\rm i} \widehat{\omega}))$,
and the global minimum of
$\lambda_{\min}(H_0({\rm i} \widehat{\omega}) + t H_1({\rm i} \widehat{\omega}))$ over $t$ is attained at a unique $t$.
These conditions guarantee also the differentiability of $\widetilde{\eta}^{\rm Herm}_k(R; B, {\rm i} \omega)$
at $\widehat{\omega}$ provided the subspace inclusion in (\ref{eq:subspace_inclusion}) holds.
This latter differentiability property is due to $H_0({\rm i} \widehat{\omega}) = H_{k,0}({\rm i} \widehat{\omega})$ and
$H_1({\rm i} \widehat{\omega}) = H_{k,1}({\rm i} \widehat{\omega})$
from part (i) of Theorem \ref{thm:main_Hermite_interpolate}.

Additionally, when the concave function $g(t) := \lambda_{\min}(H_0({\rm i} \widehat{\omega}) + t H_1({\rm i} \widehat{\omega}))$
attains its maximum, the maximizer is nearly always unique. The function $g(t)$ is the minimum of $m$ real analytic functions
\cite{Kato1995,Rellich1969} each corresponding to an eigenvalue of $H_0({\rm i} \widehat{\omega}) + t H_1({\rm i} \widehat{\omega})$.
If $g(t)$ does not have a unique maximizer, then at least one of these real analytic functions
must be constant and equal to $\widetilde{\eta}^{\rm Herm}(R; B, {\rm i} \widehat{\omega}) = \max_t g(t)$ everywhere.
Thus, a simple sufficient condition that ensures the uniqueness of the maximizer is that $H_1({\rm i} \widehat{\omega})$
has full rank, in which case all eigenvalues of $H_0({\rm i} \widehat{\omega}) + t H_1({\rm i} \widehat{\omega})$
blow up (either to $\infty$ or $-\infty$) as $t \rightarrow \infty$ implying that each of the real analytic functions is
non-constant.
}
\end{remark}

\subsubsection{The Subspace Framework for $r^{\rm Herm}(R; B)$}
The Hermite interpolation result of Theorem~\ref{thm:main_Hermite_interpolate} immediately suggests
the subspace framework in Algorithm \ref{alg:structured_kadii_sf} for the computation of $r^{\rm Herm}(R; B)$. 
This resembles the structure preserving subspace framework
to compute the unstructured stability radii $r(R; B, C)$ $\; = \;$ $r(J; B, C)$, in particular, in the way
the subspaces ${\mathcal V}_k, {\mathcal W}_k$ are built. At every iteration, a reduced problem
is solved using the ideas in Section~\ref{sec:small_scale_st} and employing Algorithm \ref{eigopt}.
Letting $\widehat{\omega}$ be the global minimizer of the reduced problem, the subspaces are
expanded so that the original function $\widetilde{\eta}^{\rm Herm}(R; B, {\rm i} \omega)$
is Hermite interpolated by its reduced counter-part at $\omega = \widehat{\omega}$.

Assuming that the sequence $\{ \omega_k \}$ converges to a minimizer $\omega_\ast$ of the
function $\widetilde{\eta}^{\rm Herm}(R; B, {\rm i}\omega)$ such that $H_1({\rm i} \omega_\ast)$ is indefinite
with full rank and $\lambda_{\min} (H_0({\rm i} \omega_\ast) + t_\ast H_1({\rm i} \omega_\ast))$
is simple where $t_\ast := \argmax_{t \in {\mathbb R}} \lambda_{\min} (H_0({\rm i} \omega_\ast) + t H_1({\rm i} \omega_\ast))$,
it can be shown that the sequence $\{ \omega_k \}$ converges at a
super-linear rate. Here the analysis in \cite{Aliyev2018} applies. The conditions
that $H_1({\rm i} \omega_\ast)$ is indefinite with full rank, $\lambda_{\min} (H_0({\rm i} \omega_\ast) + t_\ast H_1({\rm i} \omega_\ast))$
is simple, and the interpolation properties $H_0({\rm i} \omega_k) = H_{k,0}({\rm i} \omega_k)$,
$H_1({\rm i} \omega_k) = H_{k,1}({\rm i} \omega_k)$ ensure that the full function
$\widetilde{\eta}^{\rm Herm}(R; B, {\rm i} \omega)$, as well as the reduced function
$\widetilde{\eta}^{\rm Herm}_k(R; B, {\rm i} \omega)$ for all large $k$ are differentiable
at all $\omega$ in a ball ${\mathcal B}(\omega_\ast,\delta)$ centered at $\omega_\ast$
and of radius $\delta$. This differentiability property is essential for the applicability of
the rate-of-convergence analysis in \cite{Aliyev2018}.

\begin{algorithm}[ht]
 \begin{algorithmic}[1]
\REQUIRE{Matrices $B \in \mathbb{C}^{n\times m}$, $ J, R, Q \in \mathbb{C}^{n\times n}$.}
\ENSURE{The sequence $\{ \omega_k \}$.}
\STATE Choose the initial interpolation points $\omega_{1}, \dots , \omega_{j} \in {\mathbb R}.$
	\STATE $V_j \gets  {\rm orth} \begin{bmatrix} W(\ri\omega_1)^{-1}B  & W(\ri\omega_1)^{-2}B
									&	\dots		&	W(\ri\omega_j)^{-1}B  & W(\ri\omega_j)^{-2}B 	\end{bmatrix}$,	\\
	$\;\;\;\; W_j \gets QV_j(V_j^HQV_j)^{-1} $. \label{defn_init_subspaces_structured}
\FOR{$k = j,\,j+1,\,\dots$}
	\STATE
	$\; \displaystyle \omega_{k+1} \gets \: {\rm argmin}_{\omega \in {\mathbb R}} \; \widetilde{\eta}^{\rm Herm}_k(R; B, {\rm i}\omega)$. \label{solve_keduced_structured}
	\STATE $\widehat{V}_{k+1} \gets  \begin{bmatrix} W(\ri\omega_{k+1})^{-1}B  & W(\ri\omega_{k+1})^{-2}B \end{bmatrix}$.
	\label{defn_later_subspaces_structured}
	\STATE $V_{k+1} \gets \operatorname{orth}\left(\begin{bmatrix} V_{k} & \widehat{V}_{k+1} \end{bmatrix}\right)
		\quad \text{and}\quad W_{k+1} \gets Q V_{k+1}(V_{k+1}^H Q V_{k+1})^{-1}$.
\ENDFOR
\end{algorithmic}
\caption{$\;$ Subspace method for large-scale
computation of the structured stability radius $r^{\rm Herm}(R; B)$.}
\label{alg:structured_kadii_sf}
\end{algorithm}

\subsection{Numerical Experiments}\label{sec:str_num_exp}
In this section we present several numerical tests for  Algorithms~\ref{eigopt}
and~\ref{alg:structured_kadii_sf},
on synthetic examples and the FE model of the disk brake. The test set-up
is similar to the one for the unstructured case in Section~\ref{sec:DH_numexps}. In particular,
for Algorithm \ref{alg:structured_kadii_sf}, we use the same stopping criteria with the same
parameters, but now in terms of $f_k := \argmin_{\omega} \widetilde{\eta}^{\rm Herm}_k(R; B, {\rm i}\omega)$,
so we terminate when $| f_k - f_{k-1} | \leq \varepsilon | f_k + f_{k-1} |/2$ holds, or when one of the
other two conditions stated in Section \ref{sec:test_setup} holds. The initial subspaces are also chosen as described in that section.
\subsubsection{Synthetic Examples}
\textbf{Small Scale Examples.}
We first present numerical results for a small dense random example, where
$J, R, Q \in {\mathbb R}^{20\times 20}$ and $B \in {\mathbb R}^{20\times 2}$. These matrices are generated
by means of the MATLAB commands employed for the generation of the dense family in Section \ref{sec:numexp_syn};
only here the matrices are $20\times 20$ instead of $800 \times 800$ and $R$ is of rank $5$.

The spectrum of $(J-R)Q$ is depicted on the top in Figure \ref{fig:structured20by20_spec}. Application of Algorithm \ref{eigopt} to this example yields $r^{\rm Herm}(R;B) = 0.0501$, and the point
that is first reached on the imaginary axis under the smallest perturbation  is $1.9794{\rm i}$,  i.e., $\widetilde{\eta}^{\rm Herm}_k(R; B, {\rm i}\omega)$ is minimized at $\omega = 1.9794$.
At the bottom of Figure \ref{fig:structured20by20_spec}, the spectra of matrices
of the form $(J - (R + B \Delta B^H))Q$ are plotted for $100000$ randomly chosen
Hermitian $\Delta$ such that $\| \Delta \|_2 = 0.0501$. One can notice that some of the eigenvalues (nearly) touch the imaginary axis at $1.9794{\rm i}$ marked with a circle.
\begin{figure}
\begin{center}
		\begin{tabular}{c}
			\includegraphics[width=9cm]{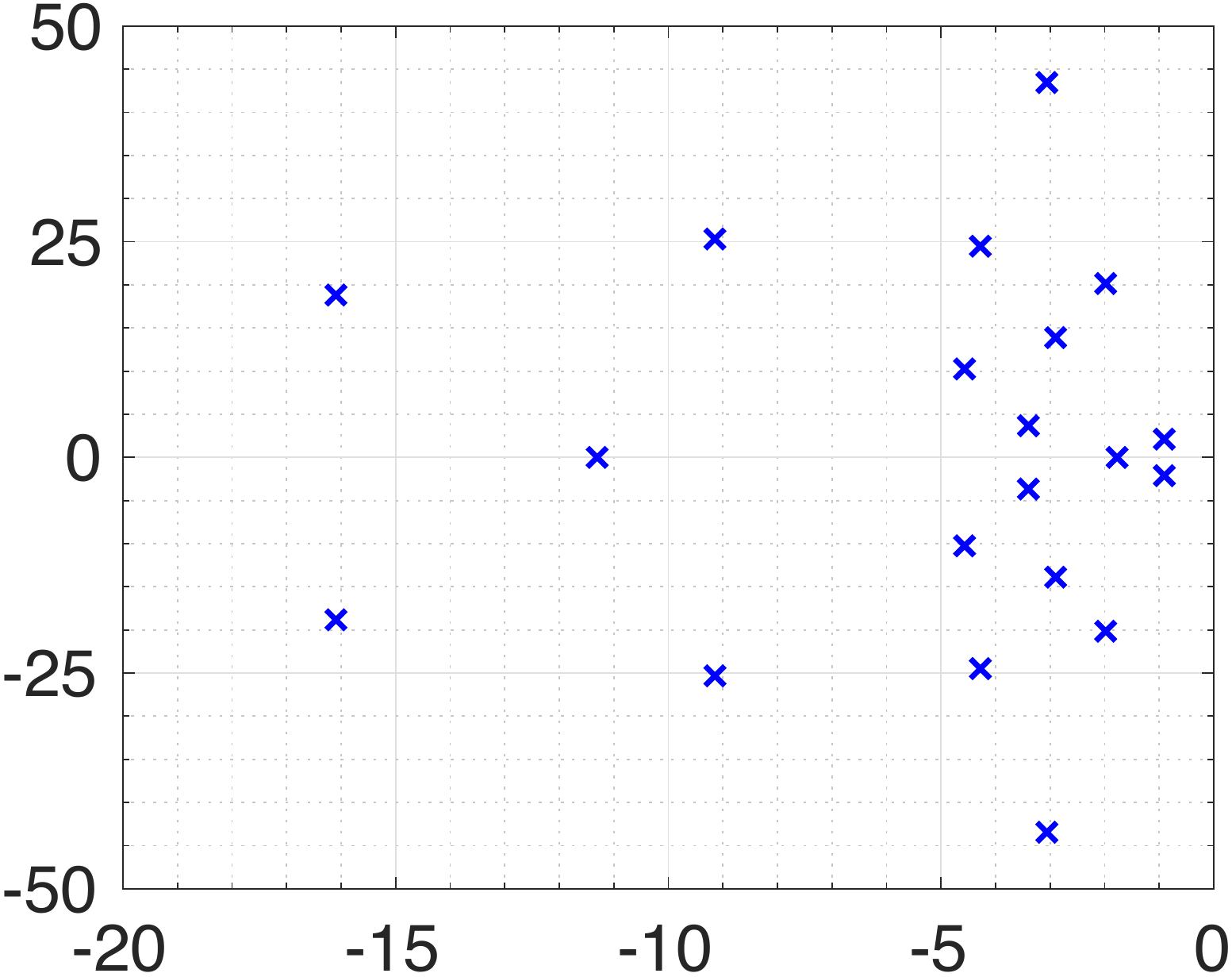} 	\\
			 \includegraphics[width=9cm]{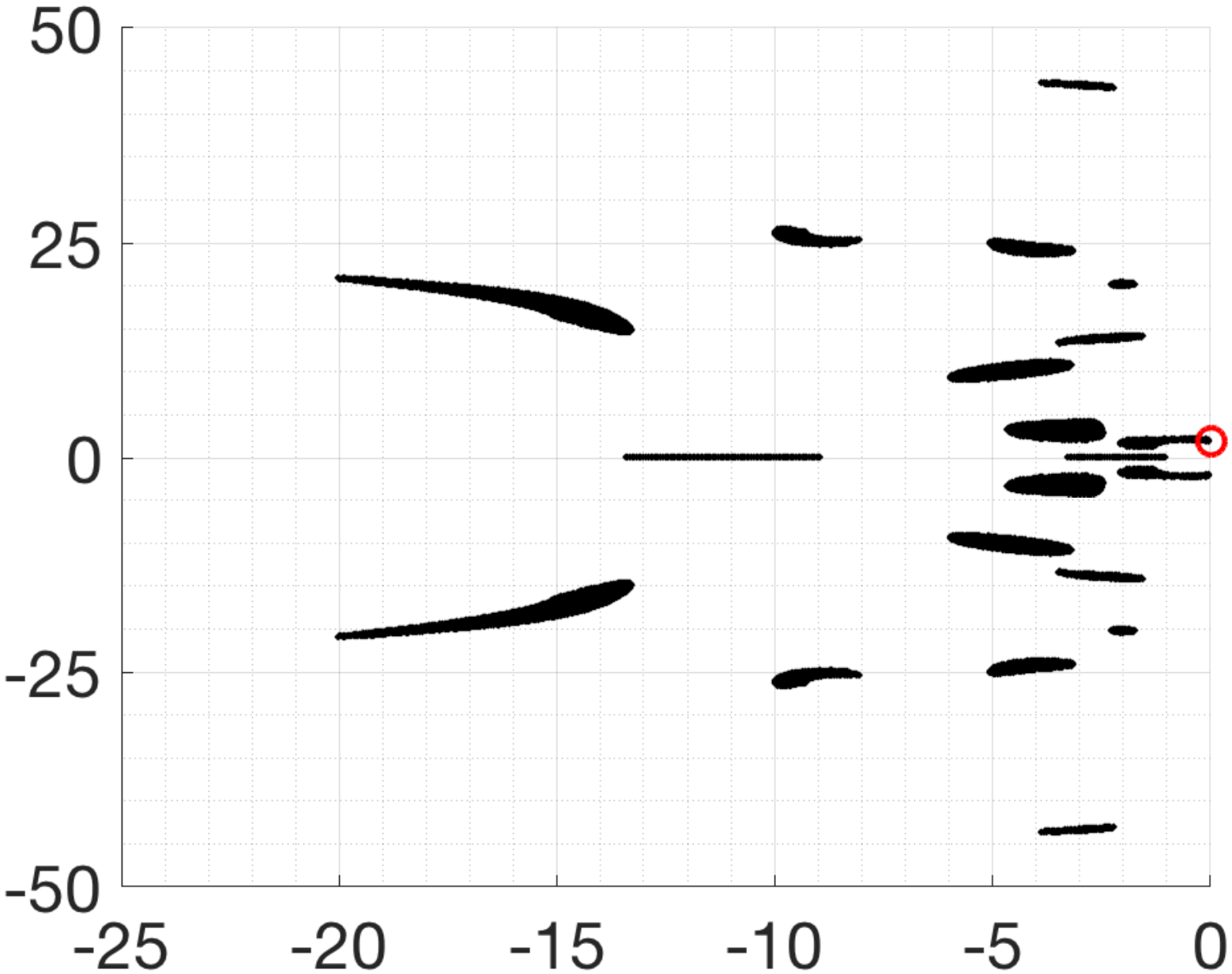}	 \\
		\end{tabular}
\end{center}
		\caption{A DH system of order $20$ with random system matrices.
		\textbf{(Top)} A plot of the spectrum for $(J-R)Q$.
		\textbf{(Bottom)} All eigenvalues of all matrices of the form $(J - (R + B \Delta B^H ))Q$
		for $100000$ randomly chosen Hermitian $\Delta$ with $\| \Delta \|_2 = r^{\rm Herm}(R; B)$
		are displayed. The circle marks $1.9794{\rm i}$, the global minimizer of
		$\widetilde{\eta}^{\rm Herm}_k(R; B, \lambda)$ over $\lambda \in {\rm i} {\mathbb R}$.}
		\label{fig:structured20by20_spec}
\end{figure}

The subspace framework for this example is illustrated in Figure \ref{fig:structured20by20_sf},
where the solid, dashed curves correspond to the plots of the full function
$\widetilde{\eta}^{\rm Herm}(R; B, {\rm i}\omega)$, the reduced function
$\widetilde{\eta}^{\rm Herm}_k(R; B, {\rm i}\omega)$, respectively, and the
circle represents the minimizer of $\widetilde{\eta}^{\rm Herm}_k(R; B, {\rm i}\omega)$.
On the top row, the framework is initiated with two interpolation points at 0 and near -20;
the dashed curve interpolates the solid curve at these points. Then the subspaces
are expanded so that the Hermite interpolation property is also satisfied at the minimizer
of the dashed curve on the top, leading to the dashed curve at the bottom, which has
nearly the same global minimizer as the solid curve. Note that starting from $\omega$
near -14 and for smaller $\omega$ values, the matrix $H_1({\rm i} \omega)$ turns out to be
definite for this example, meaning for such $\omega$ values the point ${\rm i} \omega$
is not attainable as an eigenvalue with Hermitian perturbations. In practice, we set the objective
to be minimized at such $\omega$ a value considerably larger than the minimal
value of the objective, which in this example is 0.1.

\medskip

\noindent
\textbf{Large Examples.}
The remaining synthetic examples are with larger random matrices. We created three sets of matrices $J, R, Q, B$ using the commands at the beginning of Section \ref{sec:numexp_syn} for the
generation of the dense family. Each of the three sets consists of four quadruples $J, R, Q \in {\mathbb R}^{n\times n}$,
$B \in {\mathbb R}^{n\times 2}$ with the same $n$, specifically
$n = 1000, 2000, 4000$ for the first, second, third set. The results obtained by applications of
Algorithm \ref{alg:structured_kadii_sf} to compute $r^{\rm Herm}(R; B)$
are reported in Tables \ref{table:structured_1000exs}, \ref{table:structured_2000exs}, \ref{table:structured_4000exs}
for these sets, respectively.

For the first family with  $n = 1000$, the subspace framework, i.e., Algorithm \ref{alg:structured_kadii_sf}
already needs less computing time than Algorithm \ref{eigopt} except for the last example where quite a few additional
subspace iterations have been performed. On this family, the direct application of Algorithm \ref{eigopt}
and the subspace framework return exactly the same values for $r^{\rm Herm}(R; B)$.

For bigger systems Algorithm \ref{eigopt} becomes too computationally expensive,  so we do not report results here for the larger dimensions. A remarkable fact we have observed is that the number of subspace iterations to reach the
prescribed accuracy is usually small and seems independent of $n$. By the definitions of the structured
and unstructured radii, we must have $r^{\rm Herm}(R; B) \geq r(R; B, B^H)$, and the presented radii
in the tables are in harmony with this.

All of these examples involve optimization of highly non-convex and non-smooth functions.
Figure \ref{fig:structured1000by10000_sf} depicts $\widetilde{\eta}^{\rm Herm}(R; B, {\rm i}\omega)$
as a function of $\omega$ with the solid curve for the first example in the first family with $n = 1000$.
The same figure also depicts the reduced function $\widetilde{\eta}^{\rm Herm}_k(R; B, {\rm i}\omega)$
with the dashed curve for the same example at termination after 7 subspace iterations. Even though
the reduced function $\widetilde{\eta}^{\rm Herm}_k(R; B, {\rm i}\omega)$ represented by the
dashed curve in this plot is defined for projected matrices onto $48$ dimensional subspaces, it captures
the original function $\widetilde{\eta}^{\rm Herm}(R; B, {\rm i}\omega)$ remarkably well near the
global minimizer $\omega_\ast = -70.9623$.

\begin{figure}
\begin{center}
		\begin{tabular}{cc}
			\includegraphics[width=8cm]{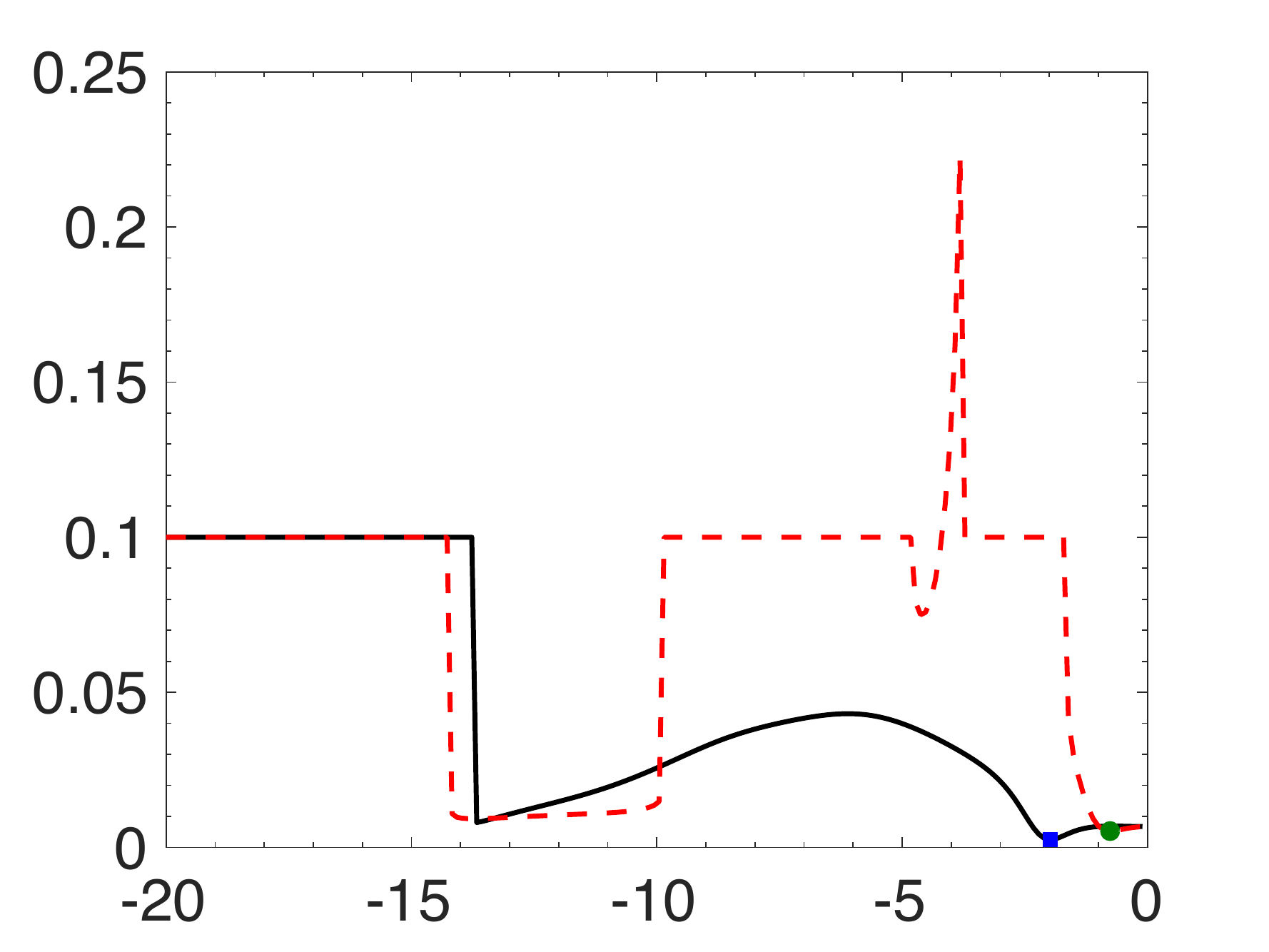} 	\\
			 \includegraphics[width=8cm]{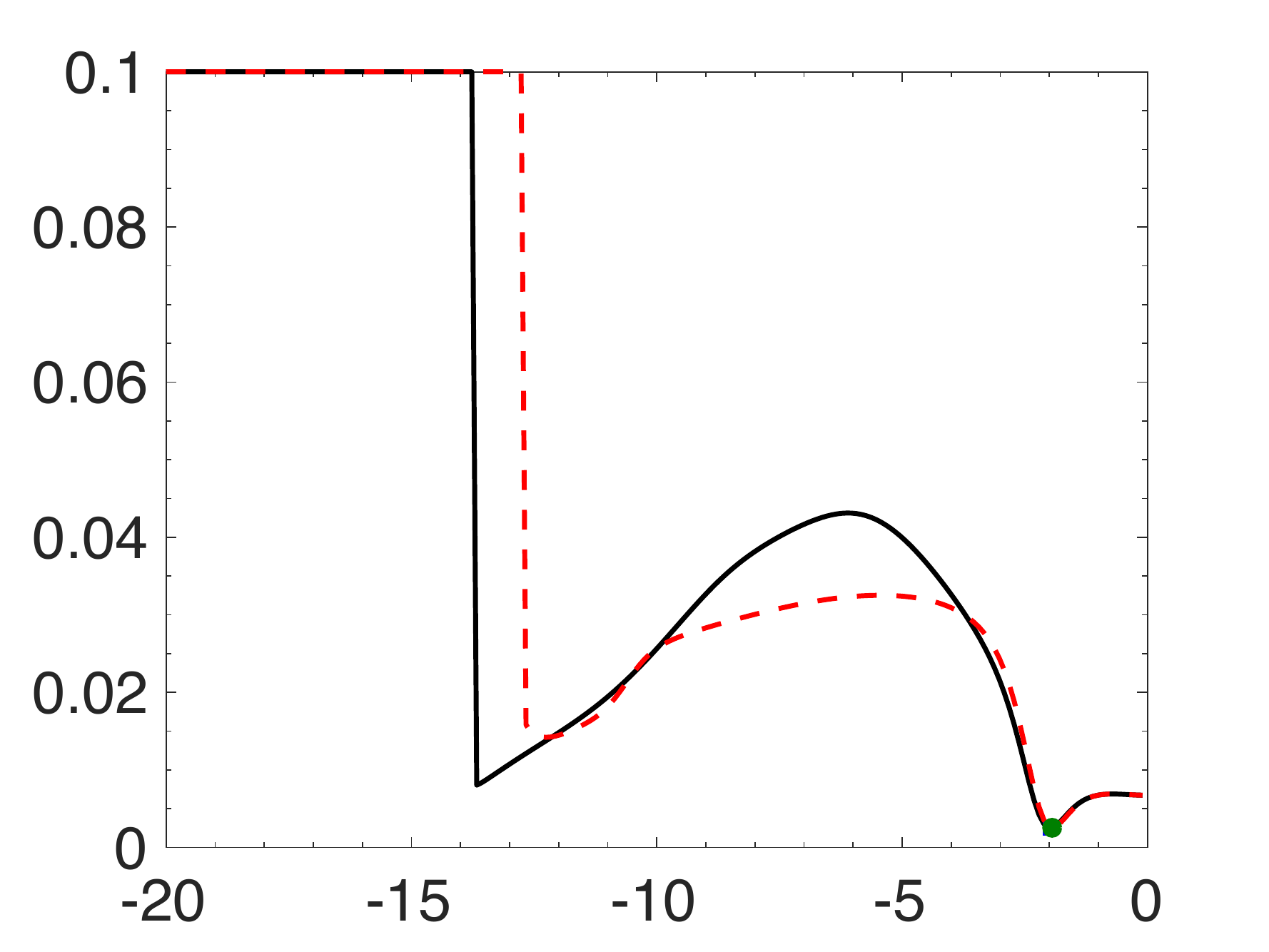}	 \\
		\end{tabular}
\end{center}
		\caption{ Progress of Algorithm \ref{alg:structured_kadii_sf} on a random DH system of order $20$. The solid, dashed
		curves display $\widetilde{\eta}^{\rm Herm}(R; B, {\rm i}\omega)$, $\widetilde{\eta}^{\rm Herm}_k(R; B, {\rm i}\omega)$ as
		functions of $\omega$, respectively. The square, circle represent the global minima of
		$\widetilde{\eta}^{\rm Herm}(R; B, {\rm i}\omega)$, $\widetilde{\eta}^{\rm Herm}_k(R; B, {\rm i}\omega)$.
		\textbf{(Top Plot)} The reduced function interpolates the full function at two points,
		at $\omega = 0$ and near $\omega=-20$.
		\textbf{(Bottom Plot)} Plot of the refined reduced function after applying one subspace iteration. }
		\label{fig:structured20by20_sf}
\end{figure}

\begin{figure}
\begin{center}
		\begin{tabular}{c}
			 \includegraphics[width=8cm]{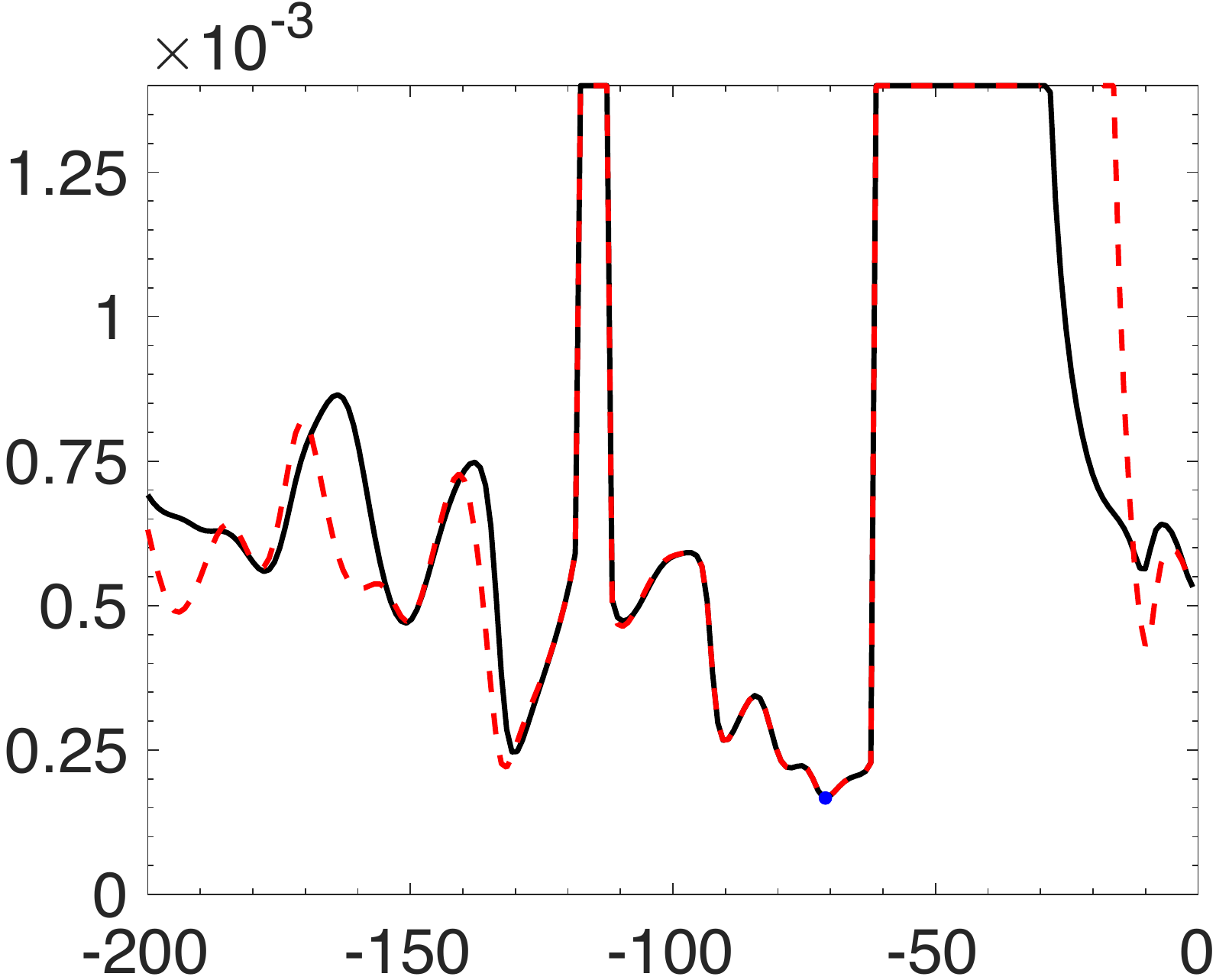}	
		\end{tabular}
\end{center}
		\caption{Application of Algorithm~\ref{alg:structured_kadii_sf} to compute $r^{\rm Herm}(R;B)$
		on a dense random DH system of order $1000$. The full and reduced functions $\widetilde{\eta}^{\rm Herm}(R; B, {\rm i}\omega)$
		and $\widetilde{\eta}^{\rm Herm}_k(R; B, {\rm i}\omega)$ at termination after $7$
		subspace iterations are plotted with the solid and dashed curves, respectively.
		The point $(\omega_\ast, \widetilde{\eta}^{\rm Herm}_k(R; B, {\rm i}\omega_\ast))$ with $\omega_\ast$
		denoting the global minimizer of $\widetilde{\eta}^{\rm Herm}_k(R; B, {\rm i}\omega)$ is marked with a circle. }
		\label{fig:structured1000by10000_sf}
\end{figure}

\begin{table}
\begin{center}
\begin{tabular}{|c||cc|c|c|cc|}
\hline
  &  \multicolumn{2}{c|}{$r^{\rm Herm}(R; B)$}  &
  						\multicolumn{1}{c|}{$r(R; B, B^H)$}		&		
  									\multicolumn{1}{c}{iterations}		&
			\multicolumn{2}{|c|}{run-time} \\ [0.5ex]
  $\#$ &  Alg.~\ref{eigopt} & Alg.~\ref{alg:structured_kadii_sf}
  			& Alg.~\ref{alg:dti}
  			&  Alg.~\ref{alg:structured_kadii_sf}
					&	Alg.~\ref{eigopt} & Alg.~\ref{alg:structured_kadii_sf}	\\ [0.5ex]  \hline	\hline

1     &   0.0129    &     0.0129       & 	0.0113 & 7	&	99.2	&	66.2	\\

2      & 0.0101    &      0.0101       &	0.0101 & 1	&	66.5	&  7.1	\\

3      &  0.0141     &     0.0141  	&	0.0128 & 5	&	128.5 &	72.5	\\

4   	&  0.0096     &     0.0096      &	0.0072 & 20	&	121.2 &	247.2	\\

\hline

\end{tabular}
\end{center}
\caption{ Performance of Algorithm \ref{alg:structured_kadii_sf} to compute $r^{\rm Herm}(R;B)$
		on dense random DH systems of order $1000$. For comparison, the results from Algorithm \ref{eigopt},
		as well as the unstructured radii $r(R;B, B^H)$ by Algorithm \ref{alg:dti} are also included.
		 The fourth column contains the number of subspace iterations, and the fifth the run-times (in $s$). }
\label{table:structured_1000exs}
\end{table}

\begin{table}
\begin{center}
\begin{tabular}{|c||c|c|c|c|}
\hline
  &  \multicolumn{1}{c|}{$r^{\rm Herm}(R; B)$}  &
  						\multicolumn{1}{c|}{$r(R; B, B^H)$}		&		
  									\multicolumn{1}{c}{iterations}		&
			\multicolumn{1}{|c|}{run-time} \\ [0.5ex]
  $\#$ & Alg.~\ref{alg:structured_kadii_sf}
  			& Alg.~\ref{alg:dti}
  			&  Alg.~\ref{alg:structured_kadii_sf}
					& Alg.~\ref{alg:structured_kadii_sf}	\\ [0.5ex]  \hline	\hline

1      &     0.0102      & 	0.0061 & 6	&	149.1	\\

2       &      0.0040       &	0.0012 & 1	&  	17.1	\\

3      &     0.0125  	&	0.0109 & 6	&	175.9	\\

4   	&     0.0101      &	0.0100 & 3	&	126.7	\\

\hline

\end{tabular}
\end{center}
\caption{Performance of Algorithm \ref{alg:structured_kadii_sf} to compute $r^{\rm Herm}(R;B)$ 		on dense random DH systems
of order $n = 2000$. The fourth column contains the number of subspace iterations, and the fifth the run-times (in $s$).
}
\label{table:structured_2000exs}
\end{table}

\begin{table}
\begin{center}
\begin{tabular}{|c||c|c|c|c|}
\hline
  &  \multicolumn{1}{c|}{$r^{\rm Herm}(R; B)$}  &
  						\multicolumn{1}{|c|}{$r(R; B, B^H)$}		&		
  									\multicolumn{1}{|c|}{iterations}		&
			\multicolumn{1}{|c|}{run-time} \\ [0.5ex]
  $\#$ & Alg.~\ref{alg:structured_kadii_sf}
  			& Alg.~\ref{alg:dti}
  			&  Alg.~\ref{alg:structured_kadii_sf}
					& Alg.~\ref{alg:structured_kadii_sf}	\\ [0.5ex]  \hline	\hline

1     &     0.0090      & 	0.0087 	& 	2	&	282.6	\\

2      &      0.0084       &	0.0068 	& 	3	&  	319.3	\\

3      &     0.0104  	&	0.0086 	& 	2	&	333.1	\\

4   	&     0.0040     &	0.0007 	& 	46	&	786.5	\\

\hline

\end{tabular}
\end{center}
\caption{Performance of Algorithm \ref{alg:structured_kadii_sf} to compute $r^{\rm Herm}(R;B)$ 		on dense random DH systems
of order $n = 4000$. The fourth column contains the number of subspace iterations, and the fifth the run-times (in $s$). }
\label{table:structured_4000exs}
\end{table}

\subsubsection{FE Model of a Disk Brake}
We also applied our implementation of
Algorithm~\ref{alg:structured_kadii_sf} to the FE model of the disk brake described in Section \ref{sec:numexp_brake} which is of the form (\ref{insta}), (\ref{eq:insta_matrices}) with
$G(\Omega), D(\Omega), K(\Omega), M \in {\mathbb R}^{4669\times 4669}$, and
$J, R, Q \in {\mathbb R}^{9338\times 9338}$.

The computed value of the structured radius $r^{\rm Herm}(R; B)$ along with the computed global minimizer
$\omega_\ast$ of $\widetilde{\eta}^{\rm Herm}(R; B, {\rm i}\omega)$, as well as the number of subspace
iterations, run-time (in $s$) and the subspace dimension at termination are listed in
Table \ref{tab:brake_squeal_structured} for various values of $\Omega$.  It is worth comparing the computed values of
$r^{\rm Herm}(R; B)$ in this table with those for the unstructured
stability radius $r(R; B, B^T)$ listed in Table \ref{tab:brake_squeal_unstructured}. The computed structured
and the unstructured stability radii are close, though the structured stability radii are slightly larger as expected in theory.

\begin{table}
	\hskip -0.2ex
	\begin{tabular}{|c||ccccc|}
	\hline
			$\Omega$		&	$r^{\rm Herm}(R; B)$		&	$\omega_\ast$			&	iterations		&	run-time 	&	dimension	 \\
	\hline
	\hline
			 2.5			&	0.01067		&	$-1.938\times 10^5$		&	2			&	145.8		&	72	\\		
			 5			&	0.01038		&	$-1.938\times 10^5$		&	2			&	133.5		&	72	\\
			 10			&	0.01026		&	$-1.938\times 10^5$		&	2			&	132.0		&	72	\\
			 50			&	0.01000		&	$-1.938\times 10^5$		&	2			&	132.5		&	72	\\
			 100			&	0.00988		&	$-1.938\times 10^5$		&	1			&	94.4			&	66	\\	
			 1000		&	0.00810		&	$-1.789\times 10^5$		&	2			&	127.7		&	72	\\
			 1050		&	0.00794		&	$-1.789\times 10^5$		&	2			&	126.4		&	72	\\
			 1100		&	0.00835		&	$-1.789\times 10^5$		&	3			&	171.2		&	78	\\
			 1116		&	0.01092		&	$-1.789\times 10^5$		&	2			&	127.4		&	72	\\
			 1150		&	0.00346		&	$-1.742\times 10^5$		&	2			&	124.3		&	72	\\
			 1200		&	0.00408		&	$-1.742\times 10^5$		&	2			&	121.2		&	72	\\
			 1250		&	0.00472		&	$-1.742\times 10^5$		&	2			&	119.1		&	72	\\
			 1300		&	0.00517		&	$-1.742\times 10^5$		&	2			&	117.1		&	72	\\
	\hline
	\end{tabular}
\caption{ Structured stability radii $r^{\rm Herm}(R; B)$ computed by Algorithm \ref{alg:structured_kadii_sf} for the FE model of a disk brake of order $9338$  for several values of $\Omega$.
The column $\omega_\ast$ depicts the computed global minimizer of $\widetilde{\eta}^{\rm Herm}(R; B, {\rm i}\omega)$,
whereas the last three columns depict the number of subspace iterations, the total run-time (in $s$), and the subspace dimension at termination.
}
\label{tab:brake_squeal_structured}
\end{table}

\section{Concluding Remarks}
We have proposed subspace frameworks to compute the stability radii for large scale dissipative Hamiltonian systems. The frameworks operate on the eigenvalue optimization characterizations of the stability radii derived
in \cite{MehMS16a}. At every iteration, we apply DH structure preserving Petrov-Galerkin projections to small subspaces. This leads to the computation of the corresponding stability radii for the reduced system. We expand the subspaces used in the
Petrov-Galerkin projections so that Hermite-interpolation properties between the objective eigenvalue function
of the full and the reduced problems are attained at the optimizer of the reduced problem. This strategy results in super-linear convergence with respect to the subspace dimensions. We have illustrated that the frameworks work well
in practice on several synthetic examples, and a FE model of a disk brake.

Matlab implementations of the proposed algorithms and subspace frameworks are made publicly available
on the web\footnote{\url{http://home.ku.edu.tr/~emengi/software/DH-stabradii.html}}. Some of the data
(including the one associated with the disk brake example) used in the numerical experiments
are also available on the same website.

One difficulty is that the proposed frameworks converge only locally. As a remedy for this, we have
initiated the subspaces to attain Hermite interpolation at several points on the imaginary axis between the full and initial
reduced problems. One potential strategy that is currently investigated is to employ equally spaced interpolation points. Another potential strategy finds the poles closest
to these equally spaced points, then employs the imaginary parts of the poles as the initial interpolation points.

Another research direction that is currently investigated is the maximization of the stability radii, when $J, R, Q$ depend on
parameters in a given parameter set. As an example, for the dissipative Hamiltonian system arising from the FE model of a disk
brake, even in the simple setting considered here, $J, R, Q$ depend on the rotation speed $\Omega$.

\bibliography{PH_stabradii}

\end{document}